\documentclass[12pt, a4paper]{article}
\usepackage[utf8]{inputenc}
\RequirePackage[OT1]{fontenc}
\usepackage{amsmath}
\usepackage{amssymb}
\usepackage{graphicx}
\usepackage{xcolor}
\usepackage{mathtools}
\usepackage{placeins}
\RequirePackage[numbers]{natbib}
\usepackage[colorlinks,citecolor=blue,urlcolor=blue]{hyperref}

\newcommand{\Norm}[1]{\left\|#1\right\|}
\newcommand{\NormInfinity}[1]{\left\|#1\right\|_{\infty}}

\newcommand{\BlackBox}{\rule{1.5ex}{1.5ex}}  
\newcommand{\qed}{\hfill\BlackBox\\[2mm]}

\newcommand{\risk}{\mathcal{L}}

\newcommand{\cT}{\mathcal{T}}
\newcommand{\cX}{\mathcal{X}}

\newcommand{\bayes}{s}
\newcommand{\loss}[1]{\ell(s,#1)}

\newcommand{\xloss}[1]{\ell_X(#1)}

\newcommand{\learnrule}{\mathcal{A}}

\newcommand{\ERMho}[1]{\widehat{f}^{\,\mathrm{ho}}_{#1}}
\newcommand{\ERMag}[1]{\widehat{f}^{\,\mathrm{ag}}_{#1}}
\newcommand{\ERMcv}[1]{\widehat{f}^{\,\mathrm{cv}}_{#1}}
\newcommand{\ERMacv}[1]{\widehat{f}^{\,\mathrm{acv}}_{#1}}

\newcommand{\numberthis}{\addtocounter{equation}{1}\tag{\theequation}}

\newcommand\independent{\protect\mathpalette{\protect\independenT}{\perp}}
    \def\independenT#1#2{\mathrel{\rlap{$#1#2$}\mkern2mu{#1#2}}}
\newenvironment{proof}{\par\noindent{\bf Proof\ }}{\qed}
\DeclareMathOperator*{\argmin}{argmin}

\newtheorem{lemma}{Lemma}[section]
\newtheorem{Theorem}[lemma]{Theorem}
\newtheorem{corollary}[lemma]{Corollary}
\newtheorem{proposition}[lemma]{Proposition}
\newtheorem{Definition}[lemma]{Definition}

\newtheorem{claim}{Claim}[lemma]
\newtheorem{hyp}{Hypothesis}[section]

\title{Aggregated hold out for sparse linear regression with a robust loss function}

\author{Guillaume Maillard}

\begin{document}

\maketitle




%


\begin{abstract}
 Sparse linear regression methods generally have a free hyperparameter which controls the amount of sparsity, and is subject to a bias-variance tradeoff. This article considers the use of Aggregated hold-out to aggregate over values of this hyperparameter, in the context of linear regression with the Huber loss function. Aggregated hold-out (Agghoo) is a procedure which averages estimators selected by hold-out (cross-validation with a single split). In the theoretical part of the article, it is proved that Agghoo satisfies a non-asymptotic oracle inequality when it is applied to sparse estimators which are parametrized by their zero-norm. In particular, this includes a variant of the Lasso introduced by Zou, Hastié and Tibshirani \cite{Zou_Has_Tib:2007}. Simulations are used to compare Agghoo with cross-validation. They show that Agghoo performs better than CV when the intrinsic dimension is high and when there are confounders correlated with the predictive covariates.    
\end{abstract}



\section{Introduction}

From the statistical learning point of view, linear regression is a risk-minimization problem wherein the aim is to minimize the average \emph{prediction error} $\phi(Y - \theta^T X)$ on a new, independent data-point $(X,Y)$, as measured by a \emph{loss function} $\phi$. When $\phi(x) = x^2$, this yields classical least-squares regression; however, Lipschitz-continuous loss functions have better robustness properties and are therefore preferred in the presence of heavy-tailed noise, since they require fewer moment assumptions on $Y$ \cite{Chinot2020,Huber2009}. Similarly to the $L^2$ norm in the least-squares case, measures of performance for estimators can be derived from robust loss functions  by substracting the risk of the (distribution-dependent) optimal predictor, yielding the so-called \emph{excess risk}. 

In the high-dimensional setting, where $X \in \mathbb{R}^d$ with potentially $d > n$, full linear regression cannot be achieved in general: the minimax excess risk is bounded below by a positive function of $\frac{d}{n}$ (proposition \ref{prop_minimax_huber_reg}). 
Stronger assumptions on the regression coefficient $\theta$ are needed in order to estimate it consistently. 

A popular approach is to suppose that only a small number $k_*$ of covariates are relevant to the prediction of $Y$, so that $\theta$ may be sought among the \emph{sparse} vectors with less than $k_*$ non-zero components. Estimators which target such problems include the Lasso \cite{Tib:1996}, least-angle regression \cite{efron2004} (a similar, but not identical method \cite[Section 3.4.4]{Hastie2009}), and stepwise regression \cite[Section 3.3.2]{Hastie2009}. In the robust setting, variants of the Lasso with robust loss functions have been investigated by a number of authors \cite{lambert-lacroix2011, Rosset-Zhu2007,Chen-Wang-McKeown2010, Wang-Li-Jiang2007}. 

Such methods generally introduce a free hyperparameter which regulates the "sparsity" of the estimator; sometimes this is directly the number of non-zero components, as in stepwise procedures, sometimes not, as in the case of the Lasso, which uses a regularization parameter $\lambda$. 
In any case, the user is left with the problem of calibrating this hyperparameter.

Several goals are conceivable for a hyperparameter selection method, such as support recovery - finding the "predictive" covariates - or estimation of a "true" underlying regression coefficient with respect to some norm on $\mathbb{R}^{d}$. From a prediction perspective, hyperparameters should be chosen so as to minimize the risk, and a good method should approach this minimum. As a consequence, the proposed data-driven choice of hyperparameter should allow the estimator to attain all known convergence rates without any a priori knowledge, effectively adapting to the difficulty of the problem.  

For the Lasso and some variants, such as the fused Lasso, Zou, Wang, Tibshirani and coauthors have proposed \cite{Zou_Has_Tib:2007} and investigated \cite{Wang:2007, Tib:2012} a method based on Mallow's $C_p$ and estimation of the "degrees of freedom of the Lasso". However, consistency of this method has only been proven \cite{Wang:2007}  in an asymptotic setting where the dimension is fixed while $n$ grows, hence not the setting considered here. Moreover, the method depends on specific properties of the Lasso, and may not be readily applicable to other sparse regression procedures.

A much more widely applicable procedure is to choose the hyperparameter by cross-validation. For the Lasso, this approach has been recommended by Tibshirani \cite{Tibshirani1996}, van de Geer and Lederer \cite{vdG2011} and Greenshtein \cite{greenshtein2004}, among many others. More generally, cross-validation is the default method for calibrating hyperparameters in practice. For exemple, R implementations of the elastic net \cite{glmnet:2010} (package \href{https://cran.r-project.org/web/packages/glmnet/index.html}{glmnet}), LARS \cite{efron2004} (package \href{https://cran.r-project.org/web/packages/lars/index.html}{lars}) and the huberized lasso \cite{hqreg:2017} (package \href{https://cran.r-project.org/web/packages/hqreg/index.html}{hqreg}) all incorporate a cross-validation subroutine to automatically choose the hyperparameter.

Theoretically, cross-validation has been shown to perform well in a variety of settings \cite{Arl_Cel:2010:surveyCV}. For cross-validation with one split, also known as the hold-out, and for a bagged variant of v-fold cross-validation \cite{lecue2012}, some general oracle inequalities  are available in least squares regression \cite[Corollary 8.8]{Mas:2003:St-Flour} \cite{wegkamp2003} \cite{lecue2012}. However, they rely on uniform boundedness assumptions on the estimators which may not hold in high-dimensional linear regression. For the more popular V-fold procedure, results are only available in specific settings. Of particular interest here is the article \cite{Nav_Saum:2017} which proves oracle inequalities for linear model selection in least squares regression, since linear model selection is very similar to sparse regression (the main difference being that in sparse regression, the "models" are not fixed a priori but depend on the data). This suggests that similar results could hold for sparse regression.

However, in the case of the Lasso at least, no such general theoretical guarantees exist, to the best of our knowledge. Some oracle inequalities \cite{lecue2012, Mio_Mont:2018} and also fast rates \cite[Theorem 1]{lasso_cv} have been obtained, but only under strong distributional assumptions: \cite{lecue2012} assumes that $X$ has a log-concave distribution, \cite{Mio_Mont:2018} that $X$ is a gaussian vector, and \cite[Theorem 1]{lasso_cv} assumes that there is a true model and that the variance-covariance matrix is diagonal dominant. 
Recently, Chetverikov et al. \cite{Chet_Cher:2021} have obtained fast rates (up to log-terms) for a certain class of conditional distributions (of $Y$ given $X$) which are smooth transformations of Gaussian distributions.
In contrast, there are also theorems \cite{Chat_Jaf:2015} \cite[Theorem 2]{lasso_cv} which make much weaker distributional assumptions but only prove convergence of the (in-sample) error at the "slow" rate $\mathcal{O}(r \sqrt{\frac{\log d}{n}})$ or slower. Though this rate is basically minimax \cite{rigollet2011} for the model 
\begin{equation}
Y = \langle X, \theta_* \rangle + \varepsilon, \mathbb{E}[\varepsilon | X] = 0, \mathbb{E}[\varepsilon^2 | X] \leq 1, X \in \mathbb{R}^d, \Norm{\theta_*}_{\ell^1} \leq r,
\end{equation}
a hyperparameter selection method should adapt also to the favorable cases where the Lasso converges faster (\cite[Theorem 14]{Koltch:2011}); these results do not show that CV has this property.

Thus, the theoretical justification for the use of standard CV, which selects a single hyperparameter by minimizing the CV risk estimator, is somewhat lacking. In fact, two of the articles mentioned above introduce variants of CV which modify the final hyperparameter selection step; a bagged CV in \cite{lecue2012} and the aggregation of two hold-out predictors in \cite{Chat_Jaf:2015}.
In practice too, there is reason to consider alternatives to hyperparameter selection in sparse regression: sparse estimators are unstable, and selecting only one estimator can result in arbitrarily ignoring certain variables among a correlated group with similar predictive power \cite{sparse_unstable}. For the Lasso, these difficulties have motivated researchers to introduce several aggregation schemes, such as the Bolasso \cite{bolasso}, stability selection \cite{stab_select}, the lasso-zero \cite{lasso0} and the random lasso \cite{wang2011}, which are shown to have some better properties than the standard Lasso.

Since aggregating the Lasso seems to be advantageous, it seems logical to consider aggregation rather than selection to handle the free hyperparameters.
In this article, we consider the application to sparse regression of the aggregated hold-out procedure. Aggregated hold-out (agghoo) is a general aggregation method which mixes cross-validation with bagging. It is an alternative to cross-validation, with a comparable level of generality. In a previous article with Sylvain Arlot and Matthieu Lerasle \cite{agghoo_rkhs}, we formally defined and studied Agghoo, and showed empirically that it can improve on cross-validation when calibrating the level of regularization for kernel regression. Though we came up with the name and the general mathematical definition, Agghoo has already appeared in the applied litterature in combination with sparse regression procedures \cite{HoyosIdrobo2015}, among others \cite{Varoquaux2017}, under the name "CV + averaging" in this case. 

In the present article, the aim is to study the application of Agghoo to sparse regression with a robust loss function. Theoretically, assuming an $L^\infty - L^2$ norm inequality to hold on the set of sparse linear predictors, it is proven that Agghoo satisfies an asymptotically optimal oracle inequality. This result applies also to cross-validation with one split (the so-called hold-out), yielding a new oracle inequality which allows norms of the sparse linear predictors to grow polynomially with the sample size. Empirically, Agghoo is compared to cross-validation in a number of simulations, which investigate the impact of correlations in the design matrix and sparsity of the ground truth on the performance of aggregated hold-out and cross-validation. Agghoo appears to perform better than cross-validation when the number of non-zero coefficients to be estimated is not much smaller than the sample size. The presence of confounders correlated to the predictive variables also favours Agghoo relative to cross-validation.

\section{Setting and Definitions}
The problem of non-parametric regression is to infer a predictor $t: \cX \rightarrow \mathbb{R}$ from a dataset $(X_i,Y_i)_{1 \leq i \leq n}$ of pairs, where $X_i \in \cX$ and $Y_i \in \mathbb{R}$. The pairs will be assumed to be i.i.d, with joint distribution $P$. The prediction error made at a point $(x,y) \in \cX \times \mathbb{R}$ is measured using a non-negative function of the residual $\phi(y-t(x))$. The global performance of a predictor is assessed on a new, independent data point $(X,Y)$ drawn from the same distribution $P$ using the risk $\risk(t) = E[\phi(Y-t(X))]$. The optimal predictors $\bayes$ are characterized by $\bayes(x) \in \argmin_u E[\phi(Y-u) | X = x]$ a.s. The risk of any optimal predictor is (in general) a non-zero quantity which characterizes the intrinsic amount of ``noise'' in $Y$ unaccounted for by the knowledge of $X$. A predictor $t$ can be compared with this benchmark by using the \emph{excess risk} $\loss{t} = \risk(t) - \risk(\bayes)$. Taking $\phi(x) = x^2$ yields the usual least-squares regression, where $s(x) = E[Y | X = x]$ and $\loss{t} = \Norm{(s - t)(X)}_{L^2}^2$. However, the least-squares approach is known to suffer from a lack of robustness \cite[Chapter 7]{Huber2009}. For this reason, in the field of robust statistics, a number of alternative loss functions are used. One popular choice was introduced by Huber \cite{huber1964}.

\begin{Definition} \label{def_hub}
Let $c > 0$. Huber's loss function is $\phi_c(u) = \frac{u^2}{2} \mathbb{I}_{|u| \leq c} + c\left( |u| - \frac{c}{2} \right) \mathbb{I}_{|u| > c} $.
\end{Definition}

When $c \rightarrow + \infty$, $\phi_c$ converges to the least-squares loss.
When $c \rightarrow 0$, $\frac{1}{c} \phi_c$  converges to the absolute value loss $x \rightarrow |x|$ of median regression. Thus, the $c$ parameter allows a trade-off between robustness and approximation of the least squares loss. 

The rest of the article will focus on sparse linear regression with the loss function $\phi_c$. Thus, notations $\bayes$, $\loss{t}$ and $\risk$ are to be understood with respect to $\phi_c$. 
\subsection{Sparse linear regression} \label{sec_hub}

With finite data, it is impossible to solve the optimization problem $\min \risk(t)$ over the set of all predictors $t$. Some modeling assumptions must be made to make the problem tractable. A popular approach is to build a finite set of features $(\psi_j(X))_{1 \leq j \leq d}$ and consider predictors that are linear in these features: $\exists \theta \in \mathbb{R}^d, \forall x \in \cX, t(x) = \sum_{j = 1}^d \theta_j \psi_j(x)$. This is equivalent to replacing $X \in \cX$ with $\tilde{X} = (\psi_j(X))_{1 \leq j \leq d} \in \mathbb{R}^d$ and regressing $Y$ on $\tilde{X}$. For theoretical purposes, it is thus equivalent to assume that $\cX = \mathbb{R}^d$ for some $d$ and predictors are linear: $t(x) = \theta^T x$.

As the aim is to reduce the risk $\risk(t)$, a logical way to choose $\theta$ is by \emph{empirical risk minimization}:
\[\hat{\theta} \in \argmin_{\theta \in \mathbb{R}^d} \frac{1}{n} \sum_{i = 1}^n \phi_c(Y_i - \theta^T X_i).\]
Empirical risk minimization works well when $d \ll n$ but will lead to overfitting in large dimensions \cite{Vapnik:1999}.
Indeed, if $d$ is too large, no estimator can succeed at minimizing the risk over $\mathbb{R}^d$, as the following proposition shows.

\begin{proposition} \label{prop_minimax_huber_reg}
Let $\sigma > 0$ and $\Sigma$ be a positive definite matrix of dimension $d$.
For any $\theta \in \mathbb{R}^d$, let $P_\theta$ denote the distribution such that $(X,Y) \sim P_{\theta}$ iff almost surely, $Y = \langle \theta, X \rangle + \sigma \varepsilon$, where $X \sim \mathcal{N}(0,\Sigma)$, $\varepsilon \sim \mathcal{N}(0,1)$ and $\varepsilon,X$ are independent. 
 Then for any $n > d$,
\[ \inf_{\hat{\theta}} \sup_{\theta \in \mathbb{R}^d} \mathbb{E}_{D_n \sim P_{\theta}^{\otimes n}} \left[\ell\bigl(\theta^T,\hat{\theta}(D_n)^T \bigr) \right] 
\geq E \bigl[ \min(\sigma^2 \varepsilon^2, c \sigma |\varepsilon| ) \bigr] 
\left( \sqrt{1 + \frac{2 d}{\pi n}} - 1 \right), \]
where $\inf_{\hat{\theta}}$ denotes the infimum over all estimators and $\theta^T$ denotes the linear functional $x \mapsto \langle \theta, x \rangle$.
\end{proposition} 

Proposition \ref{prop_minimax_huber_reg} is proved in appendix \ref{proof_prop_minimax}. 
With respect to $\sigma$, the lower bound of proposition \ref{prop_minimax_huber_reg} scales as $\sigma^2$ when $\sigma \ll c$ and 
as $c\sigma$ when $\sigma \gg c$, as could be expected from the definition of the Huber loss (Definition \ref{def_hub}). With respect to $d$ and $n$, it scales as $\frac{d}{n}$ when $d \ll n$. Moreover, there is a positive lower bound on the minimax risk when $d$ is of order $n$. 
Thus, for such large values of $d$, consistent risk minimization cannot be achieved uniformly over the whole of $\mathbb{R}^d$.

Sparse regression attempts instead to locate a ``good'' subset of variables in order to optimize risk for a given model dimension.
The Lasso \cite{Tibshirani1996} is now a standard method of achieving sparsity. The specific version of the Lasso which we consider here is given by the following definition.

\begin{Definition} \label{def_lasso}
Let $n \in \mathbb{N}$ and let $D_n = (X_i,Y_i)_{1 \leq i \leq n}$ be a dataset 
such that $X_i \in \mathbb{R}^d$ and $Y_i \in \mathbb{R}$ for all $i \in [|1;n|]$ and some $d \in \mathbb{N}$.
 Let $\phi_c$ be the Huber loss defined in Definition \ref{def_hub}. For any $r \geq 0$, let
 \[ \hat{\mathcal{C}}(r) = \argmin_{(q,\theta) \in \mathbb{R}^{d+1}: \Norm{\theta}_1 \leq r} \frac{1}{n} \sum_{i = 1}^n 
 \phi_c\bigl(Y_i - q - \theta^T X_i \bigr) \text{ and }\]
 \begin{equation} \label{eq_tie_break}
  ( \hat{q}(r), \hat{\theta}(r) ) \in \argmin_{(q,\theta) \in \hat{\mathcal{C}}(r)} \bigl| q  + <\theta, \frac{1}{n} \sum_{i = 1}^n X_i> \bigr|.
 \end{equation}
 Now let 
 \[ \learnrule^{lasso}(r)(D_n) : x \rightarrow \hat{q}(r) + \hat{\theta}(r)^T x.  \]
\end{Definition}
The intercept $q$ is left unconstrained in definition \ref{def_lasso}, as is usually the case in practice \cite{hqreg:2017}. 
Equation \eqref{eq_tie_break} is a tiebreaking rule which simplifies the theoretical analysis.

\subsection{Hyperparameter tuning}

The zero-norm of a vector $\theta$ is the integer $\Norm{\theta}_0 = |\{i : \theta_i \neq 0 \}|$. Many sparse estimators, such as best subset or forward stepwise \cite[Section 3.3]{Hastie2009}, are directly parametrized by their desired zero-norm, which must be chosen by the practitioner. It controls the ``complexity'' of the estimator, and hence the bias-variance tradeoff.  
In the case of the standard Lasso (Definition \ref{def_lasso} with $\phi(x) = x^2$), Zou, Hastie and Tibshirani \cite{Zou_Has_Tib:2007} 
showed that $\Norm{\hat{\theta}(\lambda)}_0$ is an unbiased estimator of the ``degrees of freedom'' of the estimator $\learnrule(\lambda)$. 
As a consequence, \cite{Zou_Has_Tib:2007} suggests reparametrizing the lasso by its zero-norm. Applying their definition to the present setting yields the following.

\begin{Definition} \label{def_param_lasso}
For any dataset $D_n$, let $(\hat{q},\hat{\theta})$ be given by Definition \ref{def_lasso}, equation \eqref{eq_tie_break} .
Let $M \in \mathbb{N}$ and $(r_m)_{1 \leq m \leq M}$ be the finite increasing sequence at which the sets $\{ i : \hat{\theta}(r)_i \neq 0 \}$ change. Let $r_0 = 0$.
 For any $k \in \mathbb{N}$ let
 \[ \hat{m}^{last}_{k,R} = \max \left\{ m \in \mathbb{N} | \bigl\| \hat{\theta}(r_m) \bigr\|_0 = k \text{ and } r_m \leq R \right\},  \]
 with the convention $\max \emptyset = 0$.
 Let then
 \begin{equation}
  \learnrule^{lasso}_{k,R} (D_n) = \learnrule^{lasso} \left( r_{\hat{m}^{last}_{k,R}} \right) \left(D_n \right).
 \end{equation}
 Let $\learnrule^{lasso}_{k} = \learnrule^{lasso}_{k,+ \infty}$ 
 denote the unconstrained sequence (corresponding to \cite{Zou_Has_Tib:2007}'s original definition).
\end{Definition}
The (optional) constraint $\Norm{\hat{\theta}(r_m)}_{\ell^1} \leq r_m \leq R$ has some potential practical and theoretical benefits. From the practical viewpoint, it allows to reduce the computational complexity by excluding lasso solutions with excessively large $\ell^1$ norm, which may be expected to perform poorly anyway. From a theoretical viewpoint, it helps control the $L^p$ norms of the predictor $\langle \hat{\theta}(r_m), X \rangle$, thus avoiding inconsistency issues encountered by the empirical risk minimizer for some pathological designs \cite{Mourtada:2019} .

More generally, consider \emph{any} sequence $\left(\learnrule_k \right)_{k \in \mathbb{N}}$ of learning rules which output linear predictors $\learnrule_k(D_n): x \rightarrow \hat{q}_k(D_n) + \langle \hat{\theta}_k(D_n), x \rangle $. To prove the main theoretical result of this article (Theorem \ref{agcv_hub}), we make the following assumptions on the collection $\left(\learnrule_k \right)_{k \in \mathbb{N}}$.
\begin{hyp} \label{hyp_sparse_reg}
For any $n \in \mathbb{N}$, let $D_n \sim P^{\otimes n}$ denote a dataset of size $n$. Assume that
 \begin{enumerate}
   \item Almost surely, for all $k \in [|1;n|]$, $\Norm{\hat{\theta}_k(D_n)}_0 \leq k$.

 \item For all $k \in [|1;n|]$, $\hat{q}_k(D_n) \in \argmin_{q \in \hat{Q}\left( D_n,\hat{\theta}_k(D_n) \right)} \left| q + \langle\hat{\theta}_k(D_n), \frac{1}{n} \sum_{i = 1}^n X_i \rangle  \right|$, 
 
 where
 $\hat{Q}(D_n,\theta) = \argmin_{q \in \mathbb{R}} \frac{1}{n} \sum_{i = 1}^n \phi_c \left( Y_i - \langle \theta, X_i \rangle - q \right).$
 \end{enumerate}
\end{hyp}
For the reparametrized Lasso given by definition \ref{def_lasso} and \ref{def_param_lasso}, hypothesis \ref{hyp_sparse_reg} holds by construction.

Moreover, condition 1 is naturally satisfied by 
 such sparse regression methods as forward stepwise and best subset \cite[Section 3.3]{Hastie2009}.
 Condition 3 states that the intercept $q$ is chosen by empirical risk minimization, with a specific tie-breaking rule in case the minimum is not unique.

\subsection{Aggregated hold out applied to the zero-norm parameter}

The tuning of the zero-norm $k$ is important to ensure good prediction performance by optimizing the bias-variance tradeoff. 
Depending on the application, practicioners may want more or less sparsity, depending on their requirements in terms of computational load or interpretability. 
For this reason, we consider the problem of selecting the zero-norm among the set $\{1,\ldots,K\}$, for some $K \in \mathbb{N}$ which may depend on the sample size. This article investigates the use of Agghoo in this context, as an alternative to cross-validation. Agghoo is a general hyperparameter aggregation method which was defined in \cite{agghoo_rkhs}, in a general statistical learning context.
Let us briefly recall its definition in the present setting. For a more detailed introductory discussion of this procedure, we refer the reader to \cite{agghoo_rkhs}. 
 To simplify notations, fix a collection $(\hat{q}_k, \hat{\theta}_k)_{1 \leq k \leq K}$ of linear regression estimators.
First, we need to define \emph{hold-out} selection of the zero-norm parameter.
\begin{Definition} \label{def_ho}
 Let $D_n = (X_i,Y_i)_{1 \leq i \leq n}$ be a dataset. For any $T \subset \{1,\ldots,n\}$, denote $D_n^T = (X_i,Y_i)_{i \in T}$.
 Let then
 \[ \hat{k}_T(D_n) = \min \argmin_{1 \leq k \leq K} \frac{1}{|T^c|} \sum_{i \notin T} \phi_c \left(Y_i - \hat{q}_k(D_n^T) - \langle \hat{\theta}_k(D_n^T),X_i  \rangle \right). \]
 Using the hyperparameter $\hat{k}_T(D_n)$ together with the dataset $D_n^T$ to train a linear regressor yields the \emph{hold-out predictor}
 \[ \ERMho{T}(D_n): x \rightarrow \hat{q}_{\hat{k}_T(D_n)}(D_n^T) + \langle \hat{\theta}_{\hat{k}_T(D_n)}(D_n^T), x \rangle. \]
\end{Definition}
Aggregation of hold-out predictors is performed in the following manner.
\begin{Definition} \label{def_agghoo}
 Let $\cT = (T_1,\ldots,T_V)$ be a collection of subsets of $\{1,\ldots,n\}$,
 where $V = |\cT|$. Let:
 \begin{align*}
  \hat{\theta}^{ag}_{\cT} &= \frac{1}{V} \sum_{i = 1}^V \hat{\theta}_{\hat{k}_{T_i}(D_n)}(D_n^{T_i}) \\
  \hat{q}^{ag}_{\cT} &= \frac{1}{V} \sum_{i = 1}^V \hat{q}_{\hat{k}_{T_i}(D_n)}(D_n^{T_i}).
 \end{align*}
Agghoo outputs the linear predictor:
\[ \ERMag{\cT}(D_n): x \rightarrow \hat{q}^{ag}_{\cT} + \langle \hat{\theta}^{ag}_{\cT}, x  \rangle . \]
\end{Definition}

Thus, Agghoo also yields a linear predictor, which means that it can be efficiently evaluated on new data. If the $\hat{\theta}_{\hat{k}_T(D_n)}$ have similar support, $\hat{\theta}^{ag}_{\cT}$ will also be sparse: this will happen if the hold-out reliably identifies a true model. On the other hand, if the supports have little overlap, the Agghoo coefficient will lose sparsity, but it can be expected to be more stable and to perform better.

The linear predictors $x \rightarrow \hat{q}_{\hat{k}_{T_i}(D_n)}(D_n^{T_i}) + \langle \hat{\theta}_{\hat{k}_{T_i}(D_n)}(D_n^{T_i}), x\rangle$ aggregated by Agghoo are only trained on part of the data. This subsampling (typically) decreases the performance of each individual estimator, but combined with aggregation, it may stabilize an unstable procedure and improve its performance, similarly to bagging. 

An alternative would be to \emph{retrain} each regressor on the whole data-set $D_n$, yielding the following procedure, which we call "Aggregated cross-validation" (Agcv).

\begin{Definition} \label{def_acv}
 Let $\cT = (T_1,\ldots,T_V)$ be a collection of subsets of $\{1,\ldots,n\}$,
 where $V = |\cT|$. Let:
 \begin{align*}
  \hat{\theta}^{acv}_{\cT} &= \frac{1}{V} \sum_{i = 1}^V \hat{\theta}_{\hat{k}_{T_i}(D_n)}(D_n) \\
  \hat{q}^{acv}_{\cT} &= \frac{1}{V} \sum_{i = 1}^V \hat{q}_{\hat{k}_{T_i}(D_n)}(D_n).
 \end{align*}
The output of Agcv is the linear predictor:
\[ \ERMacv{\cT}(D_n): x \rightarrow \hat{q}^{acv}_{\cT} + \langle \hat{\theta}^{acv}_{\cT}, x  \rangle . \]
\end{Definition}

Agghoo is easier to study theoretically than Agcv due to the conditional independence: $\left( \hat{\theta}_k \bigl(D_n^T \bigr) \right)_{1 \leq k \leq K} \independent \hat{k}_T(D_n) \ \Bigl| D_n^T $. For this reason, the theoretical section will focus on Agghoo, while in the simulation study, both Agghoo and Agcv will be considered.

In comparison to Agghoo and Agcv, consider the following definition of a general cross-validation method.

\begin{Definition} \label{def_cv}
Let $\cT = (T_1,\ldots,T_V)$ be a collection of subsets of $\{1,\ldots,n\}$,
 where $V = |\cT|$.
Let
 \[ \hat{k}_{\cT}^{cv}(D_n) = \min \argmin_{1 \leq k \leq K} \frac{1}{V} \sum_{j = 1}^V \frac{1}{|T_j^c|} \sum_{i \notin T_j} \phi_c \left(Y_i - \hat{q}_k(D_n^{T_j}) - \langle \hat{\theta}_k(D_n^{T_j}),X_i  \rangle \right). \]
 Let then
 \begin{align*}
 \hat{\theta}^{cv}_{\cT} &= \hat{\theta}_{\hat{k}_{\cT}^{cv}(D_n)}(D_n) \\
  \hat{q}^{acv}_{\cT} &= \hat{q}_{\hat{k}_{\cT}^{cv}(D_n)}(D_n).
 \end{align*}
 CV outputs the linear predictor 
 \[ \ERMcv{\cT}(D_n): x \rightarrow \hat{q}^{cv}_{\cT} + \langle \hat{\theta}^{cv}_{\cT}, x  \rangle . \]
\end{Definition}

This makes clear the difference between cross-validation and Agghoo (or Agcv): cross-validation averages the hold-out \emph{risk estimates} (and selects a \emph{single} linear predictor) whereas Agghoo and Agcv aggregate the \emph{selected predictors} $(\hat{q}_{\hat{k}_{T_i}}, \hat{\theta}_{\hat{k}_{T_i}})$. 
If the parameter $\hat{k}_{\cT}^{cv}$ is used instead of the $\hat{k}_{T_i}$ in Definition \ref{def_agghoo}, this yields the bagged CV method of Lecué and Mitchell \cite{lecue2012}. This method applies bagging to individual estimators $\hat{q}_k, \hat{\theta}_k$, whereas Agghoo also bags the \emph{estimator selection} step. When there is a single, clearly established optimal model of small dimension, the advantages of a more accurate model selection step (as in CV and its bagged version) may outweigh the gains due to aggregation. In contrast, when there are many different sparse linear predictors with close to optimal performance, model selection will be unstable and aggregation should provide benefits relative to selection of a single parameter $k$.

\subsection{Computational complexity}
There are two types of computational costs to take into account
when considering a (sparse) linear predictor such as $\ERMag{\cT}(D_n)$: the cost of \emph{calculating} the parameters $\hat{q}^{ag}_{\cT}(D_n), \hat{\theta}^{ag}_{\cT}(D_n)$ at \emph{training time} and the cost of \emph{making a prediction} on new data, i.e computing $\ERMag{\cT}(D_n)(x)$ for some $x$.
In this section, Agghoo, Agcv and cross-validation are compared with respect to these two types of complexity.

Let $(\hat{q}_k, \hat{\theta}_k)_{1 \leq k \leq K}$ be some finite collection of sparse linear regression estimators.
Let $S(n) = \mathbb{E} \left[\max_{1 \leq k \leq K} \Norm{\hat{\theta}_k(D_n)}_0 \right]$ denote the expected maximal number of non-zero coefficients. 
In particular, under point 1 of hypothesis \ref{hyp_sparse_reg},
$S(n) \leq K$.  Let $V = |\cT|$ and $n_v = n - n_t$, where $n_t$ is given by hypothesis \textbf{(Reg-}$\cT$\textbf{)} below (equation \ref{hyp.T}).

\paragraph{Computational complexity at training time}

 Agghoo, Agcv and cross-validation must all compute the hold-out risk estimator for each subset in $\mathcal{T}$ and each $k \in \{1,\ldots,K\}$. Let $\hat{C}_{hos}$ denote the number of operations needed for this. 

For a given subset $T_i$, the estimators $\hat{q}_k(D_n^{T_i}), \hat{\theta}_k(D_n^{T_i})$ must be computed for all $k$, which may be more or less expensive depending on the method. In the case of the Lasso, the whole path can be computed efficiently using the LARS-Lasso algorithm \cite{efron2004}. 

Then, the empirical risk of all estimators must be calculated on the test set. On average, this takes at least $S(n_t) n_v$ operations to compute the risk of the least sparse $\hat{\theta}_k$ ($n_v$ scalar products involving an average of $S(n_t)$ non-zero coefficients) and at most $\mathcal{O}(K S(n_t) n_v)$ operations in general. In particular, $\mathbb{E}[\hat{C}_{hos}] \geq V S(n_t) n_v$. 

In a next step, Agghoo and agcv compute the minima of $V$ vectors of length $K$, whereas cross-validation averages these vectors and calculates the argmin of the average. Both operations have complexity of order $VK$.

It is in their final step that the three methods differ slightly. 
Agghoo uses the $\hat{\theta}_{\hat{k}_{T_i}}(D_n^{T_i})$ which have been computed in a previous step, whereas
Agcv and cross-validation must compute the $\hat{\theta}_{\hat{k}_{T_i}}(D_n)$ and $\hat{\theta}_{\hat{k}^{cv}_\cT}(D_n)$, respectively. The complexity of this depends on the method, but can be expected to be small compared to $\hat{C}_{hos}$, as there is only one estimator to fit instead of $K$.   

Finally, Agghoo and Agcv must aggregate $V$ vectors drawn from the $\hat{\theta}_k(D_n^{T_i})$ and $\hat{\theta}_k(D_n)$, with respective complexity $\mathcal{O}(V S(n_t))$ and $\mathcal{O}(VS(n))$, provided that a suitably "sparse" representation is used  for the $\hat{\theta}_k$. Assuming $S(n) \approx S(n_t)$, this is negligible compared to $\mathbb{E}[\hat{C}_{hos}]$.

All in all, Agghoo, Agcv and cross-validation have a similar complexity at training time, of order $\mathbb{E}[\hat{C}_{hos}] + VK$, with $\mathbb{E}[\hat{C}_{hos}]$ most likely being the dominant term. 

\paragraph{Evaluation on new data}

Given new data $x$, the complexity of evaluating $q+ \langle \theta, x \rangle$ is proportional to $\Norm{\theta}_0$. If the sparse estimators $\hat{\theta}_k$ perform as intended and consistently identify similar subsets of predictive variables, then Agghoo and Agcv sould not lose much sparsity compared to CV, as the $\hat{\theta}_{\hat{k}_{T_i}}(D_n^{T_i}), \hat{\theta}_{\hat{k}_{T_i}}(D_n)$ and $\hat{\theta}^{cv}_{\cT}$ should all have similar supports.

At worst, if the supports of the $\hat{\theta}_{\hat{k}_{T_i}}(D_n^{T_i})$ are disjoint, $\Norm{\hat{\theta}^{ag}_{\cT}}_0$  may be as much as $V$ times greater than $\Norm{\hat{\theta}_{\hat{k}_{T_1}}(D_n^{T_1})}_0$. In contrast, $\Norm{\hat{\theta}^{cv}_{\cT}}_0 = \Norm{\hat{\theta}_{\hat{k}_{\cT}^{cv}(D_n)}(D_n)}$ should heuristically be of the same order as  $\Norm{\hat{\theta}_{\hat{k}_{T_1}}(D_n^{T_1})}_0$ -- as both $\hat{k}_{\cT}^{cv}$ and $\hat{k}_{T_1}$ optimize the same bias-variance tradeoff with respect to the "complexity parameter" $k$ . However, this situation is one in which the hold-out is very unstable, so Agghoo can be expected to yield significant improvements in exchange for the increased computational cost. The same argument applies to agcv.

\section{Theoretical results}
Let $n \in \mathbb{N}$ and $D_n = (X_i, Y_i)_{1 \leq i \leq n}$ denote an i.i.d dataset with common distribution $P$. Let $\left(\hat{q}_k,\hat{\theta}_k \right)_{1 \leq k \leq K}$ be a collection of linear regressors which satisfies assumption \ref{hyp_sparse_reg}. Let $\cT$ be a collection of subsets of $\{1,\ldots,n\}$. In this section, we give bounds for the risk of the Agghoo estimator $\ERMag{\cT}$ (Definition \ref{def_agghoo}) built from the collection $\left(\hat{q}_k,\hat{\theta}_k \right)_{1 \leq k \leq K}$.

\subsection{Hypotheses}
To state and prove our theoretical results, a number of hypotheses are required. 
First, the collection of subsets $\cT$ - chosen by the practitioner - should satisfy the following two conditions.
\paragraph{(Reg$-\cT$)}
There exists an integer $n_t$ such that $\max(3,\tfrac{n}{2}) \leq n_t < n$ and
\begin{equation} \label{hyp.T}
\begin{split}
 \cT &\subset \{ T \subset \{ 1,\ldots,n \} : |T| = n_t \} \\
 \cT &\text{ is independent from } D_n \enspace.
\end{split}
\end{equation} 
Let also $n_v = n - n_t$ denote the size of the validation sets.

\paragraph{}
Independence of $\cT$ from $D_n$ ensures that for $T \in \cT$, $D_n^T$ is also iid with distribution $P$. The assumption that $\cT = (T_1, \ldots, T_V)$ contains sets of equal size ensures that the pairs $\hat{q}_{\hat{k}_{T_i}(D_n)}(D_n^{T_i}), \hat{\theta}_{\hat{k}_{T_i}(D_n)}(D_n^{T_i})$ are equidistributed for $i \in \{1,\ldots,V\}$. Most of the data partitioning procedures used for cross-validation satisfy hypothesis \textbf{(Reg-}$\cT$ \textbf{)},
including leave-$p$-out, $V$-fold cross-validation (with $n - n_t = n_v =n/V$) and 
Monte-Carlo cross-validation \cite{Arl_Cel:2010:surveyCV}. 

To state an upper bound for $\loss{\ERMag{\cT}}$, we also need to quantify the amount of noise in the distribution of $Y$ given $X$, in a way appropriate to the Huber loss $\phi_c$. That is the purpose of the following assumption.  
 
%

\paragraph{(Lcs)}
Let $(X,Y) \sim P$.
 Let $\bayes$ denote an optimal predictor, i.e a measurable function $\mathbb{R}^{d} \to \mathbb{R}$ such that $s(x) \in \argmin_{u \in \mathbb{R}} \mathbb{E}[\phi_c(Y - u) | X = x]$ for almost all $x \in \mathbb{R}^{d}$. Assume that there exists $\bayes$ and a positive real number $\eta$ such that
 \begin{equation} \label{hyp_noise}
  P\left[ |Y - \bayes(X)| \leq \frac{c}{2} \ \Big| X \right] \geq \eta \text{ a.s},
 \end{equation}
 where $c$ denotes the parameter of the Huber loss.
\paragraph{}
Equation \eqref{hyp_noise} is specific to the Huber loss: it requires the conditional distribution of the residual $Y - \bayes(X)$ to put sufficient mass in a region where the Huber function $\phi_c$ is quadratic. 
For example, assume that $Y = s(X) + \sigma \varepsilon$ where $\varepsilon$ is independent from $X$ and has a continuous, positive density $q$ in a neighbourhood of $0$. If the Huber parameter $c$ is proportional to or larger than $\sigma$, then a constant value of $\eta$ can be chosen, independently of $\sigma$. On the other hand, if $c \ll \sigma$, the optimal value of $\eta$ satisfies $\eta = \eta(\sigma) \sim_{\frac{c}{\sigma} \to 0} \frac{q(0)c}{2 \sigma}$. 

Finally, some hypotheses are needed to deal with pathological design distributions which can in general lead to inconsistency of empirical risk minimization \cite{Mourtada:2019}. To illustrate the problem as it applies to the hold-out, consider a distribution $P$ such that $0 < P(X \in H) < 1$ for some vector subspace $H$, as in \cite{Mourtada:2019}. Assume to simplify that $Y = \langle \theta_*,X \rangle + \varepsilon$.
Let $p_H$ denote the orthogonal projection on $H$.
With small, but positive probability, $X_i \in H$ for all $i \in \{1,\ldots,n\}$.
On this event, it is clearly impossible to estimate $\theta_* - p_H(\theta_*)$. Likewise, the hold-out cannot correctly assess the impact of the orthogonal components $\hat{\theta}_k - p_H(\hat{\theta}_k)$ of the estimators $\hat{\theta}_k$ on the risk, since $\langle \hat{\theta}_k, X_i \rangle $ only depends on $p_H(\hat{\theta}_k)$, whereas out of sample predictions $\langle \hat{\theta}_k, X \rangle$ may depend on $\hat{\theta}_k - p_H(\hat{\theta}_k)$ (since $P(X \in H) < 1$).  
This means that the hold-out-selected predictors $\ERMho{T_i}$ may be arbitrarily far from optimal in general. 

To avoid this issue, two sets of assumptions have been made in the litterature. First, there are boundedness assumptions: for example, if the predictors $ \hat{q}_k + \langle \hat{\theta}_k, X \rangle$ and the variable $Y$ are uniformly bounded, this clearly limits the impact of low-probability events such as $\{ \forall i \in \{1,\ldots,n \}, X_i \in H \}$ on the risk.
Such hypotheses have been used to prove general oracle inequalities for the hold-out \cite[Chapter 8]{Gyrfi2002} \cite[Corollary 8.8]{Mas:2003:St-Flour} and cross-validation \cite{vdL_Dud_vdV:2006}. Alternatively, pathological designs can be excluded from consideration by assuming an $L^p - L^q$ norm inequality or "small ball" type condition \cite{Mendelson2014LearningWC, Mendelson2018}: this has been used to study empirical risk minimization over linear models  \cite{Mourtada:2019, audibert2011}.  

In this article, a combination of both approaches is used.    
First, we assume a weak uniform upper bound on $L^1$ norms of the predictors (hypothesis \textbf{(Uub)}). The bound is allowed to grow with $n_t$ at an arbitrary polynomial rate. 

\paragraph{(Uub)}
Let $(X_i,Y_i)_{1 \leq i \leq n_t} = D_{n_t}$ be iid with distribution $P$, where $n_t$ is given by hypothesis \textbf{(Reg-}$\cT$ \textbf{)}. Let $X \sim X_1$ be independent from $D_{n_t}$. 
 There exist real numbers $L,\alpha$ such that 
 \begin{enumerate}
 \item $\mathbb{E}\Bigl[ \max_{1 \leq k \leq n_t } \max_{1 \leq i \leq n_t} \bigl| \langle \hat{\theta}_k(D_{n_t}), X_i - EX \rangle \bigr| \Bigr] \leq L n_t^{\alpha}$  
 \item $\mathbb{E}\Bigl[ \max_{1 \leq k \leq n_t } \mathbb{E} \left[ \bigl| \langle \hat{\theta}_k(D_{n_t}), X - EX \rangle \bigr| | D_{n_t} \right] \Bigr] \leq L n_t^{\alpha}.$ 
 \end{enumerate}

\paragraph{}
For the Lasso, if $R \leq n_t^{\alpha_1}$ in Definition \ref{def_lasso}, then hypothesis \textbf{(Uub)}
holds if in addition $E \left[ \NormInfinity{X - E X} \right] \leq n_t^{\alpha - \alpha_2}$. This is the case if the components of $X$ have variance $1$ and $d $ is polynomial in $n$, or if the components of $X$ are sub-exponential 
with constant $1$ and $\log p$ is polynomial in $n$.

Hypothesis \textbf{(Uub)} is much weaker than boundedness assumptions usually made in the litterature, where typically the $L^\infty$ norm is used instead of the $L^1$ norm, and the bound is a constant rather than a polynomial function of $n_t$.
Point $1$ of Hypothesis \textbf{(Uub)} is natural in the sense that an estimator $\hat{\theta}_k$ which violates it cannot perform well anyway: assuming that $P(|Y|) < +\infty$ , by definition of $\phi_c$, for any $(q,\theta)$,
 \begin{align*}
 E \left[ \phi_c(Y - q - \langle \theta, X \rangle) \right]
 &\geq c E \left[ \bigl|Y - q - \langle \theta, X \rangle \bigr| \right] - \frac{c^2}{2} \\
 &\geq c E \left[ \bigl| q + \langle \theta, X \rangle \bigr| \right] - c E [|Y|] - \frac{c^2}{2} \\
 &\geq \frac{c}{2} E \left[ \bigl| \langle \theta, X - E X \rangle \bigr| \right] - c E [|Y|] - \frac{c^2}{2}. \numberthis \label{eq_comp_huber_risk_l1_norm}
\end{align*}   
 Thus, if $\mathbb{E}  \left[ \bigl| \langle \hat{\theta}_k(D_{n_t}), X - P X \rangle \bigr| \right]$ grows faster than $n_t^\alpha$, then so do the expected risk and expected excess risk  of $\learnrule_k(D_{n_t})$. 
 Point $2$ of Hypothesis \textbf{(Uub)} can be seen as an "empirical version" of point 1, 
 wherein the independent variable $X$ is replaced by the elements of $D_{n_t}$.
 The lack of independence between $\hat{\theta}_k$ and $X_i$ makes this condition less straightforward than $1$. However, by the Cauchy-Schwarz inequality, it is always the case that $\mathbb{E} \left[ \bigl| \langle \hat{\theta}_k, X_i - P X_i \rangle \bigr|  \right] \leq \sqrt{d} \mathbb{E}[\langle \hat{\theta}_k, X - P X \rangle^2]^{\frac{1}{2}}$. Thus, it is enough to suppose that $d $ and $\mathbb{E}[\langle \hat{\theta}_k, X - P X \rangle^2]$ are bounded by $L n_t^{\alpha}$ for some $\alpha > 0$.
 

Together with the weak uniform bound \textbf{(Uub)}, we assume that for sparse linear predictors $x \mapsto \langle \theta, x - EX \rangle$ with $\Norm{\theta}_0 \leq K$, the $L^2$ norm is equivalent to the stronger "Orlicz norm" defined below.

\begin{Definition}  \label{def_orlicz}
Let $Z$ be a real random variable. Let $\psi_1: x \mapsto e^x - 1$.
The $\psi_1-$norm of $Z$ is defined by the formula
\[ \Norm{Z}_{L^{\psi_1}} = \inf \left\{ u > 0 : E \left[ 
\psi_1 \left( \frac{Z}{u} \right) \right] \leq 1 \right\}, \]
with the convention $\inf \emptyset = + \infty$. We say that $Z \in L^{\psi_1}$ if $\Norm{Z}_{L^{\psi_1}} < +\infty$. 
\end{Definition}

Plainly, $\Norm{Z}_{L^{\psi_1}} < +\infty$ if and only if $Z$ is sub-exponential;
it can be shown that $\Norm{\cdot}_{L^{\psi_1}}$ is indeed a norm.

The constant relating $\Norm{\cdot}_{L^{\psi_1}}$ and $\Norm{\cdot}_{L^2}$
is allowed to depend on $n_t$ in the following way.

\paragraph{(Ni)}
Let $(X,Y) \sim P$ and $\bar{X} = X - P X$.
For any $m \in \mathbb{N}$, let
  \begin{equation}  \label{def_kappa}
 \kappa(m) = \sup_{ \theta \neq 0, \Norm{\theta}_0 \leq 2m} \frac{\Norm{\langle \bar{X},\theta \rangle}_{L^{\psi_1}}}{\Norm{\langle \bar{X},\theta \rangle}_{L^2}} \vee \frac{1}{\log 2}.
\end{equation}
There exists a constant $\nu_0$ such that
\begin{equation} \label{Hw_huber}
 \kappa(K) \log \kappa(K)  \leq \nu_0 \sqrt{\frac{n_v}{\log(n_t \vee K)}}. 
\end{equation}

\paragraph{}
The interpretation of this hypothesis is not obvious. Note first that $\kappa(K)$ is a non-decreasing function of $K$,
and in particular,
\[ \kappa(K) \leq \kappa(d) = \sup_{\theta \neq 0} \frac{\Norm{\langle \bar{X}, \theta \rangle}_{L^{\psi_1}}}{\Norm{\langle \bar{X}, \theta \rangle}_{L^2}}.  \]
Unlike $\kappa(K)$, $\kappa(d)$ is invariant under linear transformations of $X$: in other words, it only depends on the linear space $V$ spanned by the columns of $X$. In particular, $\kappa(d)$ does not depend on the covariance matrix of $X$, provided that it is non-degenerate. The inequality $\Norm{\langle \bar{X}, \theta \rangle}_{L^{\psi_1}} \leq \kappa(d)\Norm{\langle \bar{X}, \theta \rangle}_{L^2} $ can be interpreted as an effective, scale invariant version of sub-exponentiality: it states that the tail of $\langle \bar{X}, \theta \rangle$ is sub-exponential with a scale parameter which isn't too large compared to its standard deviation. In sections \ref{sec.gauss} , \ref{sec.nonparam} and \ref{sec.trig}, we shall give examples where simple bounds can be proved for $\kappa(K)$ or $\kappa(d)$.

\subsection{Main Theorem}
When Agghoo is used on a collection $(\learnrule_k)_{1 \leq k \leq K}$ of linear regression estimators satisfying Hypothesis \eqref{hyp_sparse_reg}, such as the Lasso parametrized by the number of non-zero coefficients, as in Definition \ref{def_param_lasso}, the following theorem applies.

\begin{Theorem} \label{agcv_hub}
 Let $X \in \mathbb{R}^d$ and $Y \in \mathbb{R}$ be random variables with joint distribution $P$ such that hypothesis \textbf{(Lcs)} holds. Let $D_n = (X_i,Y_i)_{1 \leq i \leq n} \sim P^{\otimes n}$ be a dataset of size $n$. Let $n_v = n - n_t$, where $n_t$ is given by assumption \textbf{(Reg-}$\cT$ \textbf{)}. Let $c$ denote the Huber loss parameter from Definition \ref{def_hub}. 
 

 Let $K$ be an integer such that $3 \leq K \leq e^{\sqrt{n_v}}$ and $(\learnrule_k)_{1 \leq k \leq K}$ be a collection of linear regression estimators which satisfies hypothesis \eqref{hyp_sparse_reg}. Assume that hypotheses \textbf{(Ni)} and \textbf{(Uub)} hold. 

There exist numerical constants $\mu_1 > 0,\mu_2 \geq 1$ such that,
for any $\theta \in \mathbb{R}$ such that $\sqrt{\alpha + 3} \frac{\mu_2 \nu_0}{\eta} \leq \theta < 1$,
\begin{equation} \label{or_ineq_hub}
 (1-\theta) \mathbb{E}\Bigl[\loss{\ERMag{\cT}} \Bigr] \leq  (1+\theta)\mathbb{E}\Bigl[\min_{1 \leq k \leq K} \loss{\learnrule_k(D_{n_t})} \Bigr] 
+ 54 (\alpha + 3) \frac{c^2 \log (K \vee n_t)}{\theta \eta n_v} + \frac{7 \mu_1 L c \log K}{\theta n_t \sqrt{n_v}}.
\end{equation}
\end{Theorem}

Theorem \ref{agcv_hub} is proved in appendix \ref{app.sec_proof_thm}.
Theorem \ref{agcv_hub} compares the excess risk of Agghoo to that of the best linear predictor in the collection $\learnrule_k(D_{n_t})$,
trained on a subset of the data of size $n_t$.
Taking $|\cT| = 1$ in Theorem \ref{agcv_hub} yields an oracle inequality for the hold-out, which is also cross-validation with one split.
It is, to the best of our knowledge, the first theoretical guarantee on hyperparameter aggregation (or selection) for the huberized Lasso.  That $n_t$ appears in the oracle instead of $n$ is a limitation, but it is logical, since estimators aggregated by Agghoo are only trained on samples of size $n_t$. Typically, the  excess risk increases at most by a constant factor when a dataset of size $n$ is replaced by a subset of size $\tau n$, and this constant tends to $1$ as $\tau \rightarrow 1$.
This allows to take $n_v$ of order $n$ ($n_v = (1-\tau)n$), while losing only
a constant factor in the oracle term. 

In addition to the oracle, $\mathbb{E}\Bigl[\min_{1 \leq k \leq K} \loss{\learnrule_k(D_{n_t})} \Bigr]$, the right hand side of equation \eqref{or_ineq_hub} contains two remainder terms. Since $K \leq n_t$, the second of these terms is always negligible with respect to the first as $n_v,n_t \to + \infty$ for fixed $L,c$.  Assuming that $n_v, n_t$ are both of order $n$, the first remainder term is $\mathcal{O}(\frac{\log n}{n})$ with respect to $n$. In comparison, the minimax risk for prediction in the model $Y = \langle \theta_*, X \rangle + \varepsilon, \Norm{\theta_*}_0 \leq k_*, \varepsilon \sim \mathcal{N}(0,1)$ is greater than a constant times $ \frac{k_*}{n}$ by proposition \ref{prop_minimax_huber_reg}. 
Thus, if more than $\log n$ independent components of $X$ are required for prediction of $Y$, the remainder term can be expected to be negligible compared to the oracle as a function of $n$.

As a function of a scale parameter $\sigma$ in a model $Y = s(X) + \sigma \varepsilon$, where $\varepsilon$ is distributed symmetrically around $0$, the remainder term scales as $\frac{c^2}{\eta}$, where $\eta$ depends only on $\sigma$ and on the fixed distribution of $\varepsilon$. When $\frac{\sigma}{c}$ is lower bounded and if $\varepsilon$ is sufficiently regular, then $\frac{c^2}{\eta} = \mathcal{O}(c\sigma)$ (see the discussion of hypothesis \textbf{(Lcs)}). In that case, the rate $c\sigma$ is the same as in the minimax lower bounds of Proposition \ref{agcv_hub}, and can therefore be considered correct. When $\frac{\sigma}{c} \to 0$, $\frac{c^2}{\eta} \sim c^2$ is suboptimal for Gaussian distributions $\sigma \varepsilon$, where the correct scaling is $\sigma^2$ (by Proposition \ref{prop_minimax_huber_reg} and a simple comparison with least squares). However, Theorem \ref{agcv_hub} makes \emph{no} moment assumptions whatsoever on the residual $Y - s(X)$ - thus, it is logical that the parameter $c$, which controls the robustness of the Huber loss, should appear in the bound.  

In equation \eqref{or_ineq_hub}, there is a tradeoff between the oracle and the remainder terms, governed by the tuning parameter $\theta \in (0;1]$. $\theta$ must be larger than a positive constant depending on $\alpha, \nu_0$ and $\eta$; as a result, Theorem \ref{agcv_hub} only yields a nontrivial result when $\nu_0 < \frac{\eta}{\mu_2 \sqrt{\alpha + 3}}$. Note that hypothesis \textbf{(Ni)}, which defines $\nu_0$, allows $\nu_0$ to decrease with $n$ as fast as $\sqrt{\frac{\log n}{n}}$, in case $\kappa(K)$ is a constant - as when $X$ is gaussian (see section \ref{sec.gauss} below). Assuming only that $\nu_0 = \nu_0(n) \to 0$ and that the remainder term is negligible compared to the oracle, equation \eqref{or_ineq_hub} proves that $\mathbb{E}\Bigl[\loss{\ERMag{\cT}} \Bigr] \sim \mathbb{E}\Bigl[\min_{1 \leq k \leq K} \loss{\learnrule_k(D_{n_t})} \Bigr] $ by taking $\theta = \theta_n \to 0$ slowly enough - an "optimal" oracle inequality.

\subsection{Gaussian design} \label{sec.gauss}
In the case where $X \in \mathbb{R}^d$ is a Gaussian vector, 
$\langle \theta, X - E X \rangle$ follows a centered normal distribution.
As a result, $\kappa(K)$ - defined in equation \eqref{def_kappa} -  is a fixed numerical constant, equal to $\max(\Norm{Z}_{L^{\psi_1}}, \frac{1}{\log 2})$, where $Z \sim \mathcal{N}(0,1)$. It follows that for any fixed $\nu_0$, hypothesis \textbf{(Ni)} holds as soon as $\frac{n_v}{\log(n_t \vee K)}$ is large enough.

Moreover, for Gaussian design, it is possible to show that the Lasso estimators of Definition \ref{def_param_lasso} satisfy hypothesis \textbf{(Uub)} \emph{for any} $R \geq 0$ (including $R = +\infty$), as long as $Y$ has some moments and $K$ isn't too large. More precisely, hypothesis \textbf{(Uub)} holds with $L,\alpha$ independent from $R$. This leads to the following corollary.

\begin{corollary} \label{cor_gauss}
Assume that $X \in \mathbb{R}^{d}$ is a Gaussian vector, that for some $u \in (0;1]$, $Y \in L^{1+u}$ and that hypothesis \textbf{(Lcs)} holds. 
Let $R \in \mathbb{R} \cup \{ + \infty \}$ and let $\ERMag{\cT}$ be the Agghoo estimator built from the collection $\left(\learnrule_{k,R}^{lasso} \right)_{1 \leq k \leq K}$.
Assume that $n_t \geq 13 + \frac{6}{u}$ and

\begin{equation} \label{eq_bound_K}
3 \leq K \leq \min \left( \frac{n_t}{\log n_t}, \frac{n_t}{\log d}, \frac{2(n_t-1)}{5} \right).
\end{equation} 

There exist numerical constants $\mu_5, \mu_8$ such that for all $\theta \in \left[ \frac{\mu_5}{\eta} \sqrt{\frac{\log n_t}{n_v}} ; 1 \right]$ and all $q \in \mathbb{R}$,
\begin{align*}
(1 - \theta) \mathbb{E} \left[ \loss{\ERMag{\cT}} \right] &\leq 
 (1 + \theta) \mathbb{E} \left[ \min_{1 \leq k \leq K} \loss{\learnrule_{k,R}^{lasso}(D_{n_t})} \right] + 243 \frac{c^2 \log n_t}{\theta \eta n_v} \\
 &\quad + \left( c \vee \Norm{Y_1 - q}_{L^{1+u}} \right) \frac{\mu_8 c}{\theta n_t \sqrt{n_v}}.
\end{align*}
\end{corollary}

Corollary \ref{cor_gauss} allows to take $\theta \to 0$
at any rate slower than $\sqrt{\frac{\log n_t}{n_v}}$, 
so that the asymptotic constant in front of the oracle 
is $1$. 
The constraint \eqref{eq_bound_K} imposed on $K$ by Corollary \ref{cor_gauss} is mild, since there are strong practical and theoretical reasons to take $k$ much smaller than $\frac{n_t}{\log n_t}$ anyway: this enforces sparsity -- minimizing computational complexity and improving interpretability -- and allows better control of the minimax risk (Proposition \ref{prop_minimax_huber_reg}). Equation \eqref{eq_bound_K} serves only to prove that $\hat{\theta}_{k,R}^{lasso}$ satisfies hypothesis \textbf{(Uub)}, hence it could be replaced by a polynomial bound on $R$ and on $ X - E X$, as explained in the discussion of hypothesis \textbf{(Uub)}. 

\subsection{Nonparametric bases} \label{sec.nonparam}
Given real random variables $U \in [a,b], Y \in \mathbb{R}$, a linear model may be a poor approximation to the actual regression function $s_0(U)$. A popular technique to obtain a more flexible model is to replace the one-dimensional variable $U$ with a vector $X = \psi_j(U)_{1 \leq j \leq d_n}$, where $(\psi_j)_{1 \leq j \leq d_n}$ spans a space of functions $W_{d_n}$ known for its good approximation properties, such as trigonometric polynomials, wavelets or splines (\cite[Chapter 5]{Hastie2009}). $d_n$ is practically always 
allowed to tend to $+ \infty$ as $n$ grows to make sure that the approximation error of $s$ by functions in $W_{d_n} = \langle (\psi_j)_{1 \leq j \leq d_n} \rangle $ converges to $0$.  In this section, we discuss conditions under which Theorem \ref{agcv_hub} applies to such models.

It turns out that most of the classical function spaces satisfy an equation of the form
\[ \forall f \in W_{d_n}, \NormInfinity{f} \leq \mu(a,b) \sqrt{d_n} \Norm{f}_{L^2([a,b])},  \]
where $\mu(a,b)$ is some constant independent of $d_n$ \cite[Section 3.1]{birge1998}. By replacing  $\psi_j(x)$, defined on $[a;b]$, by $\psi_j(\tfrac{x - a}{b - a})$ defined on $[0;1]$, we can see that the correct scaling with respect to $a,b$ is $\mu(a,b) = \frac{\mu(0,1)}{\sqrt{b - a}} $. 
Thus, if the distribution of $U$ dominates the uniform measure on $[a,b]$, in the sense that for some $p_0 > 0$ and any measurable $A \subset [a,b]$, $P(U \in A) \geq \frac{p_0}{b-a} \int_A dx$, then
\[ \forall f \in W_{d_n}, \Norm{f(U)}_{L^\infty} \leq \frac{\mu(0,1)}{\sqrt{p_0}} \sqrt{d_n} \Norm{f(U)}_{L^2}.  \]
In particular, if $W_{d_n}$ contains the constant functions - which is the case with splines, wavelets and trigonometric polynomials - then equation \eqref{def_kappa} holds with $\kappa(d_n)$ of
order $\sqrt{d_n}$. Thus, equation \eqref{hyp_bound_K} of hypothesis \textbf{(Ni)} holds under the assumption that $d_n \leq \mu \nu_0 \frac{n_v}{\log n_t}$ for some constant $\mu$. Assuming that $n_v$ and $n_t$ are both of order $n$ (for example, a $V-$fold split with fixed $V$), this assumption is mild: as a consequence of \cite[Theorem 11.3]{Gyrfi2002} and approximation-theoretic properties of the spaces $W_{d_n}$ \cite{devore1993}, taking $d_n \leq \frac{n}{\log^2 n}$, for example, is sufficient to attain minimax convergence rates \cite{stone1982} \cite[Theorem 3.2]{Gyrfi2002} over standard classes of smooth functions. 

Note that even though $\kappa(d_n) \approx \sqrt{d_n}$, this does not in general imply that $\kappa(K) = \mathcal{O}(\sqrt{K})$: for example, in the case of regular histograms on $[0,1]$, $\psi_j = \sqrt{d_n} \mathbb{I}_{\left[ \frac{j}{d_n}, \frac{j+1}{d_n} \right]}$ so $\frac{\NormInfinity{\psi_j}}{\Norm{\psi_j}_{L^2}} = \sqrt{d_n}$ and when $U \sim \text{Unif}([0;1])$, $\kappa(1) \sim_{d_n \to + \infty} \sqrt{d_n}$. The property $\kappa(K) = \mathcal{O}(\sqrt{K})$ does, however, hold in the case of the Fourier basis: as a result, $d_n$ may be arbitrarily large, and only bounds on $K$ (the maximal zero-norm of the estimators) are required.
We examine this case in detail in the following section. 

\subsection{The Fourier basis} \label{sec.trig}
Suppose that real variables $(U,Y)$ are given, and that we wish to find
the best predictor of $Y$ among $1-$periodic functions of $U$.
Let $s_{per}$ denote the minimizer of the risk $E[\phi_c(Y - t(U))]$ among all measurable $1-$periodic functions on $\mathbb{R}$.
For all $k \in \mathbb{N}$, let $\psi_{2k}(x) = \sqrt{2} \cos(2\pi k x)$
and $\psi_{2k-1}(x) = \sqrt{2} \sin(2 \pi k x)$. Let $X = (\psi_j(U))_{1 \leq j \leq d}$, where $d \in \mathbb{N}$ and $d \geq 2$. 
One can easily show that $s_{per}(U) = s(X)$, where $s$ minimizes $P[\phi_c(Y - t(X))]$ among measurable functions $t$ on $\mathbb{R}^{d}$.
By taking $d$ large and using sparse methods, it is possible to approximate functions $s_{per}$ which have only a small number of  non-zero Fourier coefficients, but potentially at high frequencies,
as is commonly the case in practice \cite{sparse_fourier}. 

Let $(\hat{q}_k, \hat{\theta}_k)_{1 \leq k \leq K}$ be a collection of sparse linear regression estimators satisfying hypothesis \ref{hyp_sparse_reg} and let $\hat{t}_k$ denote the predictor $\hat{t}_k: x \mapsto \hat{q}_k(D_{n_t}) + \langle \hat{\theta}_k(D_{n_t}), x  \rangle$. 
Given this initial collection of linear predictors, Definition \ref{def_modif_estim_trigo} below constructs a second collection $(\tilde{q}_k, \tilde{\theta}_k)_{1 \leq k \leq K}$ which also satisfies hypothesis \textbf{(Uub)} under an appropriate distributional assumption (Corollary \ref{cor_trigo}, equation \eqref{eq_lb_density_X}).

\begin{Definition} \label{def_modif_estim_trigo}
Let $(\tilde{q}_k, \tilde{\theta}_k)_{1 \leq k \leq K}$ be defined by
\begin{equation} \label{eq_trunc_estim}
(\tilde{q}_k, \tilde{\theta}_k) =
\begin{cases}
& (\hat{q}_k, \hat{\theta}_k) \text{ if } \Norm{\hat{\theta}_k}_{\ell^2} \leq n_t^{\frac{3}{2}} \\
& (\tilde{q},0) \text{ otherwise},
\end{cases} 
\end{equation}
where 
\begin{align*}
  \tilde{q}(D_{n_t}) &\in \argmin_{q \in \hat{Q}(D_{n_t})} |q| \\
  \hat{Q}(D_{n_t}) &= \argmin_{q \in \mathbb{R}} \sum_{i = 1}^{n_t} \phi_c(Y_i - q).
\end{align*}
For any $k$, 
let $\tilde{t}_k: x \mapsto \tilde{q}_k(D_{n_t}) + \langle \tilde{\theta}_k(D_{n_t}), x \rangle$. 
\end{Definition}
By construction, $(\tilde{q}_k, \tilde{\theta}_k)$ also satisfies hypothesis \ref{hyp_sparse_reg}.
Replacing $(\hat{q}_k, \hat{\theta}_k)$ by $(\tilde{q}_k, \tilde{\theta}_k)$ may improve performance and cannot significantly degrade it, as proposition \ref{prop_trunc_improves} below makes clear.

\begin{proposition} \label{prop_trunc_improves}
Assume that $Y \in L^\alpha$ for some $\alpha \in (0,1]$ and 
let $q_* \in \mathbb{R}$. If
\begin{equation} \label{hyp_lb_nt_cor_trigo}
n_t \geq \max \left(\frac{16}{\alpha}, \frac{4}{\eta^2}, c  + 10 \Norm{s(X) - q_*}_{L^1} \right),
\end{equation}
for some numerical constant $\mu_{10} \geq 0$,
\begin{equation} \label{eq_ub_tilde_cor_trigo}
 \mathbb{E} \left[ \min_{1 \leq k \leq K} \loss{\tilde{t}_k} \right] \leq \mathbb{E} \left[ \min_{1 \leq k \leq K} \loss{\hat{t}_k} \right] + \frac{\mu_{10} c}{n_t^3} \left(c \vee 2^{\frac{2}{\alpha}} \Norm{Y - q_*}_{L^\alpha} \right)^4.
\end{equation}
\end{proposition}

Theorem \ref{agcv_hub} can be applied to the collection $(\tilde{q}_k, \tilde{\theta}_k)_{1 \leq k \leq K}$,  which yields the following Corollary.

\begin{corollary} \label{cor_trigo}
Assume that $U$ has a density $p_U$ such that
\begin{equation} \label{eq_lb_density_X}
\inf_{t \in [0;1)} \sum_{j \in \mathbb{Z}} p_U(t + j) \geq p_0 > 0.
\end{equation}
Assume that there exists $\eta > 0$ such that almost surely,
\[ \mathbb{P} \left( |Y - s_{per}(U)| \leq \frac{c}{2} \right) \geq \eta. \]
There exists a constant $\mu_9 \geq \sqrt{8}$ such that, if
\begin{equation} \label{hyp_ub_K_cor_trigo}
K \leq p_0 \left(\frac{\theta \eta}{\mu_9} \right)^2 \frac{n_v}{\log^3 n_t}
\end{equation}
for some $\theta \in (0;1]$, then
\begin{equation} \label{or_ineq_cor_trigo}
(1-\theta) \mathbb{E}\Bigl[\loss{\ERMag{\cT}} \Bigr] \leq  (1+\theta)\mathbb{E}\Bigl[\min_{1 \leq k \leq K} \loss{\tilde{t}_k} \Bigr] 
+ 270 \frac{c^2 \log n_t}{\theta \eta n_v} 
+ \frac{5 \mu_1 c K \log K}{\theta n_t^{2} \sqrt{n_v}}.
\end{equation}
\end{corollary}


If the $1-$periodicity of $s_{per}$ represents (say) a yearly cycle, then Equation \eqref{eq_lb_density_X} states that each "time of year" $u \in [0;1]$ is sampled with a positive density, i.e that the density of $U - \lfloor U \rfloor$ is lower bounded by a positive constant $p_0$ on $[0;1]$. This ensures that equation \eqref{def_kappa}
holds with $\kappa(K)$ of order $\sqrt{\frac{K}{p_0}}$, so that hypothesis \eqref{Hw_huber} reduces to $K \leq p_0 \left(\frac{\theta \eta}{\mu_9} \right)^2 \frac{n_v}{\log n_t}$.    In particular, if $\theta$ is constant and $n_v$ is of order $n$, then $K$ is allowed to grow with $n$ at rate $\frac{n}{\log n}$. This is a reasonable restriction, as by Proposition \ref{prop_minimax_huber_reg}, one cannot expect to estimate more than $\frac{n}{\log n}$ coefficients with reasonable accuracy (a $\frac{1}{\log n}$ convergence rate being too slow for most practical purposes).

Corollary then deduces an oracle inequality with leading constant $\frac{1 + \theta}{1 - \theta}$ (arbitrarily close to $1$) and remainder term of order $\frac{c^2 \log n}{\eta n}$,
which is typically negligible in the non-parametric setting of this corollary. For this reason, corollary \ref{cor_trigo} can be said to be optimal, at least up to constants.

\subsection{Effect of V}

The upper bound given by Theorem \ref{agcv_hub} only depends on $\mathcal{\cT}$ through $n_v$ and $n_t$. The purpose of this section is to show that for a given value of $n_v$, increasing $V = |\cT|$ always decreases the risk. This is proved in the case of monte carlo subset generation defined below.

\begin{Definition} \label{def_mcbag}
 For $\tau \in \left[ \frac{1}{n};1 \right]$ and $V \in \mathbb{N}^*$, let $\mathcal{T}^{mc}_{\tau,V}$ be generated independently of the data $D_n$ by drawing $V$ elements independently and uniformly in the set \[ \left\{ T \subset [|1;n|] : |T| = \lfloor \tau n \rfloor \right\}. \] 
\end{Definition}

For fixed $\tau$, the excess risk of Agghoo is a non-increasing function of $V$.

\begin{proposition} \label{prop_ag_improves}
 Let $U \leq V$ be two non-zero integers. Let $\tau \in \left[ \frac{1}{n};1 \right]$. Then:
 \[ \mathbb{E} \left[ \loss{\ERMag{\cT^{mc}_{\tau,V}}} \right] \leq \mathbb{E} \left[ \loss{\ERMag{\cT^{mc}_{\tau,U}}} \right]. \]
\end{proposition}

\begin{proof}
Let $ (T_i)_{i = 1,\ldots,V} = \cT^{mc}_{\tau,V}$. 
 Let $\mathcal{I} = \left\{ I \subset [|1;V|] : |I| = U \right\}$. Then 
 \begin{align*}
  \ERMag{\cT^{mc}_{\tau,V}} &= \sum_{i = 1}^V \frac{1}{V} \ERMho{T_i} \\
  &= \sum_{i = 1}^V \frac{{V-1 \choose U-1}}{U  {V \choose U}} \ERMho{T_i} \\
  &= \frac{1}{U} \sum_{i = 1}^V \frac{\sum_{I \in \mathcal{I}} \mathbb{I}_{i \in I}}{|\mathcal{I}|} \ERMho{T_i} \\
  &= \frac{1}{|\mathcal{I}|} \sum_{I \in \mathcal{I}} \frac{1}{U} \sum_{i \in I} \ERMho{T_i}.
 \end{align*}
It follows by convexity of $f \mapsto \loss{f}$ that
\[ \mathbb{E} \left[ \loss{\ERMag{\cT^{mc}_{\tau,V}}} \right] \leq \frac{1}{|\mathcal{I}|} \sum_{I \in \mathcal{I}} \mathbb{E} \left[ \loss{\frac{1}{U} \sum_{i \in I} \ERMho{T_i} } \right]. \]
For any $I \in \mathcal{I}$, $(T_i)_{i \in I} \sim \mathcal{T}^{mc}_{\tau,U}$ and is independent of $D_n$, therefore $ \frac{1}{U} \sum_{i \in I} \ERMho{T_i} \sim \ERMag{\mathcal{T}^{mc}_{\tau,U}}$. This yields the result.
\end{proof}

It can be seen from the proof that the proposition also holds for Agcv. Thus, increasing $V$ can only improve the performance of these methods. The same argument does not apply to CV, because CV takes an argmin after averaging, and the argmin operation is neither linear nor convex. Indeed, no comparable theoretical guarantee has been proven for CV, to the best of our knowledge, even though increasing the number of CV splits (for given $\tau$) generally improves performance in practice. 

Proposition \ref{prop_ag_improves} does not quantify the gain due to aggregation. This gain depends on the properties of the convex functional $t \mapsto \loss{t}$, in particular on its modulus of strong convexity in a neighbourhood of the target $s$ (assuming that at least some estimators in the collection are close to $s$). Moreover, as for any loss function, the gain due to aggregation depends on the diversity of the collection $(\ERMho{T_i})_{1 \leq i \leq V}$: the more the hold-out estimators $\ERMho{T}$ vary with respect to $T$, the greater the effect of aggregation.

More precisely, under hypothesis \textbf{(Lcs)}, we can prove the following improvement to Proposition \ref{prop_ag_improves}.

\begin{proposition} \label{prop_ag_quant_improve}
Let $(X,Y) \sim P$ be independent from $D_n$. Assume that $P$ satisfies hypothesis \textbf{(Lcs)}. For any $i \in \{1,\ldots,V\}$, let $E_i(c)$ denote the event $|(\ERMho{T_i} - s)(X)| \leq \frac{c}{2}|$.
Then for any $V \in \mathbb{N}$,
\begin{equation} \label{eq_comp_agghoo_ho}
    \mathbb{E} \left[ \loss{\ERMag{\cT^{mc}_{\tau,V}}} \right] \leq 
    \mathbb{E} \left[ \loss{\ERMho{T_1}} \right] - \eta \frac{V-1}{4V} \mathbb{E} \left[ \left(\ERMho{T_1} - \ERMho{T_2} \right)^2(X) \mathbb{I}_{E_1(c)} 
    \mathbb{I}_{E_2(c)} \right].
\end{equation}
When $\loss{\ERMho{T_1}}$ is small enough, the event $E_1(c)$ occurs with high probability. As a consequence, 
if $\mathbb{E} \left[ \loss{\ERMho{T_1}} \right] \leq \frac{\eta c^2}{64}$, then
\begin{equation} \label{eq_comp_agghoo_ho_med}
    \mathbb{E} \left[ \loss{\ERMag{\cT^{mc}_{\tau,V}}} \right] \leq \mathbb{E} \left[ \loss{\ERMho{T_1}} \right] - \eta \frac{V-1}{16 V} \text{Med} \left[ \left(\ERMho{T_1} - \ERMho{T_2} \right)^2(X) \right],
\end{equation}
where $\text{Med}[Y]$ denotes the largest median of a random variable $Y$.
\end{proposition}
Proposition \ref{prop_ag_quant_improve} is proved in appendix \ref{proof_prop_ag_quant_improve}. It quantifies the gain due to aggregation in terms of the parameter $c$ of the Huber loss, the constant $\eta$ given by hypothesis \textbf{(Lcs)} and the distance between two hold-out estimators that are close enough to $s$. Taking $c \to + \infty$ recovers the least-squares case, where $\eta = 1$ and there are no constraints on $\ERMho{T_i} - s$. 
Only two indices $1,2$ appear in the right-hand side of equation \eqref{eq_comp_agghoo_ho}: that is a consequence of the exchangeability of the collection $(\ERMho{T_i})_{1 \leq i \leq V}$ for Monte-Carlo subset generation. The same result also applies to $V-$fold Agghoo, since it also yields an exchangeable collection. For arbitrary $\cT$, all distinct pairs of indices would have to be considered. 

Going beyond proposition \ref{prop_ag_quant_improve} requires giving nontrivial lower bounds on $\left(\ERMho{T_1} - \ERMho{T_2} \right)^2(X)$, which is no easy task, given the complex dependencies involved. 
Results in this direction have only recently been obtained in the setting of least-squares density estimation \cite[Chapters 5-6]{these}.

A few general heuristics apply: first, if there is one learning rule $\learnrule_{k_*}$ in the collection which is much better than the others, the hold-out can be expected to select it most of the time: in that case, Agghoo reduces to bagging, and potential gains depend on the stability of $\learnrule_{k_*}$. In contrast, if there are many rules $\learnrule_k$ which are close to optimal, while being distant from each other, then the gains of aggregation can be expected to be large, even if the individual rules $\learnrule_k$ are stable.

\section{Simulation study}
This section focuses on hyperparameter selection for the Lasso with huber loss, either using a fixed grid or using the reparametrization from Definition \ref{def_param_lasso}. The methods considered for this task are Aggregated hold-out given by Definition \ref{def_agghoo}, Aggregated cross-validation given by Definition \ref{def_acv} and standard cross-validation given by Definition \ref{def_cv}. In all cases, the subsamples are generated independently from the data and uniformly among subsets of a given size $\tau n$, as in Definition \ref{def_mcbag}. Thus, all three methods share the same two hyperparameters: $\tau$, the fraction of data used for training the Lasso, and $V$, the number of subsets used by the method.

For the huberized Lasso with a fixed grid, the hqreg\_raw function from the R package \href{https://cran.r-project.org/web/packages/hqreg/index.html}{hqreg} \cite{hqreg:2017} is used with a fixed grid designed to emulate the default choice: a geometrically decreasing sequence of length $100$, with maximum value $\lambda_{max}$ and minimum value $\lambda_{min} = 0.05 \lambda_{max}$. The fixed value of $\lambda_{max}$ is obtained by averaging the (data-dependent) default value chosen by hqreg\_raw over $10$ independent datasets. To compute the reparametrization given by Definition \ref{def_param_lasso}, we implemented the LARS-based algorithm described by Rosset and Zhu \cite{Rosset-Zhu2007}, which allows to compute the whole regularization path. 

I.i.d training samples of size $n = 100$ are generated according to a distribution $(X,Y)$, where $X \in \mathbb{R}^{1000}$ and $Y = w_*^T X + \varepsilon$, with $\varepsilon$ independent from $X$. To illustrate the robustness of the estimators, Cauchy noise is used: $\varepsilon \sim \text{Cauchy}(0,\sigma)$.
The performance of Agghoo and cross-validation may depend on the presence of correlations between the covariates $X$ and the sparsity of the ground truth $w_*$. To investigate these effects, three parametric families of distribution are considered for $X$, in sections \ref{exp_setup1}, \ref{exp_setup2} and \ref{exp_setup3}.

The risk of each method is evaluated on an independent training set of size $500$, and results are averaged over $1000$ repetitions of the simulation. More precisely, $1000$ training sets $D_j$ of size $n = 100$ are generated, along with $1000$ test sets $(X_{i,j}', Y_{i,j}')_{1 \leq i \leq 500}$, each of size $500$.
For each simulation $j$ and any learning rule $\learnrule_{\tau,V}$ among the six obtained by combining Agghoo, monte carlo CV and AGCV with either a fixed grid or the zero-norm parametrization, the average excess risk \[\hat{R}_j(\learnrule, \tau, V) =  \frac{1}{500} \sum_{i = 1}^{500} \left[\phi_c \left( Y_{i,j}' - \learnrule_{\tau,V}(D_j)(X_{i,j}') \right) - \phi_c \left( Y_{i,j}'  - \bayes(X_{i,j}') \right) \right]\] is computed on the test set 
for all values of $V \in \left\{1,2,5,10 \right\}$ and $\tau \in \left\{ \frac{i}{10} : 1 \leq i \leq 9 \right\}$.

\subsection{Experimental setup 1} \label{exp_setup1}

$X$ is generated using the formula $ X_i = \frac{1}{\Norm{u}_2} \sum_{j = 1}^d u_{i-j} Z_j$,
where $Z_j$ are independent standard Gaussian random variables, $u_i = \mathbb{I}_{|i| \leq cor} e^{- \frac{2.33^2 i^2}{2cor^2}}$ and $cor \in \mathbb{N}$ is a parameter regulating the strength of the correlations. The regression coefficient has a support of size $r = 3*k$ drawn at random from $[|1;1000|]$, and is defined by $w_{*,j} = u_{*,g(j)}$, where $g$ is a uniform random permutation, $u_{*,j} = b$ if $1 \leq j \leq k$ and $u_{*,j} = \frac{b}{4}$ if $2k+1 \leq j \leq 3k$, with $b$ calibrated so that $\Norm{X w_*}_{L^2} = 1$. The noise parameter is $\sigma = 0.08$, while the Huber loss parameter $c$ is set to $2$ -- a suboptimal choice in this setting, but convenient for computing the huberized Lasso regularization path.   

\paragraph{Choice of $\tau$ parameter}
For all methods, in most cases the optimal value of $\tau$ is $0.8$ or $0.9$, similarly to what was observed in the rkhs case \cite{agghoo_rkhs}, where $\tau = 0.8$ was recommended. Table 1 displays the quantity
\[ \hat{G}(\learnrule, \tau, V) = \frac{\text{Mean}\left[(\hat{R}_j(\learnrule,\tau,V) - \hat{R}_j(\learnrule,\tau_*,V))_{1 \leq j \leq 1000} \right]}{\text{Sd}\left[ (\hat{R}_j(\learnrule,\tau,V) - \hat{R}_j(\learnrule,\tau_*,V))_{1 \leq j \leq 1000} \right]}, \]
where Sd denotes the (empirical) standard deviation and $\tau_*$ the optimal choice of $\tau$, $\tau_* = \argmin_{\tau \in \{0.1,\ldots,0.9\}} \text{Mean}\left[(\hat{R}_j(\learnrule,\tau,V))_{1 \leq j \leq 1000} \right].$
Thus, values of $\hat{G}(\learnrule, \tau, V)$ bigger than a few units suggest that $\tau$ is suboptimal to a statistically significant degree. 
When $\tau_* = 0.9$, $\hat{G}(\learnrule, 0.8, V)$ is displayed in black on table 1. When $\tau_* = 0.8$, $\hat{G}(\learnrule, 0.9, V)$ is displayed in blue on table 1. Exceptions where $\tau_* \notin \{ 0.8,0.9 \}$ are highlighted in red, with the value $\min \left( \hat{G}(\learnrule, 0.8, V), \hat{G}(\learnrule, 0.9, V)  \right)$.  

\begin{table}[ht]
\centering
\begin{tabular}{cl|c|cc|cc|cc|}
  \hline
  & & & \multicolumn{2}{c|}{r = 150} 
 & \multicolumn{2}{c|}{r = 60} & \multicolumn{2}{c|}{r = 24} \\
 \hline
 & method & V & 15 & 1 & 15 & 1 & 15 & 1 \\ 
  \hline
1 & grid agghoo & 1 & 2.2 & \textcolor{red}{2.7} & 3.0 & \textcolor{blue}{2.7} & 0.5 & 5.6 \\ 
  2 & grid agghoo & 2 & 2.5 & \textcolor{red}{2.1} & 3.1 & \textcolor{blue}{1.4} & 1.0 & 7.9 \\ 
  3 & grid agghoo & 5 & 2.5 & \textcolor{blue}{6.8} & 3.5 & 0.6 & 0.6 & 11.9 \\ 
  4 & grid agghoo & 10 & 0.7 & \textcolor{blue}{7.2} & 3.7 & 1.1 & 4.5 & 16.7 \\ 
  \hline
  5 & grid cv & 1 & 1.0 & \textcolor{red}{3.9} & \textcolor{blue}{1.6} & \textcolor{red}{0.1} & \textcolor{blue}{1.2} & 1.5 \\ 
  6 & grid cv & 2 & 0.8 & \textcolor{red}{5.0} & \textcolor{blue}{2.6} & \textcolor{red}{0.5} & \textcolor{blue}{1.4} & 1.1 \\ 
  7 & grid cv & 5 & 1.4 & \textcolor{red}{2.8} & 1.5 & \textcolor{blue}{0.8} & \textcolor{blue}{0.5} & 3.7 \\ 
  8 & grid cv & 10 & 2.0 & \textcolor{red}{2.6} & 2.9 & \textcolor{blue}{1.1} & 1.6 & 5.9 \\ 
  \hline
  9 & grid agcv & 1 & 1.0 & \textcolor{red}{3.9} & \textcolor{blue}{1.6} & \textcolor{red}{0.1} & \textcolor{blue}{1.2} & 1.5 \\ 
  10 & grid agcv & 2 & \textcolor{blue}{0.3} & \textcolor{red}{2.0} & \textcolor{blue}{1.4} & \textcolor{blue}{1.9} & \textcolor{blue}{0.3} & 0.8 \\ 
  11 & grid agcv & 5 & 0.3 & \textcolor{blue}{2.2} & \textcolor{blue}{0.5} & \textcolor{blue}{0.7} & \textcolor{blue}{0.5} & 1.1 \\ 
  12 & grid agcv & 10 & \textcolor{blue}{0.5} & \textcolor{blue}{0.4} & 0.0 & \textcolor{blue}{0.3} & \textcolor{blue}{0.8} & 1.0 \\ 
  \hline
  13 & $0-$norm agghoo & 1 & 1.3 & \textcolor{red}{4.1} & 2.0 & 0.3 & 0.5 & 5.6 \\ 
  14 & $0-$norm agghoo & 2 & 3.0 & \textcolor{red}{1.4} & 3.2 & 1.3 & 1.9 & 9.2 \\ 
  15 & $0-$norm agghoo & 5 & 4.0 & \textcolor{blue}{6.7} & 5.1 & 3.3 & 4.0 & 13.7 \\ 
  16 & $0-$norm agghoo & 10 & 4.6 & \textcolor{blue}{7.3} & 7.0 & 3.7 & 5.2 & 18.5 \\ 
  \hline
  17 & $0-$norm cv & 1 & \textcolor{blue}{4.3} & \textcolor{red}{9.4} & \textcolor{blue}{4.3} & \textcolor{red}{1.1} & \textcolor{red}{2.0} & \textcolor{blue}{3.9} \\ 
  18 & $0-$norm cv & 2 & \textcolor{blue}{1.9} & \textcolor{red}{7.2} & \textcolor{blue}{1.8} & \textcolor{blue}{4.4} & \textcolor{blue}{4.8} & \textcolor{blue}{2.7} \\ 
  19 & $0-$norm cv & 5 & 2.7 & \textcolor{red}{5.3} & 2.4 & \textcolor{blue}{3.3} & \textcolor{blue}{1.5} & \textcolor{blue}{0.7} \\ 
  20 & $0-$norm cv & 10 & 6.1 & \textcolor{red}{4.6} & 5.4 & \textcolor{blue}{3.5} & \textcolor{blue}{0.6} & \textcolor{blue}{0.1} \\ 
  \hline
  21 & $0-$norm agcv & 1 & \textcolor{blue}{4.3} & \textcolor{red}{9.4} & \textcolor{blue}{4.3} & \textcolor{red}{1.1} & \textcolor{red}{2.0} & \textcolor{blue}{3.9} \\ 
  22 & $0-$norm agcv & 2 & \textcolor{blue}{1.9} & \textcolor{red}{5.8} & \textcolor{blue}{2.4} & \textcolor{blue}{4.5} & \textcolor{blue}{5.9} & \textcolor{blue}{3.5} \\ 
  23 & $0-$norm agcv & 5 & 2.1 & \textcolor{red}{1.9} & 1.0 & \textcolor{blue}{4.0} & \textcolor{blue}{5.7} & \textcolor{blue}{3.7} \\ 
  24 & $0-$norm agcv & 10 & 4.5 & \textcolor{red}{1.0} & 3.3 & \textcolor{blue}{3.6} & \textcolor{blue}{7.3} & \textcolor{blue}{3.9} \\ 
   \hline
\end{tabular}
\caption{$\hat{G}(\learnrule, \tau, V)$ for sub-optimal $\tau \in \{0.8,0.9\}$ and various distributions. 
Colours show optimal $\tau_*$: blue for $\tau_* = 0.8$, black for $0.9$, red when $\tau_* \notin \{0.8,0.9\}$.}
\end{table}

Most of the exceptions $\tau_* \notin \{0.8,0.9\}$ occur on the column $r = 150$, $cor = 1$, while most of the others are of low statistical significance, with values less than $1.1$ on the fourth column ($r = 60$ and $cor = 1$).  
Thus, table 1 confirms the claim that $\tau_* \in \{0.8,0.9\}$ for all methods, in most cases. 
For grid agghoo, $0-$norm agghoo, grid agcv and $V \geq 5$, $\tau_* \in \{0.8,0.9\}$ for all simulations.
Comparing now $\tau = 0.8$ and $\tau = 0.9$, grid agghoo and $0-$norm agghoo with $V \geq 5$ show a clear pattern: $\tau = 0.9$ is better or as good as $\tau = 0.8$ in all cases except $r = 150, cor = 1$ where $\tau = 0.8$ is significantly better. For other methods, results are not so clear and the difference in risk between the two values of $\tau$ is often insignificant.

\paragraph{Choice of $V$}
For all methods considered, performance is expected to improve when $V$ is increased, but by how much? If the performance increase is too slight, it may not be worth the additional computational cost. In figure $1$, the mean excess risk for the optimal value of $\tau$ is displayed as a function of $V$, with error bars corresponding to one standard deviation. The scale used for the vertical axis in each graph is the average excess risk of the oracle with respect to the fixed grid over the $\lambda$ parameter. Quantifying performance as a percentage of the oracle risk, when $cor = 15$, Agghoo improves by roughly $20\%$ from $V = 1$ to $V = 2$, by roughly $10 \%$ from $V = 2$ to $V = 5$ and by a few percent more from $V = 5$ to $V = 10$. CV with the standard grid behaves similarly in these two simulations, while CV with the zero-norm parametrization shows much less improvement when $V$ is increased. Thus, taking $V \geq 5$ is advantageous, but there are clearly diminishing returns to choosing $V$ much larger than this. For CV with the zero-norm parametrization, $V = 2$ seems sufficient in these simulations .   

\begin{figure}
\includegraphics[width = 6cm, height = 7cm]{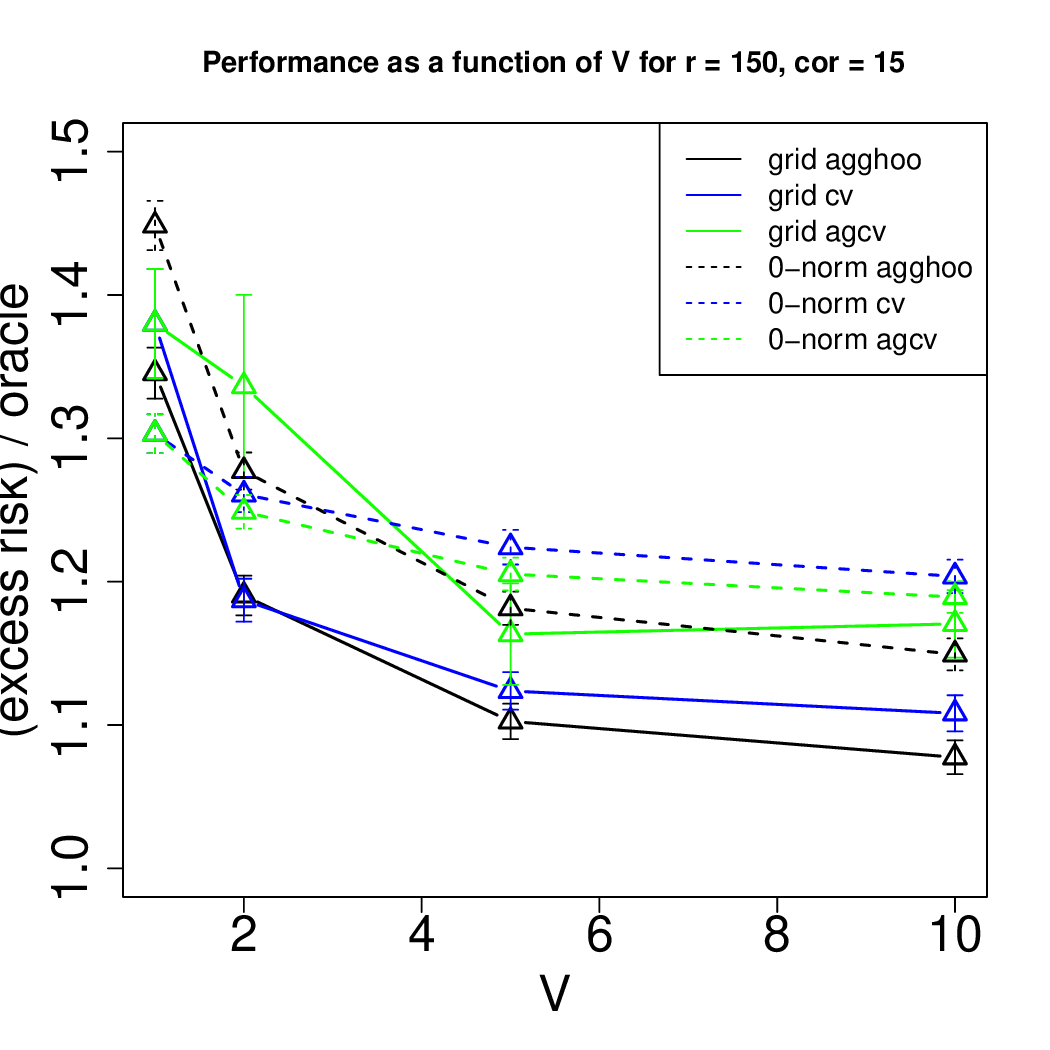}
\includegraphics[width = 6cm, height = 7cm]{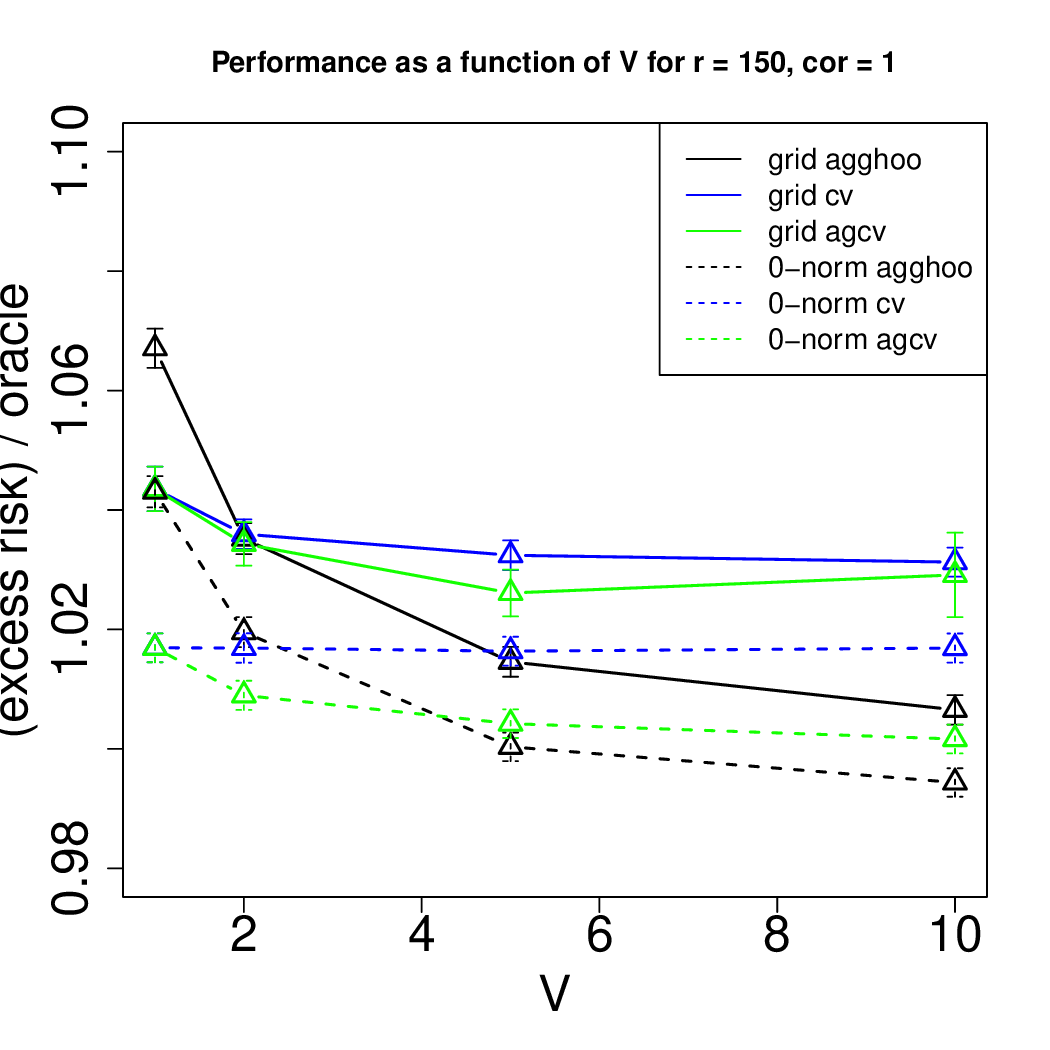}
\includegraphics[width = 6cm, height = 7cm]{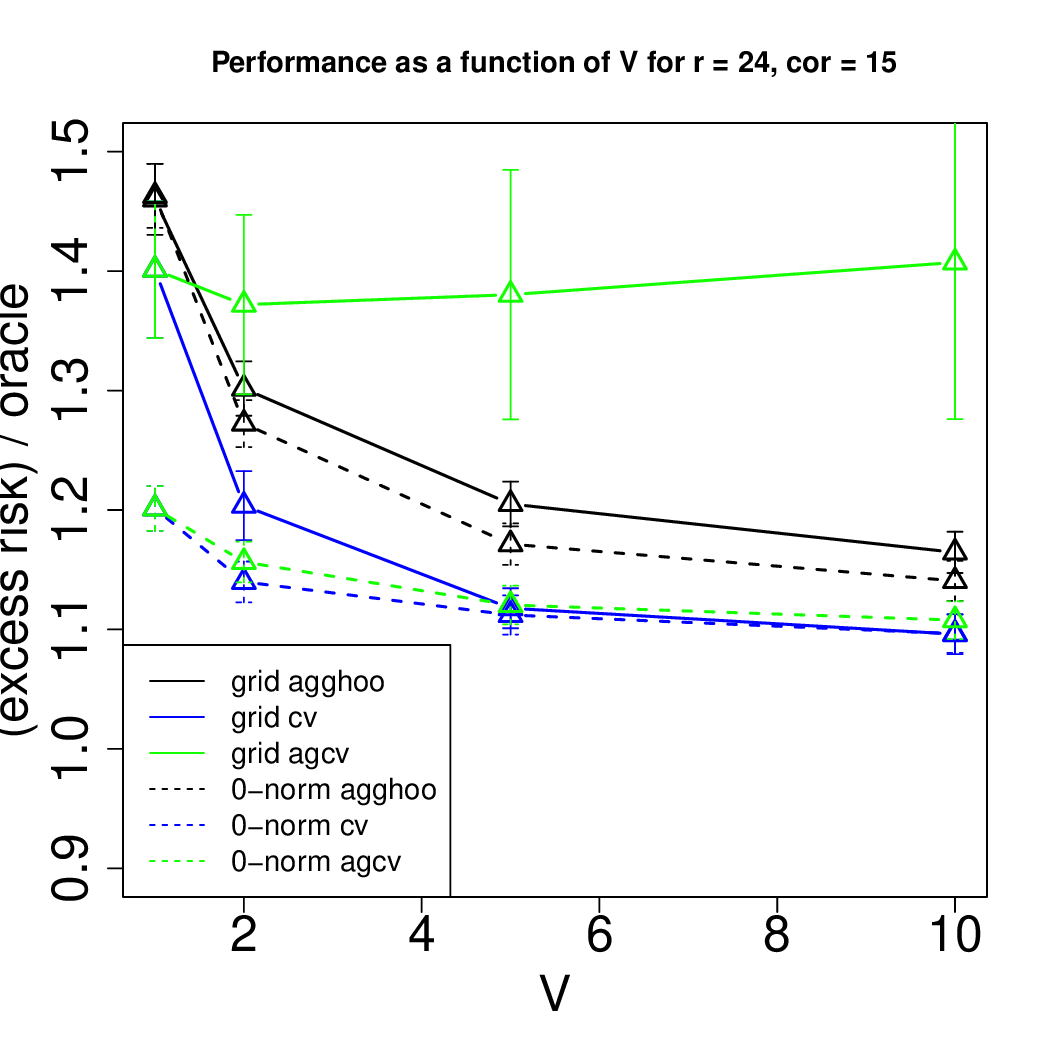}
\includegraphics[width = 6cm, height = 7cm]{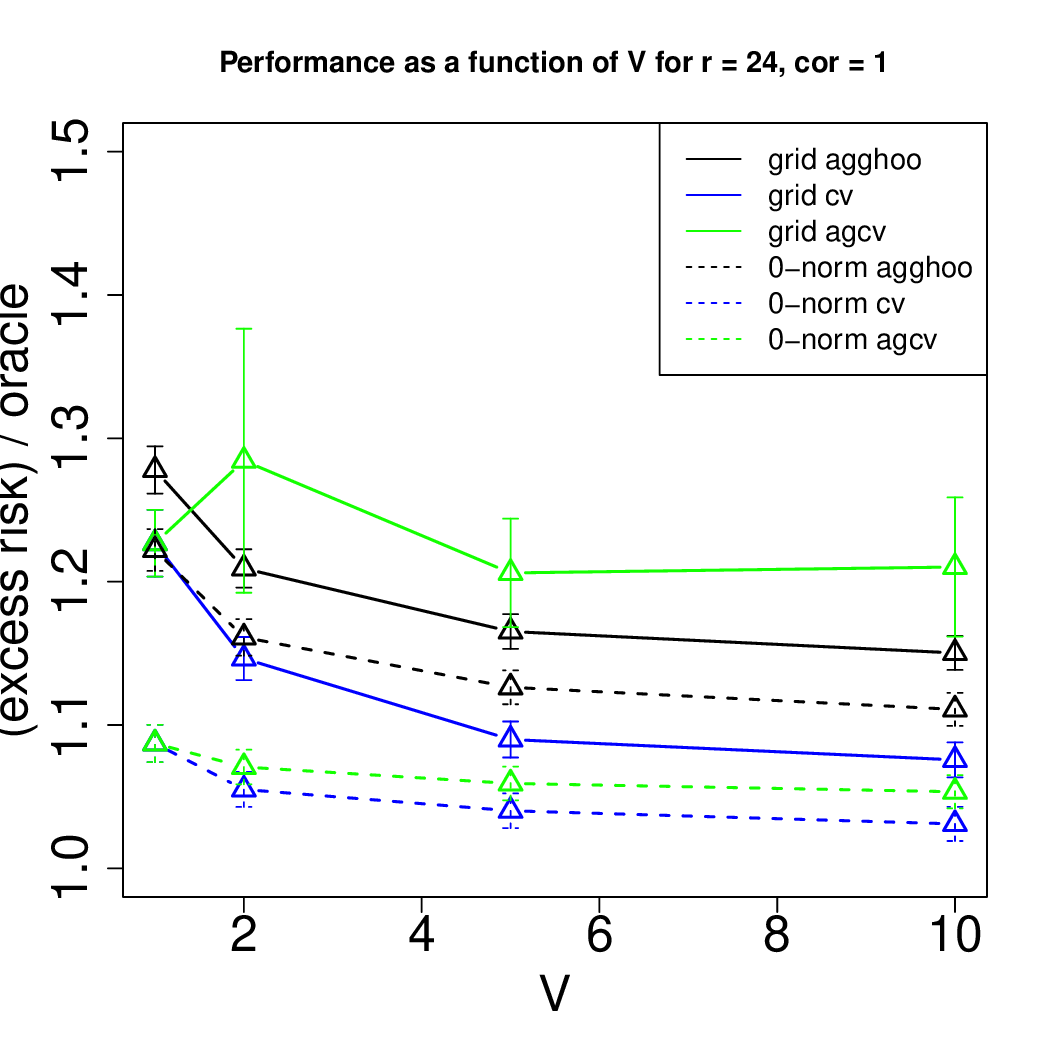}
\caption{Performance relative to the oracle, as a function of $V$}
\end{figure}

\paragraph{Comparison between methods}
From figure 1, it appears that grid agcv is a very poor choice, being worse than both grid agghoo and grid cv for all values of $V$ when $r = 150, cor = 15$ , and being the worst of all the methods for $V \geq 2$ when $r = 24$, as well as highly unstable, as the size of the error bars clearly shows.

Interestingly, $0-$norm agcv behaves much better, being the second best method when $cor = 1$, and very close to the best when $r = 24$ and $cor = 15$.

Generally speaking, of the two types of parametrization of the Lasso,
the zero-norm parametrization appears to perform better than the standard grid when correlations are small ($cor = 1$), while the performance is significantly worse when $r = 150$ and $cor = 15$.   

Comparing now Agghoo and CV, Agghoo appears to be better than CV when $V \geq 2$ in situations where $r$ is larger ($r = 150$). This seems to hold for both the standard parametrization (grid agghoo) and the zero-norm one ($0-$norm agghoo).  The relation is reversed for small $r$, with CV performing better than Agghoo for all values of $V$ when $r = 24$.

\paragraph{Further studies}

The previous simulations suggest that Agghoo performs better than CV in the case of high intrinsic dimension. This behaviour is logical, since the cross-validated Lasso will ignore some predictive variables when there are too many of them, and randomized aggregation may help recover more of the support. However, the effect of correlations is unclear. Experimental setup 1 mixes different types of correlations: correlations between predictive variables, correlations between predictive and non-predictive variables, and correlations among non-predictive variables. It is possible that one type of correlation favours Agghoo while another favours CV. 

To gain a more accurate idea of when Agghoo is advantageous over CV, two more settings are studied, considering separately correlations among predictive variables, and between predictive and non-predictive variables. Since previous simulations showed that $\tau = 0.8,0.9$ and $V = 10$ were the optimal parameters, only those parameters will be considered in the following. 

Since the choice of lasso parametrization did not seem to affect the relative performance of Agghoo and CV, we only consider the standard parametrization, as it is more popular and also easier to use in our simulations.
Agcv is not considered either, since it was discovered to be unreliable in previous simulations.

\subsection{Experimental setup 2: correlations between predictive and noise variables} \label{exp_setup2}
Let $r$ be the number of predictive variables and let each predictive covariate have $s$ "noise" covariates which are correlated with it at level $\rho = 0.8$. Assume that $rs \leq d$, where $d$ is the total number of variables.
Let $(Z^0_j)_{1 \leq j \leq r}$, $(Z_{j,i})_{1\leq j \leq r, 1 \leq i \leq s}$ and $(W_k)_{1 \leq k \leq d - rs}$ be independent standard gaussian variables.
For any $j \in [|0:r-1|]$ and any $i \in [|1;s|]$, let $X_{js + i} = \sqrt{0.8} Z^0_{j+1} + \sqrt{0.2} Z_{j+1,i}$ and for $rs < i \leq d$, let $X_i = W_{i - rs}$.
For the regression coefficient, choose $w_* = \frac{3*u}{\Norm{Xu}_{L_2}}$, where $u = (\mathbb{I}_{s | (j-1)} \mathbb{I}_{j \leq rs})_{1 \leq j \leq d}$. Let then $Y$ be distributed conditionnally on $X$ as $\text{Cauchy}(\langle w_*, X \rangle, 0.3)$.
The loss function used here is $\phi_c$ with $c = 2$. 

\paragraph{Results} 
Figure \ref{fig_anticorr} shows a bar plot of the average excess risk of CV and Agghoo as a fraction of the average risk of the oracle. 90 \% error bars were estimated using a normal approximation. Parameters used for Agghoo and CV were $\tau = 0.9$ and $V = 10$ ($\tau = 0.8$ yields similar result). 

\begin{figure} 
 \includegraphics[scale = 0.7]{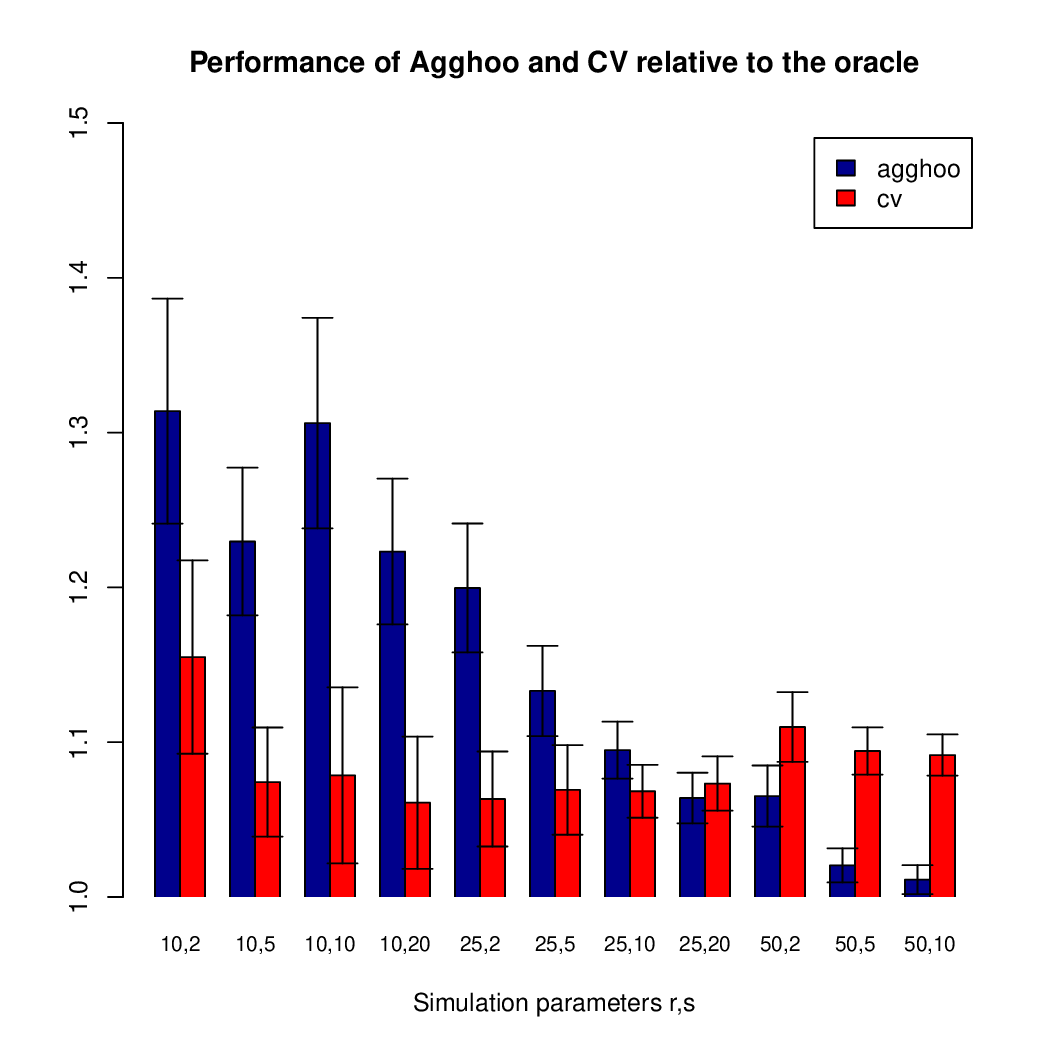}
 \caption{Relative risk in experimental setup 2 (section \ref{exp_setup2})}
 \label{fig_anticorr}
\end{figure}

Overall, Agghoo's risk relative to the oracle significantly decreases as the zero-norm of $w_*$ increases from $r = 10$ to $r = 50$ , as was observed in section \ref{exp_setup1} . For $r = 25$ and $r = 50$ separately, the risk relative to the oracle significantly decreases as $s$ increases from $2$ to $10$. For $r = 10$, this trend is unclear due to the random errors.

In contrast, CV's performance relative to the oracle shows no statistically significant trend either as a function of $r$ or as a function of $s$.

As a result of these trends, Agghoo performs significantly worse than CV for $r = 10$ and significantly better when $r = 50$, especially when $s \geq 5$. When $r = 25$, CV performs significantly better than Agghoo for $s = 2$ and $s = 5$ and they perform similarly when $s = 10$ and $s = 20$.

\subsection{Experimental setup 3: correlations between predictive variables}
\label{exp_setup3}

We consider now predictive covariates which are correlated between them, and independent from the unpredictive covariates.
As above, let $r$ denote the number of predictive variables and $\rho > 0$ be the level of correlations. Let $Z_0$, $(Z_i)_{1 \leq i \leq r}$ and $(W_i)_{1 \leq i \leq d-r}$ be standard Gaussian random variables. The random variable $X$ is then defined by $X_i = \sqrt{\rho} Z_0 + \sqrt{1-\rho}Z_i$ for $1 \leq i \leq r$ and $X_i = W_{i - r}$ for $r+1 \leq i \leq d$.  As in section \ref{exp_setup2}, the regression coefficient $w_*$ is a constant vector of the form $\frac{3*u}{\Norm{Xu}_{L^2}}$, where this time $u = \left(\mathbb{I}_{1 \leq i \leq r} \right)_{1 \leq i \leq d}$. 

$Y$ is distributed conditionnally on $X$ as $\text{Cauchy}(\langle X, w_* \rangle,0.3)$ and the loss function used is the Huber loss $\phi_2$.

\paragraph{Results} Figure \ref{fig_corr_simple} shows a barplot generated in the same way as in section \ref{exp_setup2}. Parameters used for Agghoo and CV were $V = 10$ and $\tau = 0.8$, which is optimal in this case for both Agghoo and CV.

\begin{figure} 
    \includegraphics[scale = 0.7]{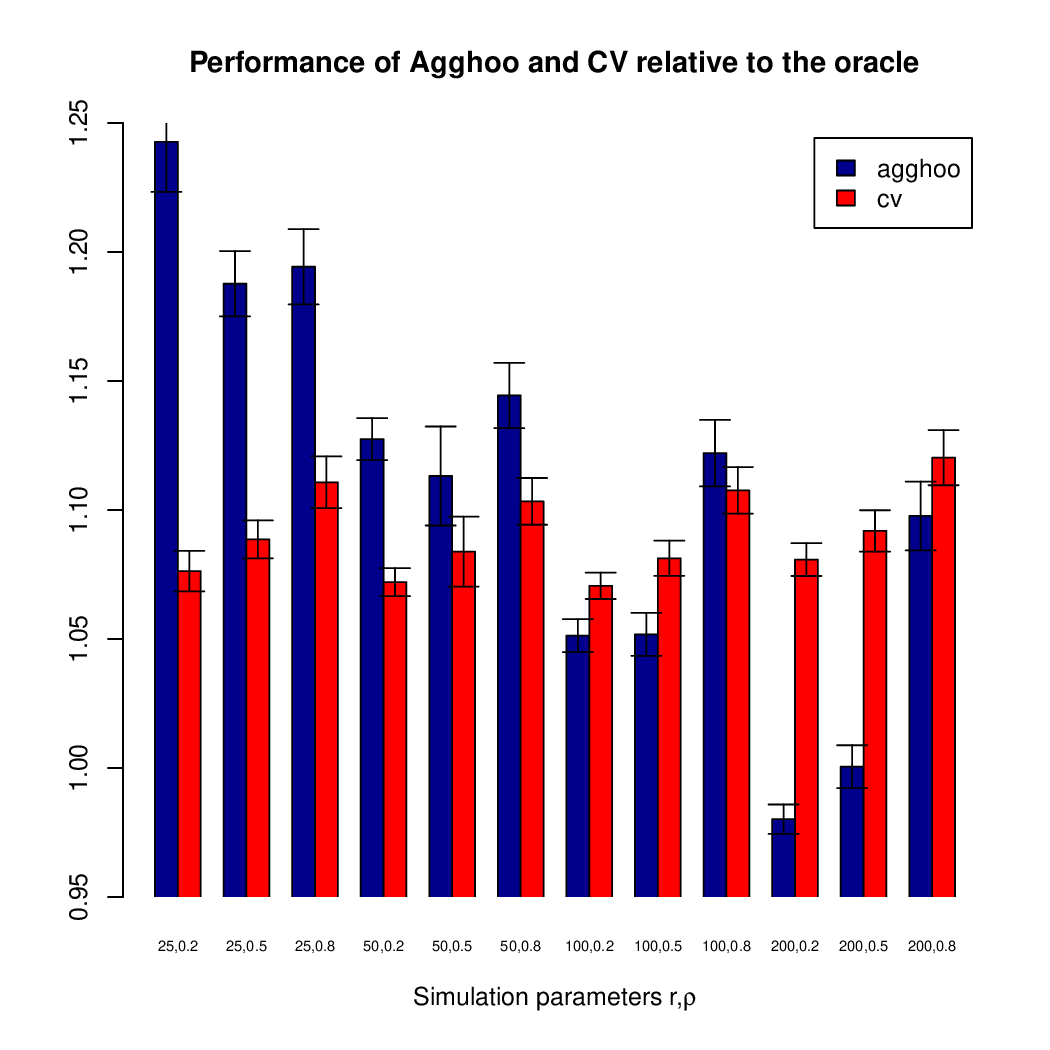}
  \caption{Relative risk in experimental setup 3 (section \ref{exp_setup3})}
  \label{fig_corr_simple}
\end{figure}

As in previous simulations, Agghoo's performance relative to the oracle improves significantly when the intrinsic dimension $r$ grows from $25$ to $200$, for a given value of $\rho$. The decrease in relative risk is faster for small values of $\rho$. As a result, Agghoo performs best, relative to the oracle, when $\rho = 0.2$ for $r = 200$, whereas best performance seems to occur at $\rho = 0.5$ for smaller values of $r$, up to random errors. 

For cross-validation, the relative risk seems more or less unaffected by the dimension $r$, but shows an increasing trend as a function of $\rho$ for all values of $r$. 

As a result, Agghoo performs better than CV for $r = 200$ and for $r = 100$ and $\rho = 0.2,0.5$. For $r = 200$ and $\rho = 0.2$, Agghoo even performs significantly better than the oracle! This is possible, since the Agghoo regression coefficient $\hat{\theta}_{\mathcal{T}}^{ag}$ does not itself belong to the Lasso regularization path.    

\FloatBarrier

\section{Conclusion}
Aggregated hold-out (Agghoo) satisfies an oracle inequality (Theorem \ref{agcv_hub}) in sparse linear regression with the Huber loss. This oracle inequality is asymptotically optimal in the non-parametric case where the intrinsic dimension tends to $+\infty$ with the sample size $n$, provided that an $L^{\psi_1} - L^2$ norm inequality holds on the set of sparse linear predictors. The condition holds for gaussian vectors and for classical approximation spaces in non-parametric regression. In the case of the trigonometric basis, this approach yields an oracle inequality in which the total dimension $d$ does not appear. 

When Monte-Carlo subsampling is used (Definition \ref{def_mcbag}), Agghoo has two parameters, $\tau$ and $V$. Theoretically, it is shown that Agghoo's performance always improves when $V$ grows for a fixed $\tau$. Simulations show a large improvement from $V = 1$ to $V = 5$ in some cases, but diminishing returns for $V > 5$. 
With respect to $\tau$, simulations show that $\tau = 0.8$ or $\tau = 0.9$ is optimal or near optimal in most cases. In particular, a default choice of $V = 10$, $\tau = 0.8$ seems reasonable.

Compared to cross-validation with the same number of splits $V$, simulations show that Agghoo performs better when the intrinsic dimension $r$ is large enough ($r = 150$ in section \ref{exp_setup1}, $r = 50$ in section \ref{exp_setup2} and $r = 100$ in \ref{exp_setup3}) for $n = 100$ observations and $d = 1000$ covariates. Correlations between predictive and non-predictive covariates, which increase the number of covariates correlated with the response $Y$, clearly favour Agghoo relative to CV and the oracle, whereas the effect of  correlations between predictive covariates is ambiguous.

\section*{Acknowledgements}
While finishing the writing of this article, the author (Guillaume Maillard) has received funding from the European Union's Horizon 2020 research and innovation program under grant agreement No 811017.

\appendix 
\section{Proof of Proposition \ref{prop_minimax_huber_reg}} \label{proof_prop_minimax}
The proof follows the same lines as the proof of \cite[Theorem 1]{Mourtada:2019}, with some differences due to the non-quadratic risk.

Since $\hat{\theta}$ is allowed to depend on $\Sigma$, which is positive definite by assumption, we can always replace
the $X_i$ by $\Sigma^{- \frac{1}{2}} X_i$. Thus, it can be assumed without
loss of generality that $\Sigma = I_n$. 
Using the notation of Proposition \ref{prop_minimax_huber_reg}
\[ \ell(\theta_*^T,\theta^T) = E \left[ \phi_c(\sigma \varepsilon + \langle \theta_* - \theta, X \rangle) - \phi_c(\sigma \varepsilon) \right], \]
where $\varepsilon,X$ are assumed to be independent from the sample $D_n$.
Since $\varepsilon,X$ are independent, centered normal variables,
$\sigma \varepsilon + \langle \theta_* - \theta, X \rangle$ is centered normal,
with variance $\sigma^2 + \Norm{\theta_* - \theta}_2^2$.

It follows that 
\[ \ell(\theta_*^T,\theta^T) = g_c(\sqrt{\sigma^2 + \Norm{\theta_* - \theta}_2^2}) - g_c(\sigma), \text{ where } g_c(x) := E [\phi_c(x Z)] \text{ for } Z \sim \mathcal{N}(0,1). \]
Let also $g_{c,\sigma}(r) = g_c(\sqrt{r^2 + \sigma^2}) - g_c(\sigma)$, so that $\ell(\theta_*^T,\theta^T) = g_{c,\sigma}(\Norm{\theta_* - \theta}_2)$.

Consider the prior $\Pi_\lambda = \mathcal{N}(0,\frac{\sigma^2}{\lambda n} I_d)$ on $\theta_*$. Then a classical computation \cite{Mourtada:2019} shows that the posterior $\hat{\pi}_n = \Pi_\lambda(\cdot | D_n)$ is gaussian and centered at the ridge estimator
\[ \hat{\theta}_{\lambda,n} = (\hat{\Sigma}_n + \lambda I_d)^{-1} \frac{1}{n} \sum_{i = 1}^n Y_i X_i, \]
where $\hat{\Sigma}_n$ is the empirical covariance matrix. Fix a sample $D_n$ and let $\tilde{\theta} \sim \hat{\pi}_n$ be independent from $\varepsilon,X$.
Notice that
\begin{align*}
E [\nabla_{\theta} \ell(\tilde{\theta}^T,\theta^T) ]
&= E [X \phi_c'(\sigma \varepsilon + \langle \tilde{\theta} - \theta, X \rangle) ] \\
&= E \left[ X E \bigl[\phi_c'(\sigma \varepsilon + \langle \tilde{\theta} - \theta, X \rangle) | X \bigr] \right].
\end{align*}
Set now $\theta = \hat{\theta}_{\lambda,n}$. Since $\tilde{\theta} \sim \hat{\pi}_n$, knowing $X$, $\langle \tilde{\theta} - \hat{\theta}_{\lambda,n}, X \rangle$ is centered normal and independent from $\varepsilon$, which is also centered normal. It follows that $\mathbb{E} \bigl[\phi_c'(\sigma \varepsilon + \langle \tilde{\theta} - \theta, X \rangle) | X \bigr] = 0$, since $\phi_c'$ is an odd function. This shows that $\hat{\theta}_{\lambda,n}$ is a Bayes estimator with respect to the prior $\Pi_\lambda$ and the loss function $\ell$.
  
Thus, for any estimator $\hat{\theta}$,
\begin{align*}
\sup_{\theta_*} \mathbb{E}_{D_n \sim P_{\theta_*}^{\otimes n}} \left[ \ell(\hat{\theta}(D_n)^T,\theta_*^T) \right]
&\geq E_{\theta_* \sim \Pi_\lambda} \left[ \mathbb{E}_{D_n \sim P_{\theta_*}^{\otimes n}} \left[ \ell(\theta_*^T, \hat{\theta}^T(D_n)) \right] \right] \\
&\geq E_{\theta_* \sim \Pi_\lambda} \left[ \mathbb{E}_{D_n \sim P_{\theta_*}^{\otimes n}} \left[ \ell(\theta_*^T, \hat{\theta}^T_{\lambda_,n}(D_n)) \right] \right] \\
&= E_{\theta_* \sim \Pi_1} \left[ P_{\frac{\theta_*}{\sqrt{\lambda}}}^{\otimes n} \left[ \ell(\tfrac{\theta_*^T}{\sqrt{\lambda}}, \hat{\theta}^T_{\lambda_,n}(D_n)) \right] \right].
\end{align*} 
$\ell(\theta_*^T,\theta^T) = g_{c,\sigma}(\Norm{\theta_* - \theta}_2)$, so by convexity of $\ell(\theta_*^T,\cdot)$, $g_{c,\sigma}$ must be convex. Hence, by Jensen's inequality,
\[ \mathbb{E} \left[ \ell(\theta_*^T,\hat{\theta}^T_{\lambda_,n}(D_n))  \right] \geq 
g_{c,\sigma} \left( \mathbb{E} \left[ \Norm{\theta_* - \hat{\theta}_{\lambda_,n}(D_n)}_2 \right] \right). \]
Under $P_{\frac{\theta_*}{\sqrt{\lambda}}}$, $Y_i = \langle \frac{\theta_*}{\sqrt{\lambda}}, X_i \rangle + \sigma \varepsilon_i $, so
\begin{align*}
\hat{\theta}_{\lambda,n} - \theta_* &=
(\hat{\Sigma}_n + \lambda I_d)^{-1} \left( \frac{1}{n} \sum_{i = 1}^n X_i X_i^T \frac{\theta_*}{\sqrt{\lambda}} + \sigma \varepsilon_i X_i \right) - \frac{\theta_*}{\sqrt{\lambda}} \\
&= (\hat{\Sigma}_n + \lambda I_d)^{-1} \hat{\Sigma}_n \frac{\theta_*}{\sqrt{\lambda}} - \frac{\theta_*}{\sqrt{\lambda}} + \frac{\sigma}{n} \sum_{i = 1}^n \varepsilon_i (\hat{\Sigma}_n + \lambda I_d)^{-1} X_i \\
&= - \sqrt{\lambda} (\hat{\Sigma}_n + \lambda I_d)^{-1} \theta_* 
+  \frac{\sigma}{n} \sum_{i = 1}^n \varepsilon_i (\hat{\Sigma}_n + \lambda I_d)^{-1} X_i.
\end{align*}
Since $d < n$, $\hat{\Sigma}_n$ is almost surely non-degenerate.
It follows that 
\[ \lim_{\lambda \to 0} \Norm{\hat{\theta}_{\lambda,n} - \theta_*} = \Norm{\frac{\sigma}{n} \sum_{i = 1}^n \varepsilon_i \hat{\Sigma}_n^{-1} X_i}. \]
Let $\hat{r}_n = \frac{\sigma}{n} \sum_{i = 1}^n \varepsilon_i \hat{\Sigma}_n^{-1} X_i$. By Fatou's lemma,
\[ \sup_{\theta_*} \mathbb{E}_{D_n \sim P_{\theta_*}^{\otimes n}} \left[ \ell(\theta_*^T,\hat{\theta}^T(D_n)) \right] 
\geq g_{c,\sigma}\left( \mathbb{E} [\Norm{\hat{r}_n}_2] \right)   \]
Since the $\varepsilon_i$ are iid normal $\mathcal{N}(0,1)$ and independent from the $X_i$, conditionnally on $X_1,\ldots,X_n$, $\hat{r}_n$ is centered normal, with covariance matrix
\[ \frac{\sigma^2}{n^2} \sum_{i = 1}^n  \hat{\Sigma}_n^{-1} X_i X_i^T \hat{\Sigma}_n^{-1} = \frac{\sigma^2}{n} \hat{\Sigma}_n^{-1}. \]
It follows by lemma \ref{lem_lb_exp_norm_gauss} that
\[ \mathbb{E} [\Norm{\hat{r}_n}_2 | (X_i)_{1 \leq i \leq n}] \geq \sqrt{\frac{2}{\pi}} \frac{\sigma}{\sqrt{n}} \sqrt{Tr(\hat{\Sigma}_n^{-1})}. \]
By convexity of the function $M \mapsto \sqrt{Tr(M^{-1})}$ 
on the positive definite matrices (lemma \ref{lem_conv_sqrt_tr_inv_mat}),
\[ \mathbb{E} [\Norm{\hat{r}_n}_2] \geq \sqrt{\frac{2}{\pi}} \frac{\sigma}{\sqrt{n}} \sqrt{Tr \left( \mathbb{E}[\hat{\Sigma}_n]^{-1} \right)} = \sigma \sqrt{\frac{2}{\pi}} \sqrt{\frac{d}{n}}.  \]
Since $g_c$ is non-decreasing and convex,
\begin{align*}
\sup_{\theta_*} \mathbb{E}_{D_n \sim P_{\theta_*}^{\otimes n}} \left[ \ell(\theta_*^T,\hat{\theta}^T(D_n)) \right] &\geq 
g_{c,\sigma}\left( \sigma \sqrt{\frac{2}{\pi}} \sqrt{\frac{d}{n}} \right) \\
&= g_c \left( \sigma \sqrt{1 + \frac{2d}{\pi n}} \right) - g_c(\sigma) \\
&\geq \sigma g_c'(\sigma) \left[\sqrt{1 + \frac{2d}{\pi n}} - 1 \right]. 
\end{align*}
By definition, $g_c(x) = E[\phi_c(xZ)]$, where $Z \sim \mathcal{N}(0,1)$, so
\[ \sigma g_c'(\sigma) = \sigma E[Z \phi_c'(\sigma Z)]
= \sigma E[\min(\sigma Z^2, c |Z|)]. \]
This proves the proposition.

\begin{lemma} \label{lem_lb_exp_norm_gauss}
Let $Y \sim \mathcal{N}(0,\Sigma)$ be a gaussian vector, where $\Sigma$ is positive definite.
Then 
\[ \mathbb{E}\left[ \Norm{Y}_2 \right] \geq \sqrt{\frac{2}{\pi}} \sqrt{Tr(\Sigma)}. \]
\end{lemma}
\begin{proof}
Let $Y_0 = \Sigma^{-\frac{1}{2}} Y \sim \mathcal{N}(0,I_d)$. Then
\begin{align*}
E\left[ \Norm{Y}_2 \right] &= E\left[ \Norm{\Sigma^{\frac{1}{2}} Y_0}_2 \right] \\
&= E \left[\sqrt{Y_0^T \Sigma Y_0} \right].
\end{align*}
Thus, the lemma is equivalent to
\[ E \left[\sqrt{Y_0^T \frac{\Sigma}{Tr(\Sigma)} Y_0} \right] \geq \sqrt{\frac{2}{\pi}}. \]
Let $\Sigma_0 = \frac{\Sigma}{Tr(\Sigma)}$.
Let $\Sigma_0 = Q^T D Q$, where $D$ is diagonal and $Q$ is orthogonal.
Let $\lambda_1,\ldots,\lambda_d$ be the diagonal coefficients of $D$ (that is to say, the eigenvalues of $\Sigma_0$).
Then 
\[ E \left[\sqrt{Y_0^T \Sigma_0 Y_0} \right] 
= E \left[\sqrt{ (QY_0)^T D (QY_0)} \right]. \]
As $Q$ is orthogonal, $Q Y_0 \sim \mathcal{N}(0,I_d)$, so
\[ E \left[\sqrt{Y_0^T \Sigma_0 Y_0} \right] = E \left[\sqrt{Y_0^T D Y_0} \right] = E \left[\sqrt{\sum_{i = 1}^d \lambda_i Y_{0i}^2} \right].   \]
The coefficients $\lambda_i$ are positive (since $\Sigma_0$ is positive definite) and sum to $1$ (since $Tr(\Sigma_0) = 1$ by construction).
It follows by Jensen's inequality that
\begin{align*}
E \left[\sqrt{Y_0^T \Sigma_0 Y_0} \right] &\geq E \left[\sum_{i = 1}^d \lambda_i |Y_{0i}| \right] \\
&= E [|Y_{01}|] \sum_{i = 1}^d \lambda_i \\
&=  E [|Y_{01}|] \\
&= \sqrt{\frac{2}{\pi}} \text{ since } Y_{01} \sim \mathcal{N}(0,1).
\end{align*}
This proves the lemma.
\end{proof}

\begin{lemma} \label{lem_conv_sqrt_tr_inv_mat}
The function $f: M \mapsto \sqrt{Tr(M^{-1})}$ is convex over
the convex cone of positive definite matries.
\end{lemma}

\begin{proof}
Let $M$ be a positive definite matrix. Let $H$ be a small, symmetric
perturbation. Then
\begin{align*}
(M + H)^{-1} &=  (I_d + M^{-1}H)^{-1} M^{-1} \\
&= \left(I_d - M^{-1}H + (M^{-1}H)^2 + o(\Norm{H}^2) \right) M^{-1} \\
&= M^{-1} - M^{-1}H M^{-1} + (M^{-1}H)^2 M^{-1} + o(\Norm{H}^2).
\end{align*}
Therefore,
\[ Tr((M + H)^{-1}) = Tr(M^{-1}) - Tr(M^{-1}H M^{-1}) + Tr((M^{-1}H)^2 M^{-1}) 
+ o(\Norm{H}^2). \]
For any positive real $a > 0$, $\sqrt{a + h} = \sqrt{a} + \frac{h}{2\sqrt{a}} - \frac{h^2}{8 a^{\frac{3}{2}}} + o(h^2)$.
It follows that
\begin{align*}
\sqrt{Tr((M + H)^{-1})} &= \sqrt{Tr(M^{-1})} - \frac{Tr(M^{-1}H M^{-1})}{2 \sqrt{Tr(M^{-1})}} \\ 
&\quad + \frac{Tr((M^{-1}H)^2M^{-1})}{2 \sqrt{Tr(M^{-1})}} 
- \frac{Tr(M^{-1}H M^{-1})^2}{8 Tr(M^{-1})^{\frac{3}{2}}} + o(\Norm{H}^2).
\numberthis \label{eq_dev_lim_sqrt_trace_inv}
\end{align*}
For any two matrices $A,B$, let $\langle A,B \rangle := Tr(M^{-\frac{1}{2}} A B^T M^{-\frac{1}{2}})$. It is easy to see that this defines a scalar product.
Thus, by the Cauchy-Schwarz inequality,
\begin{align*}
Tr(M^{-1}H M^{-1})^2 &= Tr(M^{-\frac{1}{2}} M^{-\frac{1}{2}} H M^{-\frac{1}{2}} M^{-\frac{1}{2}})^2 \\
&= \langle M^{-\frac{1}{2}} H M^{-\frac{1}{2}}, I_d \rangle^2 \\
&\leq  \langle I_d, I_d \rangle \langle M^{-\frac{1}{2}} H M^{-\frac{1}{2}}, M^{-\frac{1}{2}} H M^{-\frac{1}{2}} \rangle \\
&= Tr(M^{-1}) Tr \left( M^{-1} H M^{-1} H M^{-1} \right) \\
&= Tr(M^{-1})Tr((M^{-1}H)^2 M^{-1}).
\end{align*}
Thus,
\[
\frac{Tr((M^{-1}H)^2 M^{-1})}{2 \sqrt{Tr(M^{-1})}} 
- \frac{Tr(M^{-2}H)^2}{8 Tr(M^{-1})^{\frac{3}{2}}} \geq 
\frac{3}{8} \frac{Tr((M^{-1}H)^2 M^{-1})}{\sqrt{Tr(M^{-1})}} \geq 0. 
\]
By equation \eqref{eq_dev_lim_sqrt_trace_inv}, this proves that the Hessian of $f$ at $M$ is non-negative definite.
\end{proof}

\section{Proof of Theorem \ref{agcv_hub}} \label{app.sec_proof_thm}
The idea of the proof is to apply \cite[Theorem 17]{agghoo_rkhs} using suitable functions $(\hat{w}_{i,j})_{(i,j) \in \{1;2\}^2}$.

In this proof, we shall adopt the following notational conventions.
The notation $\mathbb{P}, \mathbb{E}$ will be reserved for probabilities and expectations which involve the sample $D_{n_t}$ (or $D_n$). 
For a (possibly random) function $f: \mathbb{R}^{d} \times \mathbb{R} \to \mathbb{R}$, $P(f) = P(f(X,Y))$ will denote the expectation taken with respect to $(X,Y) \sim P$ only (ignoring the potential randomness in the construction of $f$). The notation $E$ will be used for any other expectation.
Moreover, for any measurable function $t: \mathbb{R}^{d} \to \mathbb{R}$, we denote 
\begin{align*}
    \Norm{t}_{\alpha,P} &=  \Norm{t(X)}_{\alpha,P} ;= \Norm{t(X)}_{L^\alpha} \text{ where } (X,Y) \sim P \\
    \Norm{t}_{\psi_1,P} &=  \Norm{t(X)}_{\psi_1,P} ;= \Norm{t(X)}_{L^{\psi_1}} \text{ where } (X,Y) \sim P.
\end{align*}
For a random function $\hat{t}: \omega \mapsto (x \mapsto \hat{t}(\omega)(x))$, let 
\[ \Norm{\hat{t}}_{\alpha,P} = \Norm{\hat{t}(X)}_{\alpha,P} : \omega \mapsto \Norm{\hat{t}(\omega)}_{\alpha,P}, \]
with a similar definition for $\Norm{\hat{t}}_{\psi_1,P}$.

Fix a dataset $D_{n_t}$, $K \in \{1,\ldots, n_t \}$ and for any $k \in [|1;K|]^2$, 
let $\hat{t}_k = \learnrule_k(D_{n_t}): x \rightarrow \hat{q}_k(D_{n_t}) + \langle \hat{\theta}_k(D_{n_t}), x \rangle $. More precisely, to apply \cite[Theorem 17]{agghoo_rkhs}, one must show inequalities of the form $H(w_1,w_2, (\hat{t}_k)_{1 \leq k \leq K})$:
for all $r \geq 2$,
\begin{equation} \label{hypothesis_H}
\begin{split}
P \Bigl( \bigl\lvert \phi_c(\hat{t}_k(X) - Y)- \phi_c(\hat{t}_l(X) - Y) - c_{l}^k \bigr\rvert^r \Bigr)
&\leq r! \Bigl[ w_1\bigl(\sqrt{\loss{\hat{t}_k}}\bigr) + w_1(\sqrt{\loss{\hat{t}_l}}) \Bigr]^2 
\\
&\qquad \times \Bigl[ w_2\bigl(\sqrt{\loss{\hat{t}_k}} \bigr) + w_2 \bigl( \sqrt{\loss{\hat{t}_l}} \bigr) \Bigr]^{r-2} 
\enspace , 
\end{split}
\end{equation}
where $w_1,w_2$ are non-decreasing functions. Since $\phi_c$ is Lipschitz, it is enough to control $\Norm{\hat{t}_k - \hat{t}_l}_{\psi_1,P}$ and $\Norm{\hat{t}_k - \hat{t}_l}_{2,P}$ by functions of $\loss{\hat{t}_k}$ and $\loss{\hat{t}_l}$.

\subsection{A few lemmas}
\begin{lemma} \label{lem_int_fun_subexp}
Let $X$ be a non-negative random variable such that
\[ \forall x \in \mathbb{R}, P(X \geq x) \leq a e^{-x}, \]
where $a \geq 1$.
Let $g \in L^1(\mathbb{R}_+, e^{-x} dx)$ be an increasing,
differentiable function.
Then for all $b \in \mathbb{R}_+$,
\[ \mathbb{E} \left[ g(X) \mathbb{I}_{X \geq b} \right] \leq a \int_b^{+ \infty} e^{-v} g(v) dv . \]
\end{lemma}

\begin{proof}
\begin{align*}
   \mathbb{E} \left[ g(X) \mathbb{I}_{X \geq b} \right] &= \int_0^{+ \infty} P \left( g(X) \mathbb{I}_{X \geq b} \geq u \right) du \\
   &= g(b) P(X \geq b) + \int_{g(b)}^{+ \infty} P \left( g(X) \geq u \right) du \\
   &= g(b) P(X \geq b) + \int_{b}^{+ \infty} P \left( g(X) \geq g(v) \right) g'(v) dv \\
   &\leq g(b) P(X \geq b) + a \int_{b}^{+ \infty} e^{-v} g'(v) dv \text{ since } g \text{ increases} \\
   &\leq g(b) P(X \geq b) - a e^{-b} g(b) + a \int_b^{+ \infty} e^{-v} g(v) dv 
    \\
   &\leq a \int_b^{+ \infty} e^{-v} g(v) dv.
\end{align*}
\end{proof}

\begin{lemma} \label{lem_holder_L2_L1}
Let $Z$ be a random variable. Then for all $r > 2$,
\[ E [Z^2] \leq E [|Z|]^{\frac{r-2}{r-1}} 
E [|Z|^r]^{\frac{1}{r-1}}. \]
In particular, if $\Norm{Z}_{L^r} \leq \kappa_r \Norm{Z}_{L^2}$ for some
$r > 2, \kappa_r > 0$, 
then $\Norm{Z}_{L^2} \leq \kappa_r^{\frac{r}{r-2}} \Norm{Z}_{L^1}$.
\end{lemma}

\begin{proof}
Let $p = \frac{r-1}{r-2} > 1$, $\frac{1}{q} = 1 - \frac{1}{p}$, $\alpha = \frac{1}{p}$,
then by Hölder's inequality,
\begin{align*}
E [Z^2] &= E[|Z|^\alpha |Z|^{2 - \alpha} ] \\
&\leq E \left[|Z|^{p \alpha} \right]^{\frac{1}{p}} 
E \left[|Z|^{q(2 - \alpha)} \right]^{\frac{1}{q}}.
\end{align*}
Now by definition, $\frac{1}{p} = \frac{r-2}{r-1}$,  
$\frac{1}{q} = 1 - \frac{r-2}{r-1} = \frac{1}{r-1}$,
$p \alpha = p \times \frac{1}{p} = 1$ and
\begin{align*}
q(2- \alpha) &= \frac{2 - \frac{1}{p}}{1 - \frac{1}{p}} \\
&= \frac{2 - \frac{r-2}{r-1}}{1 - \frac{r-2}{r-1}} \\
&= \frac{2 (r-1) - (r-2)}{r-1 - (r-2)} \\
&= r.
\end{align*}
Assume now that $\Norm{Z}_{L^r} \leq \kappa_r \Norm{Z}_{L^2}$. Then
\begin{align*}
\Norm{Z}_{L^2}^2 &= E[Z^2] \\
&\leq E[|Z|]^{\frac{r-2}{r-1}}  E [|Z|^r]^{\frac{1}{r-1}} \\
&\leq \Norm{Z}_{L^1}^{\frac{r-2}{r-1}} \kappa_r^{\frac{r}{r-1}} 
\Norm{Z}_{L^2}^{\frac{r}{r-1}}.
\end{align*}
It follows that
\[ \Norm{Z}_{L^2}^{\frac{r-2}{r-1}} \leq \kappa_r^{\frac{r}{r-1}} \Norm{Z}_{L^1}^{\frac{r-2}{r-1}},  \]
which yields the result.
\end{proof}

\begin{lemma} \label{lem_ub_Lr_psi1}
Let $Z$ be a $\psi_1 -$ random variable.
Then for all $r \in \mathbb{N}$,
\[ \Norm{Z}_{L^r}^r \leq 2 r! \Norm{Z}_{L^{\psi_1}}^r \]
\end{lemma}

\begin{proof}
By definition of $\Norm{Z}_{L^{\psi_1}}$ and Markov's inequality
\[
\mathbb{P} \left( \frac{Z}{\Norm{Z}_{L^{\psi_1}}} \geq x \right)
\leq 2e^{-x}.
\]
It follows by lemma \ref{lem_int_fun_subexp} that
\begin{align*}
E \left[  \left( \frac{Z}{\Norm{Z}_{L^{\psi_1}}} \right)^r \right] 
&\leq 2 \int_0^{+ \infty} x^r e^{-x} dx \\
&\leq 2r! \text{(moment of an exponential distribution)}.
\end{align*}
\end{proof}

\begin{lemma} \label{lem_bernstein_psi1}
Let $Z$ be a $\psi_1-$random variable such that
$\Norm{Z}_{L^{\psi_1}} \leq \kappa \Norm{Z}_{L^2}$, where $\kappa \geq \sqrt{2}$.
Then for all integers $r \geq 2$,
\[ E \left[ Z^r \right] \leq r! E[Z^2] \left( (4 + 4\log \kappa) \Norm{Z}_{L^{\psi_1}} \right)^{r-2}. \]
\end{lemma}

\begin{proof}
Since $2 > 1$, the statement is true for $r = 2$.
Consider now $r \geq 3$.
Let $b > 1$ be a real number to be determined later.
Then 
\begin{align*}
E \left[ Z^r \right] &\leq E \left[ Z^r \mathbb{I}_{Z \leq b \Norm{Z}_{L^{\psi_1}}} \right] + E \left[ Z^r \mathbb{I}_{Z \geq b \Norm{Z}_{L^{\psi_1}}} \right] \\
&\leq b^{r-2} \Norm{Z}_{L^{\psi_1}}^{r-2} E[Z^2] + \Norm{Z}_{L^{\psi_1}}^{r} 
E \left[ \left(\frac{Z}{\Norm{Z}_{L^{\psi_1}}} \right)^r \mathbb{I}_{\frac{Z}{\Norm{Z}_{L^{\psi_1}}} \geq b} \right]. 
\end{align*}
By definition of $\Norm{Z}_{L^{\psi_1}}$ and  a Chernoff bound, the variable $Y = \frac{Z}{\Norm{Z}_{L^{\psi_1}}}$ satisfies 
$\mathbb{P}(Y \geq x) \leq 2e^{-x}$ for all $x$, therefore by lemma \ref{lem_int_fun_subexp},
\[ E \left[ Z^r \right] \leq 
b^{r-2} \Norm{Z}_{L^{\psi_1}}^{r-2} E[Z^2] + 2 \Norm{Z}_{L^{\psi_1}}^{r} \int_{b}^{+ \infty} t^r e^{-t} dt. \]
An easy induction argument shows that 
\begin{align*}
\int_{b}^{+ \infty} t^r e^{-t} dt &= \sum_{j = 0}^r \frac{r!}{j!} b^j e^{-b} \\
&= r! b^r e^{-b} \sum_{j = 0}^r \frac{1}{j! b^{r-j}}.  
\end{align*}
It follows that
\[ E \left[ Z^r \right] \leq 
b^{r-2} \Norm{Z}_{L^{\psi_1}}^{r-2} E[Z^2] + 2 \Norm{Z}_{L^{\psi_1}}^{r} r! b^r e^{-b} \sum_{j = 0}^r \frac{1}{j! b^{r-j}}. \]
Let $b = 4 + 4 \log \kappa \geq 4 + 2 \log 2$. Then for all $r \geq 3$,
\begin{align*}
\sum_{j = 0}^r \frac{1}{j! b^{r-j}} &= \frac{1}{b^r} + \frac{1}{b^{r-1}} 
+ \frac{1}{2b^{r-2}} + \sum_{j = 3}^r \frac{1}{j! b^{r-j}} \\
&\leq \frac{1}{b^3} + \frac{1}{b^2} 
+ \frac{1}{2b} + \frac{1}{6} + \frac{1}{(r \vee 4)!} + \frac{1}{b} \sum_{j = 4}^{+\infty} \frac{1}{j!} \\
&\leq \frac{1}{b^3} + \frac{1}{b^2} 
+ \frac{1}{2b} + \frac{1}{6} + \frac{1}{24} + \frac{1}{b}\left(e - 2 - \frac{1}{2} - \frac{1}{6} \right) \\
&< 0.36 . 
\end{align*}
As a result, for all $r \geq 3$,
\[ E \left[ Z^r \right] \leq 
\frac{r!}{3!} b^{r-2} \Norm{Z}_{L^{\psi_1}}^{r-2} E [Z^2] + 0.72 \Norm{Z}_{L^{\psi_1}}^{r} r! b^r e^{-b}. \] 
We now prove that for all $t \geq b$, $t \geq 2 \log t + 2 \log \kappa$.
For all $t \geq 4$, 
\[ \frac{d}{dt} [t - 2 \log t - 2 \log \kappa] 
= 1 - \frac{2}{t} \geq \frac{1}{2}, \]
therefore 
\begin{align*}
t - 2 \log t &\geq 4 - 2 \log(4) + \frac{t - 4}{2} \\
&\geq  \frac{t - 4}{2} .
\end{align*}
It follows that for all $t > 4 + 4\log(\kappa) = b$,
$t > 2 \log t + 2 \log(\kappa)$.
In particular, $b^2 e^{-b} \leq b^2 \exp(- 2 \log(b) - 2 \log(\kappa)) \leq \frac{1}{\kappa^2}$,
therefore 
\begin{align*}
E \left[ Z^r \right] &\leq 
\frac{r!}{6} b^{r-2} \Norm{Z}_{L^{\psi_1}}^{r-2} E[Z^2] + 0.72 \Norm{Z}_{L^{\psi_1}}^{r} r! b^{r-2} \frac{1}{\kappa^2} \\
&\leq \frac{r!}{6} b^{r-2} \Norm{Z}_{L^{\psi_1}}^{r-2} E[Z^2] + 0.72 \Norm{Z}_{L^{\psi_1}}^{r} r! b^{r-2} \frac{E[Z^2]}{\Norm{Z}_{L^{\psi_1}}^2} \\
&\leq r! E[Z^2] (b \Norm{Z}_{L^{\psi_1}})^{r-2}. 
\end{align*}
\end{proof}

\begin{lemma} \label{lem_ub_l2_l1}
There exists a constant $\mu_0$ such that, for any
sub-exponential random variable $Z$ and any $\kappa \geq \sqrt{2}$,
\[ \Norm{Z}_{L^{\psi_1}} \leq \kappa \Norm{Z}_{L^2} 
\implies \Norm{Z}_{L^2} \leq \mu_0 \kappa \log \kappa \Norm{Z}_{L^1}.   \]
\end{lemma}
\begin{proof} 
By lemmas \ref{lem_ub_Lr_psi1} and \ref{lem_holder_L2_L1},
for all $r \geq 3$,
\[ \Norm{Z}_{L^2} \leq (2^{\frac{1}{r}} \kappa r)^{\frac{r}{r-2}} \Norm{Z}_{L^1}.\]
Remark that
\begin{align*}
(2^{\frac{1}{r}}\kappa r)^{\frac{r}{r-2}} &= 2^{\frac{1}{r-2}} \kappa^{\frac{r}{r-2}} \times r \times 
r^{\frac{2}{r-2}}
\end{align*}
and 
\[ \frac{d}{dr} \log \left( r^{\frac{2}{r-2}} \right) = \frac{d}{dr} \bigl[ \frac{2 \log r}{r-2} \bigr] = \frac{2}{r-2} \bigl[\frac{1}{r} - \frac{\log r}{r-2} \bigr] \leq 0 \]
for $r \geq 3$ since $\log r \geq 1$ and $\frac{1}{r-2} \geq \frac{1}{r}$.
Let $r = 3 + \log(\kappa) \geq 3$. Thus, $r^{\frac{2}{r-2}} \leq 3^{\frac{2}{3 - 2}} = 9$ and  
\begin{align*}
(2^{\frac{1}{r}} \kappa r)^{\frac{r}{r-2}} &\leq 2 \times 9 \times  r \kappa^{\frac{r}{r-2}} \\
&\leq 18(3 + \log(\kappa))  \kappa \times \kappa^{\frac{2}{1 + \log( \kappa)}}
\\
&\leq 18(3 + \log(\kappa))  \kappa \exp \left( \frac{2 \log(\kappa)}{1 + \log(\kappa)} \right) \\
&\leq 18 e^2 (3 + \log(\kappa))  \kappa.
\end{align*}
The conclusion follows since by assumption, 
$\log \kappa \geq \log (\sqrt{2}) > 0$.
\end{proof}

\subsection{Controlling the \texorpdfstring{$\psi_1$}{Orlicz} norm \texorpdfstring{$\Norm{\hat{t}_k - \hat{t}_l}_{\psi_1,P}$}{of tk - tl}}

First, let us bound the supremum norm by the $L^2$ norm.

\begin{claim} \label{claim_l2_linf}
 For any $k \in \{1,\ldots,K\}$, recall that $\hat{t}_k = \learnrule_k(D_{n_t})$. Then:
 \[\forall (k,l) \in \{1,\ldots,K \}^2, \Norm{\hat{t}_k - \hat{t}_l}_{\psi_1,P} \leq \sqrt{2} \kappa(K) \Norm{\hat{t}_k - \hat{t}_l}_{2,P} \text{ a.s. }.\]
\end{claim}

\begin{proof}
Let $X$ be independent from $D_n$ and observe that for any $k$,
\[ \hat{t}_k(X) = \hat{b}_k + \hat{\theta}_k^T (X - P X), \]
where $\hat{b}_k = \hat{q}_k + \hat{\theta}_k^T (P X)$
 (using the notations of hypothesis \ref{hyp_sparse_reg}).
 Note that $ \Norm{1}_{\psi_1,P} = \frac{1}{\log 2}$. Hence, by the triangle inequality,
 \[ \Norm{\hat{t}_k(X) - \hat{t}_l(X)}_{\psi_1,P} \leq \frac{1}{\log 2} |\hat{b}_k - \hat{b}_l| + \Norm{(\hat{\theta}_k - \hat{\theta}_l)^T (X - P X)}_{\psi_1,P}. \]
 By hypothesis \ref{hyp_sparse_reg},
$\bigl\| \hat{\theta}_k \bigr\|_0  \leq k$.
Thus, if $K \geq \max(k,l)$, $\Norm{\hat{\theta}_k - \hat{\theta}_l}_0 \leq k + l \leq 2K$. The definition of $\kappa$ (equation \eqref{def_kappa}) implies that 
\begin{align*}
 \Norm{\hat{t}_k(X) - \hat{t}_l(X)}_{\psi_1,P} 
 &\leq \frac{1}{\log 2} |\hat{b}_k - \hat{b}_l| + \kappa(K) \Norm{(\hat{\theta}_k - \hat{\theta}_l)^T (X - P X)}_{L^2} \\
 &\leq \kappa(K) \left[ |\hat{b}_k - \hat{b}_l| + \Norm{(\hat{\theta}_k - \hat{\theta}_l)^T (X - P X)}_{L^2} \right] \\
 &\leq \sqrt{2} \kappa(K) \sqrt{|\hat{b}_k - \hat{b}_l|^2 + \Norm{(\hat{\theta}_k - \hat{\theta}_l)^T (X - P X)}^2_{L^2} } \\
 &= \sqrt{2} \kappa(K) \Norm{\hat{t}_k(X) - \hat{t}_l(X)}_{L^2}.
\end{align*}
\end{proof}

A uniform bound on the Orlicz norm is also required. 

\begin{Definition} \label{norm_cst}
Let 
\[ \hat{\beta} = \max_{1 \leq k,l \leq K} \Norm{\hat{t}_k - \hat{t}_l}_{\psi_1,P}. \]
\end{Definition}

$\mathbb{E}[\hat{\beta}]$ can be bounded as follows.

\begin{claim} \label{beta_bound}
Assume that hypotheses \textbf{(Reg-}$\cT$ \textbf{)}, \textbf{(Uub)} hold and that for some $\lambda > 0$, $\kappa(K) \log(\kappa(K)) \leq \lambda \sqrt{n_t}$.
Then
 \[ \mathbb{E} [\hat{\beta}] \leq \left( \frac{2}{\log 2} + \frac{2\mu_0 \lambda^2}{|\log \log 2|} \right) L  n_t^{1 + \alpha}.\]
\end{claim}

\begin{proof}
Let $(k,l) \in \{1,\ldots,K\}^2$.
Defining $\widetilde{X}_i = X_i - \frac{1}{n_t} \sum_{i = 1}^{n_t} X_i$ and changing variables in hypothesis \ref{hyp_sparse_reg} from $(q,\theta)$ to $\left( b = q + <\theta, \frac{1}{n_t} \sum_{i=1}^{n_t} X_i>, \theta \right)$, we can rewrite $\hat{t}_k$ as 
\[\hat{t}_k(x) = \hat{b}_k(D_{n_t}) + \hat{\theta}_k(D_{n_t})^T \left( x -  \frac{1}{n_t} \sum_{i = 1}^{n_t} X_i\right)\] where
\begin{align*}
 \hat{b}_k(D_{n_t}) &\in \argmin_{b \in \hat{Q}'\left( D_{n_t},\hat{\theta}_k(D_{n_t}) \right)} |b| \\
 \hat{Q}'(D_{n_t},\theta) &= \argmin_{b \in \mathbb{R}} \frac{1}{n_t} \sum_{i = 1}^{n_t} \phi_c \left( Y_i - b - \theta^T \widetilde{X}_i \right).
\end{align*}
Therefore, differentiating with respect to $b$,
\[ \frac{1}{n_t} \sum_{i = 1}^{n_t} \phi_c' \bigl( Y_i - \hat{b}_k - \hat{\theta}_k^T \widetilde{X}_i \bigr) = 0. \]
Assume by contradiction that
\begin{equation} \label{eq_size_b}
 \exists b > 0, \forall i \in [|1;n_t |], \hat{b}_k + b + \hat{\theta}_k^T \widetilde{X}_i \leq \hat{b}_l  + \hat{\theta}_l^T \widetilde{X}_i.
\end{equation}
Let $b$ be such that \eqref{eq_size_b} holds. 
Then by monotony of $\phi_c'$, for all $\varepsilon$ in $ [0;\frac{b}{2}]$,
\begin{align*}
 0 &= \frac{1}{n_t} \sum_{i = 1}^{n_t} \phi_c' \bigl( Y_i - \hat{b}_k - \hat{\theta}_k^T \widetilde{X}_i \bigr) \\
 &\geq 
 \frac{1}{n_t} \sum_{i = 1}^{n_t} \phi_c' \bigl( Y_i - \hat{b}_k - \varepsilon - \hat{\theta}_k^T \widetilde{X}_i \bigr) \\
 &\geq \frac{1}{n_t} \sum_{i = 1}^{n_t} \phi_c' \bigl( Y_i - \hat{b}_k - \frac{b}{2} - \hat{\theta}_k^T \widetilde{X}_i \bigr) \\
 &\geq \frac{1}{n_t} \sum_{i = 1}^{n_t} \phi_c' \bigl( Y_i - \hat{b}_l + \frac{b}{2} - \hat{\theta}_l^T \widetilde{X}_i \bigr) \\
 &\geq \frac{1}{n_t} \sum_{i = 1}^{n_t} \phi_c' \bigl( Y_i - \hat{b}_l + \varepsilon - \hat{\theta}_l^T \widetilde{X}_i \bigr) \\
 &\geq \frac{1}{n_t} \sum_{i = 1}^{n_t} \phi_c' \bigl( Y_i - \hat{b}_l - \hat{\theta}_l^T \widetilde{X}_i \bigr) \\
 &= 0.
\end{align*}
It follows that
\begin{equation}
 \forall \varepsilon \in [0;\frac{b}{2}],
 \frac{1}{n_t} \sum_{i = 1}^{n_t} \phi_c' \bigl( Y_i - \hat{b}_k - \varepsilon - \hat{\theta}_k^T \widetilde{X}_i \bigr) = \frac{1}{n_t} \sum_{i = 1}^{n_t} \phi_c' \bigl( Y_i - \hat{b}_l + \varepsilon - \hat{\theta}_l^T \widetilde{X}_i \bigr) = 0.
\end{equation}
By integration, this implies that for all $\varepsilon \in [0;\frac{b}{2}]$,
\begin{align}
 (\hat{b}_k + \varepsilon) &\in \hat{Q}'\left( D_{n_t},\hat{\theta}_k(D_{n_t}) \right) \label{b_eq_1},\\
 (\hat{b}_l - \varepsilon) &\in \hat{Q}' \left( D_{n_t},\hat{\theta}_l(D_{n_t}) \right). \label{b_eq_2}
\end{align}

If $\hat{b}_l > 0$, then for small enough $\varepsilon$, \eqref{b_eq_2} contradicts the minimality of $|\hat{b}_l|$. On the other hand, if $\hat{b}_l \leq 0$, then averaging \eqref{eq_size_b} over $i \in \{1,\ldots, n\}$ yields
\[  \hat{b}_k + b \leq \hat{b}_l \leq 0. \]
Then for $\varepsilon \in [0; \frac{b}{2}]$, \eqref{b_eq_1} contradicts the minimality of $|\hat{b}_k|$.
Thus, \eqref{eq_size_b} leads to a contradiction. Let $i$ be such that $\hat{b}_k + \hat{\theta}_k^T \widetilde{X}_i \geq \hat{b}_l + \hat{\theta}_l^T \widetilde{X}_i$. Then
\[\hat{b}_l - \hat{b}_k \leq \bigl( \hat{\theta}_k - \hat{\theta}_l \bigr)^T \widetilde{X}_i \leq \max_{i = 1,\ldots, n_t} \bigl|(\hat{\theta}_k - \hat{\theta}_l \bigr)^T \widetilde{X}_i \bigr| .\]
Exchanging $k$ and $l$ yields
\[ |\hat{b}_l - \hat{b}_k| \leq \max_{1 \leq i \leq n_t} \bigl|(\hat{\theta}_k - \hat{\theta}_l \bigr)^T \widetilde{X}_i \bigr| \leq 
2 \max_{1 \leq i \leq n_t} \bigl|(\hat{\theta}_k - \hat{\theta}_l \bigr)^T (X_i - PX) \bigr| . \]
Let $X \sim X_1$ be independent from $D_{n_t}$.
For any $k,l$,
\begin{align*}
 |(\hat{t}_k - \hat{t}_l)(X)| &\leq |\hat{b}_l - \hat{b}_k| +  \bigl|(\hat{\theta}_k - \hat{\theta}_l)^T (PX - \frac{1}{n_t} \sum_{i = 1}^{n_t} X_i)  \bigr| + \bigl|(\hat{\theta}_k - \hat{\theta}_l)^T (X - P X)  \bigr| \\
 &\leq 3 \max_{1 \leq i \leq n_t} \left| (\hat{\theta}_k - \hat{\theta}_l)^T (X_i - P X) \right| + \bigl|(\hat{\theta}_k - \hat{\theta}_l)^T (X - P X)  \bigr|.
\end{align*}
As $X$ is independent from $D_{n_t}$, conditionnally on $D_{n_t}$, by hypothesis 2,
\begin{align*}
\Norm{\hat{t}_k - \hat{t}_l}_{\psi_1,P} &\leq \frac{3}{\log 2} 
\max_{1 \leq i \leq n_t} \left| (\hat{\theta}_k - \hat{\theta}_l)^T (X_i - P X) \right| + \Norm{(\hat{\theta}_k - \hat{\theta}_l)^T (X - P X)}_{\psi_1,P} \\
&\leq \frac{3}{\log 2} 
\max_{1 \leq i \leq n_t} \left| (\hat{\theta}_k - \hat{\theta}_l)^T (X_i - P X) \right| + \kappa(K) P \left( \langle \hat{\theta}_k - \hat{\theta}_l , X - P X \rangle^2 \right)^{\frac{1}{2}}
\end{align*}
Hence, by lemma \ref{lem_ub_l2_l1},
\begin{align*}
\Norm{\hat{t}_k - \hat{t}_l}_{\psi_1,P} &\leq 
\frac{3}{\log 2} 
\max_{1 \leq i \leq n_t} \left| (\hat{\theta}_k - \hat{\theta}_l)^T (X_i - P X) \right| \\ 
&\quad + \mu_0 \kappa(K)^2 \log(\kappa(K)) P \left( \bigl| \langle \hat{\theta}_k - \hat{\theta}_l , X - P X \rangle \bigr| \right) \\
&\leq 
\frac{3}{\log 2} 
\max_{1 \leq i \leq n_t} \left| (\hat{\theta}_k - \hat{\theta}_l)^T (X_i - P X) \right| \\ 
&\quad 
+ \mu_0 \kappa(K)^2 \frac{\log^2(\kappa(K))}{|\log \log 2|} P \left( \bigl| \langle \hat{\theta}_k - \hat{\theta}_l , X - P X \rangle \bigr| \right) . 
\end{align*}
Thus, by the hypotheses of claim \ref{beta_bound},
\[ \mathbb{E} \left[ \hat{\beta}  \right] 
\leq \frac{6L}{\log 2} n_t^\alpha + \frac{2\mu_0 \lambda^2}{|\log \log 2|} L n_t^{1 + \alpha}.  \]
The result follows since for all $n_t \geq 3$,
$\frac{6L}{\log 2} n_t^\alpha \leq \frac{2L}{\log 2} n_t^{1 + \alpha}$.

\end{proof}

\subsection{Proving hypotheses \texorpdfstring{$H \left(\hat{w}_{i,1},\hat{w}_{i,2}, (\hat{t}_k)_{1 \leq k \leq K} \right)$}{H(wi1,wi2,(tk))}}
The following lemmas will be useful. 

\begin{lemma} \label{lem_max}
For any $(u,v,a,b) \in \mathbb{R}_+^4$,
\[ \max(u(a+b),v(a+b)^2) \leq 
\left( \max(\sqrt{ua},\sqrt{v} a) + \max(\sqrt{ub},\sqrt{v} b) \right)^2. \]
\end{lemma}

\begin{proof}
\begin{align*}
\left( \max(\sqrt{ua},\sqrt{v} a) + \max(\sqrt{ub},\sqrt{v} b) \right)^2 
&= \max(ua,va^2) + \max(ub,vb^2) \\ 
&\quad + 2 \max(\sqrt{ua},\sqrt{v} a) \max(\sqrt{ub},\sqrt{v} b) \\
&\geq  \max(u(a+b),v(a+b)^2).
\end{align*}
\end{proof}
 
\begin{claim} \label{lem_lb_loss}
Let $\xloss{u} = P \left[ \phi_c (Y - u) - \phi_c(Y) |X  \right]$.
Let $s(X) \in \argmin_{u \in \mathbb{R}} \xloss{u}$; $s$ is a risk minimizer.
Under hypothesis \textbf{(Lcs)}, almost surely, for any $u \in \mathbb{R}$,
\[ s(X) - \frac{c}{2} \leq u \leq s(X) + \frac{c}{2} \implies \frac{d^2}{du^2} \xloss{u} \geq \eta.  \]
As a result, for any $u \in \mathbb{R}$,
\begin{align*}
\xloss{u} - \xloss{s(X)} &\geq \frac{\eta}{2} (u - s(X))^2 \text{ if }
|u - s(X)| \leq \frac{c}{2} \\
&\geq \frac{\eta c}{4} |u - s(X)| \text{ if }
|u - s(X)| \geq \frac{c}{2}.
\end{align*}
\end{claim}  
 
 \begin{proof}
 Recall that
 \[ \phi_c(x) = \frac{x^2}{2} \mathbb{I}_{|x| \leq c} + c \bigl( |x| - \frac{c}{2} \bigr) \mathbb{I}_{|x| > c}. \]
Then $\phi_c'(x) = \text{sgn}(x)(|x| \wedge c)$ and $\phi_c''(x) = \mathbb{I}_{|x| \leq c}$. By differentiating under the expectation,
almost surely, for any $u$ such that $|u - \bayes(X)| \leq \frac{c}{2}$, 
\begin{align*} 
 \frac{d^2}{du^2} \xloss{u} &=
 \partial_u^2 P \left[ \phi_c(Y-u)  - \phi_c(Y) | X \right] \\
 &= P \left[ \phi_c''(Y-u) | X \right] \\
 &= P \left[ |Y - u| \leq c | X \right] \\
 &\geq P \left[ |Y - \bayes(X)| + |u - \bayes(X)| \leq c | X \right] \\
 &\geq P \left[ |Y - \bayes(X)| \leq \frac{c}{2} | X \right] \\
 &\geq \eta.
\end{align*}
This proves the first equation.
Since $\bayes(X)$ is a global minimum, it follows that, for any $u \in \left[ \bayes(X) - \frac{c}{2}; \bayes(X) + \frac{c}{2} \right]$, 
\[
\xloss{u} - \xloss{s(X)}  \geq \frac{\eta (u-\bayes(X))^2}{2}.\]
Because $\xloss{\cdot}$ is convex, for any $u$ such that $u \geq s(X) + \frac{c}{2}$,
\begin{align*}
\xloss{u} - \xloss{s(X)} &\geq (u - s(X)) \frac{\xloss{s(X) + \frac{c}{2}} - \xloss{s(X)}}{\frac{c}{2}} \\
&\geq (u - s(X)) \frac{2}{c} \frac{\eta c^2}{8} \\
&\geq \frac{\eta c}{4} (u - s(X)).
\end{align*}
Similarly, for $u < s(X) - \frac{c}{2}$, 
$\xloss{u} - \xloss{s(X)} \geq \frac{\eta c}{4} (s(X) - u)$.
This proves the lemma.
 \end{proof}
We now relate the $L^2$ norm to the excess risk in the following Proposition.

\begin{proposition} \label{prop_self_bound}
Let $(X,Y) \sim P$ be random variables. Let $\phi_c$ be the Huber loss with parameter $c > 0$.
Assume that $P$ satisfies hypothesis \textbf{(Lcs)}.
Let $(f_1,f_2) : \cX \rightarrow \mathbb{R}^2$ be measurable functions.
If for some $r > 2$, $\Norm{f_1 - f_2}_{r,P} \leq \kappa_r \Norm{f_1 - f_2}_{2,P}$, then
\[ \Norm{f_1 - f_2}_{2,P}^2 \leq \left(w_0(r, \kappa_r, \sqrt{\loss{f_1}}) + w_0(r, \kappa_r, \sqrt{\loss{f_2}}) \right)^2,  \]
where $w_0(r, \kappa_r,x) = \max \left( \frac{2\sqrt{2}}{\sqrt{\eta}} x,  
 \frac{8}{\eta c} 2^{\frac{r-1}{r-2}} \kappa_r^{\frac{r}{r-2}} x^2  \right)$.
 
In particular, there exists a constant $\mu_3$ such that, whenever $\Norm{f_1 - f_2}_{\psi_1,P} \leq \kappa \Norm{f_1 - f_2}_{2,P}$ for some $\kappa \geq 2$,
 \begin{equation} \label{w1_bound_hub}
 c^2 \Norm{f_1 - f_2}_{2,P}^2 \leq \left(w_1(\kappa, \sqrt{\loss{f_1}}) + 
 w_1(\kappa, \sqrt{\loss{f_2}}) \right)^2,  
\end{equation}
 where $w_1(\kappa,x) = \max \left( \frac{2\sqrt{2} c}{\sqrt{\eta}} x,  
 \frac{\mu_3}{\eta} \kappa \log(\kappa) x^2  \right)$.
 One can take $\mu_3 = 16 e^2 \times \frac{3 + \log 2}{\log 2} \times \sup_{u \geq 3} \exp \left( \frac{\log(u)}{u - 2} \right)$.
\end{proposition}

\begin{proof}
Let $f_1,f_2$ satisfy the hypotheses of proposition \ref{prop_self_bound}.
Let $U = f_1(X), V = f_2(X), S = s(X)$ where
\[ s(X) \in \argmin_{u \in \mathbb{R}} P \left[ \phi_c (Y - u) - \phi_c(Y) |X\right]. \]
Let 
\[ Z = P \left[ \phi_c(Y - U) + \phi_c(Y - V) - 2 \phi_c(Y - S) | X \right] . \]
Notice that in the notation of claim \ref{lem_lb_loss}, $Z = \xloss{U} + \xloss{V} - 2 \xloss{S}$ and in particular, $P[Z] = \loss{f_1} + \loss{f_2}$. 
Define the event $A = \{ |U - S| \leq \frac{c}{2}, |V - S| \leq \frac{c}{2} \}$. By claim \ref{lem_lb_loss},
\begin{align*}
(U - V)^2 \mathbb{I}_A 
&\leq 2 \left[ (U - S)^2 + (V-S)^2 \right] \mathbb{I}_{|U - S| \leq \frac{c}{2}} \mathbb{I}_{|V - S| \leq \frac{c}{2}} \\
&\leq \frac{4}{\eta} Z \mathbb{I}_A . \numberthis \label{eq_ub_var_near} 
\end{align*}
Let $r > 2$.
By lemma \ref{lem_holder_L2_L1},
\begin{align*}
P \left[(U - V)^2 \mathbb{I}_{A^c} \right] 
&\leq P \left[|U - V| \mathbb{I}_{A^c} \right]^{\frac{r-2}{r-1}} 
 P \left[|U - V|^r \mathbb{I}_{A^c} \right]^{\frac{1}{r-1}} \\ 
 &\leq P \left[|U - V| \mathbb{I}_{A^c} \right]^{\frac{r-2}{r-1}}
 P \left[|U - V|^r \right]^{\frac{1}{r-1}} \\
 &= P \left[|U - V| \mathbb{I}_{A^c} \right]^{\frac{r-2}{r-1}} 
 \Norm{f_1 - f_2}_{r,P}^{\frac{r}{r-1}} \\
 &\leq P \left[|U - V| \mathbb{I}_{A^c} \right]^{\frac{r-2}{r-1}}  
 \kappa_r^{\frac{r}{r-1}} \Norm{f_1 - f_2}_{2,P}^{\frac{r}{r-1}}.
\end{align*}
By definition, on $A^c$, $\max(|U - S|, |V - S|) \geq \frac{c}{2}$,
therefore by lemma \ref{lem_lb_loss},
\[|U - V| \mathbb{I}_{A^c} \leq 2 \max(|U - S|, |V - S|)\mathbb{I}_{A^c} 
\leq \frac{8}{\eta c} Z\mathbb{I}_{A^c}.\]
It follows that
\begin{equation} \label{eq_ub_var_far}
P \left[(U - V)^2 \mathbb{I}_{A^c} \right] \leq  
\left( \frac{8}{\eta c} \right)^{\frac{r-2}{r-1}} P[Z]^{\frac{r-2}{r-1}} \kappa_r^{\frac{r}{r-1}} \Norm{f_1 - f_2}_{2,P}^{\frac{r}{r-1}}. 
\end{equation}
From equations \eqref{eq_ub_var_near} and \eqref{eq_ub_var_far},
it follows that 
\begin{align*}
\Norm{f_1 - f_2}_{2,P}^2 &= P \left[(U - V)^2 \right] \\
&= P \left[(U - V)^2 \mathbb{I}_{A} \right] + P \left[(U - V)^2 \mathbb{I}_{A^c} \right] \\
&\leq  \frac{4}{\eta} P [ Z \mathbb{I}_A] + \left( \frac{8}{\eta c} \right)^{\frac{r-2}{r-1}} P[Z]^{\frac{r-2}{r-1}} \kappa_r^{\frac{r}{r-1}} \Norm{f_1 - f_2}_{2,P}^{\frac{r}{r-1}} \\
&\leq 2 \max \left( \frac{4}{\eta} P[Z],  
\left( \frac{8}{\eta c} \right)^{\frac{r-2}{r-1}} P[Z]^{\frac{r-2}{r-1}} \kappa_r^{\frac{r}{r-1}} \Norm{f_1 - f_2}_{2,P}^{\frac{r}{r-1}} \right).
\end{align*}
Therefore, either $\Norm{f_1 - f_2}_{2,P}^2 \leq 
\frac{8}{\eta} \left[ \loss{f_1} + \loss{f_2} \right] $
or
\begin{align*}
&\Norm{f_1 - f_2}_{2,P}^2  \leq 2
\left( \frac{8}{\eta c} \right)^{\frac{r-2}{r-1}} P[Z]^{\frac{r-2}{r-1}} \kappa_r^{\frac{r}{r-1}} \Norm{f_1 - f_2}_{2,P}^{\frac{r}{r-1}} \\
&\quad \iff \Norm{f_1 - f_2}_{2,P}^{\frac{r-2}{r-1}}  \leq 
2\kappa_r^{\frac{r}{r-1}} \left( \frac{8}{\eta c} P[Z] \right)^{\frac{r-2}{r-1}} \\
&\quad \iff \Norm{f_1 - f_2}_{2,P}^2 \leq 4^{\frac{r-1}{r-2}}
\kappa_r^{\frac{2r}{r-2}} \left( \frac{8}{\eta c} P[Z] \right)^2 \\
&\quad \iff \Norm{f_1 - f_2}_{2,P}^2 \leq 4^{\frac{r-1}{r-2}} \kappa_r^{\frac{2r}{r-2}} 
\left( \frac{8}{\eta c}\right)^2 \left[ \loss{f_1} + \loss{f_2} \right]^2 
\end{align*}
In either case,
\[ \Norm{f_1 - f_2}_{2,P}^2 \leq \max \left( \frac{8}{\eta} \left[ \loss{f_1} + \loss{f_2} \right], 4^{\frac{r-1}{r-2}} \kappa_r^{\frac{2r}{r-2}} 
\left( \frac{8}{\eta c}\right)^2 \left[ \loss{f_1} + \loss{f_2} \right]^2  \right). \]
Finally, by lemma \ref{lem_max},
\begin{equation} \label{eq_ub_L2_w0}
\Norm{f_1 - f_2}_{2,P}^2 \leq \left(w_0(r, \kappa_r, \sqrt{\loss{f_1}}) + 
 w_0(r, \kappa_r, \sqrt{\loss{f_2}}) \right)^2,
\end{equation}
 where $w_0(r,\kappa_r,x) = \max \left( \frac{2\sqrt{2}}{\sqrt{\eta}} x,  
 \frac{8}{\eta c} 2^{\frac{r-1}{r-2}} \kappa_r^{\frac{r}{r-2}} x^2  \right)$.
 This proves the first equation.
Let now $r = 3 + \log(\kappa) \geq 3$. By lemma \ref{lem_ub_Lr_psi1},
\begin{align*}
2^{\frac{r-1}{r-2}} \kappa_r^{\frac{r}{r-2}} &\leq 2^{\frac{r-1}{r-2}} 2^{\frac{1}{r-2}} r^{\frac{r}{r-2}} \kappa^{\frac{r}{r-2}} \\
&\leq 8 r \kappa r^{\frac{2}{r-2}} \kappa^{\frac{2}{r-2}} \\
&\leq 8 (3 + \log \kappa) \kappa r^{\frac{2}{r-2}} 
\exp \left( \log(\kappa) \frac{2}{1 + \log \kappa} \right) \\
&\leq  72 (3 + \log \kappa) \kappa  e^2
\end{align*}
since $r \mapsto r^{\frac{2}{r-2}}$ decreases on $[3; + \infty[$, as shown in the proof of lemma \ref{lem_ub_l2_l1}.
Let $\mu_3 = 576 e^2 \times \left(1 + \frac{3}{|\log \log 2|} \right).$
Then for all $\kappa \geq 2$,
\[ 2^{\frac{r-1}{r-2}} \kappa_r^{\frac{r}{r-2}} \leq \frac{\mu_3}{8} \kappa \log(\kappa). \]
It follows from equation \eqref{eq_ub_L2_w0} that
\[c^2 \Norm{f_1 - f_2}_{2,P}^2 \leq \left(w_1(\kappa, \sqrt{\loss{f_1}}) + 
 w_1(\kappa, \sqrt{\loss{f_2}}) \right)^2,  \]
where $w_1(\kappa,x) = \max \left( \frac{2 \sqrt{2} c}{\sqrt{\eta}} x,  
 \frac{\mu_3}{\eta} \kappa \log(\kappa) x^2  \right)$.
\end{proof}

We are now ready to obtain functions $(\hat{w}_{i,j})_{(i,j) \in \{1;2 \}^2}$ such that $H \left(\hat{w}_{i,1},\hat{w}_{i,2},(\hat{t}_k)_{1 \leq k \leq K} \right)$ holds.
In the following, fix $K \in [|1;n_t |]$ and write $\kappa = \kappa(K)$ for short.
Because the Huber loss $\phi_c$ is $c-$Lipschitz,
\[ \forall u,v \in \mathbb{R}, \left|\phi_c(Y - u) - \phi_c(Y - v) \right| \leq c |u-v|. \]
Therefore, for all $r \geq 2$,
\[ P \left[\bigl(\phi_c(Y - \hat{t}_k(X)) - \phi_c(Y - \hat{t}_l(X)) \bigr)^r \right] \leq c^r P \left( \bigl|\hat{t}_k(X) - \hat{t}_l(X)  \bigr|^r \right). \]
Let $\mu_4 = \frac{4}{\log 2} + 4$.
By claim \ref{claim_l2_linf}, 
$\Norm{\hat{t}_k - \hat{t}_l}_{\psi_1,P} \leq \sqrt{2} \kappa \Norm{\hat{t}_k - \hat{t}_l}_{2,P}$, hence
by lemma \ref{lem_bernstein_psi1}, since $\kappa \geq \frac{1}{\log 2} \geq \sqrt{2}$,
\begin{align*}
 &P\left[\bigl(\phi_c(Y - \hat{t}_k(X)) - \phi_c(Y - \hat{t}_l(X)) \bigr)^r \right] \\
&\leq r! \left(c^2 \Norm{\hat{t}_k - \hat{t}_l}^2_{2,P} \right) \left(\mu_4 c \log(\sqrt{2} \kappa) \Norm{\hat{t}_k - \hat{t}_l}_{\psi_1,P}\right)^{r-2}  \\ 
&\leq r! \left(c^2 \Norm{\hat{t}_k - \hat{t}_l}^2_{2,P} \right) \left(\mu_4 c \log(\sqrt{2} \kappa) \sqrt{2} \kappa \Norm{\hat{t}_k - \hat{t}_l}_{2,P}\right)^{r-2} .
\end{align*}
 Using the notation of Proposition \ref{prop_self_bound}, let 
\begin{equation} \label{def_wA_hub}
w_A(x) = w_1(\sqrt{2} \kappa(K),x) 
= \max \left( \frac{2 \sqrt{2} c}{\sqrt{\eta}} x,  
 \frac{\mu_3}{\eta} \sqrt{2} \kappa \log(\sqrt{2} \kappa) x^2  \right).
\end{equation}  
 By Proposition \ref{prop_self_bound},
\begin{align*}
&P\left[\bigl(\phi_c(Y - \hat{t}_k(X)) - \phi_c(Y - \hat{t}_l(X)) \bigr)^r \right] \\
&\leq \left(w_A(\sqrt{\loss{\hat{t}_k}}) + w_A(\sqrt{\loss{\hat{t}_l}}) \right)^2 \\ 
&\quad \times \left( \mu_4 \sqrt{2} \kappa \log (\sqrt{2} \kappa) \bigl( w_A(\sqrt{\loss{\hat{t}_k}}) + w_A(\sqrt{\loss{\hat{t}_l}}) \bigr) \right)^{r-2},
\end{align*}
which proves $H \left(w_A, \mu_4 \sqrt{2} \kappa \log(\sqrt{2} \kappa) w_A, (\hat{t}_k)_{1 \leq k \leq K} \right)$. 
Now by Definition \ref{norm_cst} and lemma \ref{lem_ub_Lr_psi1}, 
\begin{align*}
 P \left[\bigl(\phi_c(Y - \hat{t}_k(X)) - \phi_c(Y - \hat{t}_l(X)) \bigr)^r \right]&\leq c^r 
 P \left[ |\hat{t}_k - \hat{t}_l|^r \right] \\ 
&\leq 2 r! c^r \Norm{\hat{t}_k - \hat{t}_l}_{\psi_1,P}^r \\
&\leq 2 r! c^r \hat{\beta}^r,
\end{align*}
which proves $H \left(\frac{c \hat{\beta}}{\sqrt{2}}, \frac{c \hat{\beta}}{2}, (\hat{t}_k)_{1 \leq k \leq K} \right)$.

\subsection{Conclusion of the proof}
We have proved that $H \left(w_A,\mu_4 \sqrt{2} \kappa \log(\sqrt{2} \kappa) w_A, (\hat{t}_k)_{1 \leq k \leq K} \right)$ and $H \left(\frac{c \hat{\beta}}{\sqrt{2}}, \frac{c \hat{\beta}}{2}, (\hat{t}_k)_{1 \leq k \leq K} \right)$ hold, where  $w_A$ is defined in Proposition \ref{prop_self_bound}.
It remains to apply \cite[Theorem 17]{agghoo_rkhs} and to express the remainder term as a simple function of $c, n_v, n_t, \kappa,L, K$ and $\alpha$. 
We recall here the definition of the operator $\delta$ used in the statement of that theorem.

\begin{Definition} \label{def_delta}
 For any function $h: \mathbb{R}_+ \mapsto \mathbb{R}_+$ and any $\xi > 0$, let
 \[ \delta(h,\xi) = \inf \{ x \in \mathbb{R}_+ : \forall u \geq x, h(u) \leq \xi u^2 \}. \]
\end{Definition}

The following lemma will facilitate the computation of $\delta(w_A,\cdot)$.

\begin{lemma} \label{delta_calc}
 Let $r > 0, s > 0$ and $h_{r,s}(x) = (\sqrt{r}x ) \vee s x^2$.
 Then $\delta(h_{r,s}, \xi) < \infty$ if and only if $\xi \geq s$ and then
 $\delta(h_{r,s}, \xi) = \frac{\sqrt{r}}{\xi}$.
 \end{lemma}
 
\begin{proof}
To find $\delta(h_{r,s}, \xi)$, notice that given the definition of $\delta(h_{r,s}, \xi)$, 
the condition $s \leq \xi$ is obviously necessary for the infimum to be finite.
Assume now that $\xi \geq s$. For any $u \geq \frac{\sqrt{r}}{\xi}$, then $\xi u^2 \geq \sqrt{r}u$ as well as $\xi u^2 \geq s u^2$ (since we assumed $\xi \geq s$), therefore $\xi u^2 \geq h_{r,s}(u)$. 
Thus by definition, $\delta(h_{r,s}, \xi)\leq \frac{\sqrt{r}}{\xi}$ (in particular, $\delta(h_{r,s}, \xi)$ is finite).
Furthermore, by definition
of $\delta(h_{r,s}, \xi)$, $\sqrt{r}\delta(h_{r,s}, \xi) \leq \xi \delta(h_{r,s}, \xi)^2$, 
that is $\delta(h_{r,s}, \xi) \geq \frac{\sqrt{r}}{\xi}$.
\end{proof}

The following claim can now be proved.

\begin{claim} \label{claim_or_ineq_K}
Assume that hypotheses \textbf{(Reg-}$\cT$ \textbf{)} and \textbf{(Lcs)} hold.
 If $K \in \{3,\ldots,e^{\sqrt{n_v}} \}$ and $b > 1$ are such that 
 \begin{equation} 
 \sqrt{2} \kappa(K) \log \left(\sqrt{2} \kappa(K) \right) \leq \frac{\eta}{ \mu_3 \vee \mu_4} \sqrt{\frac{n_v}{8b \log K}}, \label{hyp_bound_K}
\end{equation}
then applying Agghoo to the collection $\left( \learnrule_k \right)_{1 \leq k \leq K}$ yields
the following oracle inequality.
\begin{align*}
(1-\theta) \mathbb{E}[\loss{\ERMag{\cT}}] \leq  (1+\theta)\mathbb{E}[\min_{1 \leq k \leq K} \loss{\hat{t}_k}] + 54 \theta b \frac{c^2 \log K}{\eta n_v} +  \frac{7 \log K}{\theta K^{\theta^2 b - 1}} \frac{\mu_1 c L n_t^{1 + \alpha}}{\sqrt{n_v}}. \label{or_ineq_K}
\end{align*}
\end{claim}

\begin{proof}
 Theorem \cite[Theorem 17]{agghoo_rkhs} applies with $\hat{w}_{1,1} = \frac{c\hat{\beta}}{\sqrt{2}}, \hat{w}_{1,2} = \frac{c\hat{\beta}}{2}, \hat{w}_{2,1} = w_A, \hat{w}_{2,2} = \mu_4 \sqrt{2} \kappa \log(\sqrt{2} \kappa) w_A$, $x = (\theta^2 b - 1)\log K$ and it remains to bound the remainder terms $(R_{2,i})_{1 \leq i \leq 4}$.
Now assume that equation \eqref{hyp_bound_K} holds.
\paragraph{Bound on $R_{2,1}(\theta) = \sqrt{2} \theta \mathbb{E} \left[ \delta^2\left(w_A, \frac{\theta}{2} \sqrt{\frac{n_v}{\theta^2 b \log K}}\right) \right]$} \mbox{}

By \eqref{def_wA_hub}, we can apply lemma \ref{delta_calc} with $s = \frac{\mu_3}{\eta} \sqrt{2} \kappa \log(\sqrt{2} \kappa)$, $r = \frac{8c^2}{\eta}$ and 
$\xi = \frac{1}{2} \sqrt{\frac{n_v}{b \log K}}$.
By \eqref{hyp_bound_K},
\[ s = \frac{\mu_3}{\eta} \sqrt{2} \kappa \log(\sqrt{2} \kappa) \leq \sqrt{\frac{n_v}{4 b \log K}} = \xi. \] It follows by lemma \ref{delta_calc} that
\[
 \delta \left( w_A,  \sqrt{\frac{n_v}{4 b \log K}}\right) = \frac{2\sqrt{2} c}{\sqrt{\eta}} \sqrt{\frac{4b \log K}{n_v}}.
\]
Hence,
\begin{equation}\label{eq_21}
 R_{2,1}(\theta) \leq \sqrt{2} \theta \frac{8c^2}{\eta} \frac{4b \log K}{n_v} \leq 46 \theta b \frac{c^2 \log K}{\eta n_v} 
\end{equation}
\paragraph{Bound on $R_{2,2}(\theta) =  \frac{\theta^2}{2} \mathbb{E} \left[ \delta^2 \left(\mu_4 \sqrt{2} \kappa \log(\sqrt{2} \kappa) w_A,\frac{\theta^2}{4} \frac{n_v}{\theta^2 b \log K} \right) \right]$} \mbox{}

By \eqref{def_wA_hub}, we can apply lemma \ref{delta_calc} with $s = \frac{ \mu_3 \mu_4 (\sqrt{2} \kappa \log (\sqrt{2} \kappa))^2}{\eta}$, $r = \frac{8 \mu_4^2 c^2}{\eta} (\sqrt{2} \kappa \log(\sqrt{2} \kappa))^2$ and $\xi = \frac{n_v}{4b \log K}$.
By \eqref{hyp_bound_K} and since $\eta \leq 1$,
\begin{align*}
s &= \frac{\mu_3 \mu_4 (\sqrt{2} \kappa \log(\sqrt{2} \kappa))^2}{\eta} \\
&\leq \eta \left( \frac{\mu_3 \vee \mu_4}{\eta} \sqrt{2} \kappa \log(\sqrt{2} \kappa) \right)^2 \\
&\leq \frac{n_v}{4b \log K}.
\end{align*}
Therefore, 
\begin{align*}
 \delta \left((\mu_4 \sqrt{2} \kappa \log (\sqrt{2}\kappa)) w_A,\frac{\theta^2}{4} \frac{n_v}{\theta^2 b \log K} \right)  &\leq \frac{2 \sqrt{2} c \mu_4}{\sqrt{\eta}} \sqrt{2} \kappa \log(\sqrt{2} \kappa) \frac{4b \log K}{n_v} \text{ by lemma \ref{delta_calc} } \\
 &\leq \frac{4 c \mu_4}{\mu_3 \vee \mu_4} \sqrt{\frac{\eta b \log K}{n_v} } \text{ by \eqref{hyp_bound_K}}. 
\end{align*}
Hence, since $\theta, \eta \in [0;1]$,
\begin{equation}
 R_{2,2}(\theta) \leq \frac{\theta^2}{2} 16 c^2 \frac{\eta b \log K}{n_v} \leq 8 \theta b \frac{c^2 \log K}{n_v}. \label{eq_22}
\end{equation}

\paragraph{Bound on $R_{2,3}(\theta) = \frac{1}{K^{\theta^2 b - 1}} \left( \theta + \frac{2\bigl[1 + \log(K) \bigr]}{\theta} \right) \mathbb{E} \left[ \delta^2 \bigl(\frac{c\hat{\beta}}{\sqrt{2}}, \sqrt{n_v} \bigr) \right]$} \mbox{} 

$x \rightarrow \frac{c\hat{\beta}}{\sqrt{2} x}$ is non-increasing, therefore, $\delta \bigl(\frac{c\hat{\beta}}{\sqrt{2}}, \sqrt{n_v} \bigr)$ is the unique nonnegative solution to the equation 
\[ \frac{c\hat{\beta}}{\sqrt{2}} = \sqrt{n_v} x^2 \iff x^2 = \frac{c\hat{\beta}}{\sqrt{2 n_v}}.\]
It follows that 
%
\begin{equation} \label{sol_w11_hub} 
 \delta^2 \bigl(\frac{c\hat{\beta}}{\sqrt{2}}, \sqrt{n_v} \bigr) =  \frac{c\hat{\beta}}{\sqrt{2 n_v}}.
\end{equation}
Since $K \geq 3$ by assumption, $\log K \geq 1$ and
\[
 \theta + \frac{2(1 + \log K)}{\theta} \leq \frac{5 \log K}{\theta}. 
\]
By equation \eqref{sol_w11_hub},
\begin{equation} \label{eq_23}
 R_{2,3}(\theta) \leq \frac{4 \log K}{\theta K^{\theta^2 b - 1}} \frac{c \mathbb{E}[\hat{\beta}]}{\sqrt{n_v}}.
\end{equation}

\paragraph{Bound on $R_{2,4}(\theta) = \frac{1}{K^{\theta^2 b -1}}\left( \theta + \frac{2(1 + \log K) + \log^2 K}{\theta} \right) \mathbb{E}\left[ \delta^2 \bigl(\frac{c\hat{\beta}}{2}, n_v \bigr) \right]$} \mbox{}

$\delta^2 \bigl(\frac{c\hat{\beta}}{2}, n_v \bigr)$ is the unique nonnegative solution to the equation 
\[ \frac{c\hat{\beta}}{2} = n_v x^2 \iff x^2 = \frac{c\hat{\beta}}{2 n_v},\]
which yields
\[ \delta^2 \bigl(\frac{c\hat{\beta}}{2}, n_v \bigr) = \frac{c\hat{\beta}}{2 n_v}.  \]
Moreover,
\begin{align*}
 \theta + \frac{2(1 + \log K) + \log^2 K}{\theta} &\leq \frac{5}{\theta} \log K + \frac{\log^2 K}{\theta} \\
 &\leq \frac{6 \log^2 K}{\theta} \text{ since } K \geq 3.
\end{align*}
Therefore, since by assumption $K \leq n_t \leq e^{\sqrt{n_v}}$,
\begin{equation} \label{eq_24}
  R_{2,4}(\theta) \leq \frac{6 \log^2 K}{\theta K^{\theta^2 b -1}} \frac{c \mathbb{E}[\hat{\beta}]}{2 n_v} \leq \frac{3 \log K}{\theta K^{\theta^2 b - 1}} \frac{c \mathbb{E}[\hat{\beta}]}{\sqrt{n_v}}.
\end{equation}

\paragraph{Conclusion}

Summing up equations \eqref{eq_21}, \eqref{eq_22}, \eqref{eq_23} and \eqref{eq_24}, \cite[Theorem 17]{agghoo_rkhs} implies that assuming equation \eqref{hyp_bound_K} holds for $K$, for all $\theta \in \left[ \frac{1}{\sqrt{b}}; 1 \right]$,
\begin{equation}
 (1-\theta) \mathbb{E}[\loss{\ERMag{\cT}}] \leq  (1+\theta)\mathbb{E}[\min_{1 \leq k \leq K} \loss{\hat{t}_k}] + 54 \theta b \frac{c^2 \log K}{\eta n_v}  + \frac{7 \log K}{\theta K^{\theta^2 b - 1}} \frac{c \mathbb{E}[\hat{\beta}]}{\sqrt{n_v}}.
\end{equation}
By hypothesis \eqref{hyp_bound_K} and since $n_t \geq n_v$ by hypothesis \textbf{(Reg-}$\cT$ \textbf{)}, 
\[ \kappa(K) \log(\kappa(K)) \leq \frac{\sqrt{n_v}}{4(\mu_3 \vee \mu_4)} 
\leq \frac{\sqrt{n_t}}{4(\mu_3 \vee \mu_4)}, \]
hence claim \ref{beta_bound} applies with $\lambda = \frac{1}{4(\mu_3 \vee \mu_4)}$. Thus,
\[ \mathbb{E}[\hat{\beta}] \leq \mu_1 L n_t^{1 + \alpha} \text{ where } \mu_1 = \frac{2}{\log 2} + \frac{\mu_0}{8(\mu_3 \vee \mu_4)^2 |\log \log 2|}. \]
It follows that
\begin{align*}
(1-\theta) \mathbb{E}[\loss{\ERMag{\cT}}] \leq  (1+\theta)\mathbb{E}[\min_{1 \leq k \leq K} \loss{\hat{t}_k}] + 54 \theta b \frac{c^2 \log K}{\eta n_v} 
+ \frac{7 \log K}{\theta K^{\theta^2 b - 1}} \frac{\mu_1 c L n_t^{1 + \alpha}}{\sqrt{n_v}} 
\end{align*}
This proves Claim \ref{claim_or_ineq_K}.
\end{proof}

Theorem \ref{agcv_hub} can now be derived from claim \ref{claim_or_ineq_K}.
Let $\theta$ be such that $\theta \geq \mu_2 \sqrt{\alpha + 3} \frac{\nu_0}{\eta}$ for some numerical constant $\mu_2$, to be determined later.
Then $\nu_0 \leq \frac{\theta \eta}{\mu_2 \sqrt{\alpha + 3}}$, so by hypothesis \textbf{(Ni)},
\[ \kappa(K) \log(\kappa(K)) \leq \frac{\theta \eta}{\mu_2 \sqrt{\alpha + 3}} \sqrt{\frac{n_v}{\log (n_t \vee K)}}. \]
Letting $b = \frac{3 + \alpha}{\theta^2} \left(\frac{\log n_t}{\log K} \vee 1 \right)$, we can rewrite the above equation as
\[ \kappa(K) \log(\kappa(K)) \leq \frac{\eta}{\mu_2} \sqrt{\frac{n_v}{b \log K}}. \]
Since for any $x \geq \sqrt{2}$, $\tfrac{\sqrt{2}x \log(\sqrt{2}x)}{x \log x} = \sqrt{2} \bigl[1 + \tfrac{\log \sqrt{2}}{\log x} \bigr] \leq 2 \sqrt{2}$, and  $\kappa(K) \geq \frac{1}{\log 2} \geq \sqrt{2}$ by definition,
\[ \sqrt{2} \kappa(K) \log(\sqrt{2} \kappa(K)) \leq \frac{2\sqrt{2} \eta}{\mu_2} \sqrt{\frac{n_v}{b \log K}}.  \]
Let now $\mu_2 = 8(\mu_3 \vee \mu_4)$, so that equation \eqref{hyp_bound_K} holds. By claim \ref{claim_or_ineq_K},
\[ (1-\theta) \mathbb{E}[\loss{\ERMag{\cT}}] \leq  (1+\theta)\mathbb{E}[\min_{1 \leq k \leq K} \loss{\hat{t}_k}] + 54 \theta b \frac{c^2 \log K}{\eta n_v} + \frac{7 \log K}{\theta K^{\theta^2 b - 1}} \frac{\mu_1 c L n_t^{1 + \alpha}}{\sqrt{n_v}}.\]
Since $b = \frac{3 + \alpha}{\theta^2} \left(\frac{\log n_t}{\log K} \vee 1 \right)$, $K^{\theta^2 b - 1} \geq n_t^{2 + \alpha}$ and $\theta b \log K \leq \frac{3 + \alpha}{\theta} \log(n_t \vee K)$,
which proves Theorem \ref{agcv_hub}.

\section{Applications of Theorem \ref{agcv_hub}}

\subsection{Gaussian vectors}

\begin{proof}
For any $\theta \in \mathbb{R}^d$, $Z = \langle \theta, X - PX \rangle$ is a centered gaussian random variable.
By homogeneity of norms, the quotient $\frac{\Norm{Z}_{L^{\psi_1}}}{\Norm{Z}_{L^2}}$ does not depend on the scale parameter $\sigma$; it is therefore
a numerical constant; moreover one can check that for $Z \sim \mathcal{N}(0;1)$, $ \frac{\Norm{Z}_{L^{\psi_1}}}{\Norm{Z}_{L^2}} = 
\Norm{Z}_{L^{\psi_1}} = \sqrt{2 \log 2} \leq \frac{1}{\log 2}$. Thus, we can choose $\kappa(K) = \frac{1}{\log 2}$ so that
\begin{equation} \label{eq_ub_kappalogkappa}
\kappa(K) \log(\kappa(K)) < 0.6.
\end{equation}

It remains to prove point 2 of hypothesis \ref{hyp_sparse_reg} for some constant $\alpha$. Let $k \in \{1,\ldots,K\}$.
Let $\hat{q}_{k,R}, \hat{\theta}_{k,R}$ be such that $\learnrule_{k,R}^{lasso}(D_{n_t})(x) = \hat{q}_{k,R} + \langle \hat{\theta}_{k,R}, x \rangle $. 
By the inequality $c |u| \leq \frac{c^2}{2} + \phi_c(u)$,
for any $q \in \mathbb{R}$,
\begin{align*}
\frac{1}{n_t} \sum_{i = 1}^{n_t} |\hat{q}_{k,R} - q + \langle \hat{\theta}_{k,R}, X_i \rangle| &\leq \frac{1}{n_t} \sum_{i = 1}^{n_t} |Y_i - q| + \frac{1}{n_t} \sum_{i = 1}^{n_t} |Y_i - \hat{q}_{k,R} - \langle \hat{\theta}_{k,R}, X_i \rangle| \\
&\leq \frac{1}{n_t} \sum_{i = 1}^{n_t} |Y_i - q| + \frac{c}{2} + 
\frac{1}{c n_t} \sum_{i = 1}^{n_t} \phi_c (Y_i - \hat{q}_{k,R} - \langle \hat{\theta}_{k,R}, X_i \rangle). 
\end{align*}
It follows by definition of $\hat{q}_{k,R},\hat{\theta}_{k,R}$ that
\begin{align*}
\frac{1}{n_t} \sum_{i = 1}^{n_t} |\hat{q}_{k,R} - q + \langle \hat{\theta}_{k,R}, X_i \rangle| &\leq \frac{1}{n_t} \sum_{i = 1}^{n_t} |Y_i - q| + \frac{c}{2} + 
\frac{1}{c n_t} \sum_{i = 1}^{n_t} \phi_c (Y_i - q) \\
&\leq \frac{2}{n_t} \sum_{i = 1}^{n_t} |Y_i - q| + \frac{c}{2}. \numberthis \label{eq_ub_enorm}
\end{align*}
On the other hand, letting $\bar{X}_{n_t} = \frac{1}{n_t} \sum_{i = 1}^{n_t} X_i$,
\begin{align*}
\frac{1}{n_t} \sum_{i = 1}^{n_t} |\hat{q}_{k,R} - q + \langle \hat{\theta}_{k,R}, X_i \rangle| &\geq \frac{1}{n_t} \sqrt{\sum_{i = 1}^{n_t} |\hat{q}_{k,R} - q + \langle \hat{\theta}_{k,R}, X_i \rangle|^2} \\
&\geq \frac{1}{n_t} \sqrt{\sum_{i = 1}^{n_t} \langle \hat{\theta}_{k,R}, X_i - \bar{X}_{n_t} \rangle^2} \\
&\geq \frac{1}{n_t} \max_{i \in \{1,\ldots,n_t\}} \left| \langle \hat{\theta}_{k,R}, X_i - \bar{X}_{n_t} \rangle \right|. \numberthis \label{eq_comp_enorms}
\end{align*}
For all $\theta \in \mathbb{R}^d$, let $\hat{N}(\theta) = \max_{i \in \{1,\ldots,n_t\}} \left| \langle \theta, X_i - \bar{X}_{n_t} \rangle \right|$.
Clearly, $\hat{N}$ is a semi-norm. Let $\Sigma = P[XX^T]$ be the covariance matrix of $X$. For all $I \subset \{1,\ldots,d\}$, let $E_I$ denote the vector space $\{ \theta \in \mathbb{R}^d : \forall i \notin I, \theta_i = 0 \}$. Let 
\[ \hat{\gamma}_I = \min_{\theta \in E_I : \theta^T \Sigma \theta = 1} \hat{N}(\theta). \]
Let finally $\hat{\gamma} = \min_{I \subset \{1,\ldots,d\}, |I| \leq \frac{n_t}{\log d}} \hat{\gamma}_I$.
Since by construction, 
\[ \Norm{\hat{\theta}_{k,R}}_0 \leq k \leq K \leq \frac{n_t}{\log d}, \] it follows from equations \eqref{eq_ub_enorm}, \eqref{eq_comp_enorms} and the definition of $\hat{\gamma}$ that
\begin{equation} \label{eq_ub_sigma_norm_thet}
\sqrt{\hat{\theta}_{k,R}^T \Sigma \hat{\theta}_{k,R}} \leq \frac{1}{\hat{\gamma}} \left( 2 \sum_{i = 1}^{n_t} |Y_i - q| + \frac{n_t c}{2} \right).
\end{equation}
By Hölder's inequality, for any $u > 0$,
\begin{align*}
\mathbb{E} \left[ \max_{1 \leq k \leq K}  E \bigl[ \bigl| \langle \hat{\theta}_{k,R}, X - PX  \rangle  \bigr|\bigr] \right] &\leq \mathbb{E} \left[ \max_{1 \leq k \leq K} \sqrt{\hat{\theta}_{k,R}^T \Sigma \hat{\theta}_{k,R}} \right] \\
&\leq n_t \Norm{\frac{1}{\hat{\gamma}}}_{L^{1+\frac{1}{u}}} \left( 2\Norm{Y_1 - q}_{L^{1 + u}} + \frac{c}{2} \right).  
\end{align*}
If $n_t \geq 4 + \frac{3}{u}$, then by lemma \ref{lem_lb_gam} below,
\begin{equation} \label{eq_ub_L1_thet}
\mathbb{E} \left[ \max_{1 \leq k \leq K}  E \bigl[ \bigl| \langle \hat{\theta}_{k,R}, X - PX  \rangle  \bigr| \bigr] \right] 
\leq \mu_7 n_t (\Norm{Y_1 - q}_{L^{1 + u}} \vee c). 
\end{equation}
for some numerical constant $\mu_7$.
For any $i \in \{1,\ldots,n_t \}$, the vector $X_i$
has components $X_{i,1},\ldots, X_{i,p}$.
For any $J \subset \{1,\ldots,d\}$,let 
$X_{i,J} = (X_{i,j})_{j \in J} \in \mathbb{R}^J$ 
and $\Sigma_{J J} = (\Sigma_{j, j'})_{j \in J, j' \in J}$.
By the Cauchy-Schwarz inequality, equation \eqref{eq_ub_sigma_norm_thet} and since $\Norm{\hat{\theta}_{k,R}}_0 \leq K$,
\begin{align*}
&\max_{1 \leq k \leq K} \max_{1 \leq i \leq n_t} 
\bigl| \langle \hat{\theta}_{k,R}, X_i - PX_i  \rangle  \bigr| \\ 
&\quad \leq \max_{1 \leq k \leq K} \sqrt{\hat{\theta}_{k,R}^T \Sigma \hat{\theta}_{k,R}} \times \max_{1 \leq i \leq n_t} 
\max_{J \subset \{1,\ldots,d\} : |J| \leq K} \Norm{\Sigma_{JJ}^{- \frac{1}{2}}(X_{i,J} - P X_{i,J})}_2 \\ 
&\quad \leq  \frac{1}{\hat{\gamma}} \left( 2 \sum_{i = 1}^{n_t} |Y_i - q| + \frac{n_t c}{2} \right) \times \max_{1 \leq i \leq n_t} \max_{J \subset \{1,\ldots,d\} : |J| \leq K} \Norm{\Sigma_{JJ}^{- \frac{1}{2}}(X_{i,J} - P X_{i,J})}_2.
\end{align*}
Let $r = 1 + \frac{u}{2}, r' = 1 + \frac{2}{u}$, $p = \frac{1+u}{r}$, $p' = \frac{1}{1 - \frac{1}{p}}$. 
Let
\begin{equation} \label{eq_def_RK}
\hat{R}_K = \max_{1 \leq i \leq n_t} \max_{J \subset \{1,\ldots,d\} : |J| \leq K} \Norm{\Sigma_{JJ}^{- \frac{1}{2}}(X_{i,J} - P X_{i,J})}_2.
\end{equation}
Then, by two applications of Hölder's inequality,
\begin{align*}
&\mathbb{E} \left[\max_{1 \leq k \leq K} \max_{1 \leq i \leq n_t} 
\bigl| \langle \hat{\theta}_{k,R}, X_i - P X_i  \rangle  \bigr| \right] \\
&\quad \leq \Norm{\frac{2}{\hat{\gamma}} \sum_{i = 1}^{n_t} |Y_i - q| + \frac{n_t c}{2 \hat{\gamma}}}_{L^r} \Norm{\hat{R}_K }_{L^{r'}} \\
&\quad \leq n_t \Norm{\frac{1}{\hat{\gamma}}}_{L^{p'r}} \left( 2\Norm{Y_1 - q}_{L^{pr}} + \frac{c}{2} \right) 
\Norm{\hat{R}_K}_{L^{r'}}.  
\end{align*}
By definition, $pr = 1 + u$, 
\begin{align*}
p'r &= \frac{r}{1 - \frac{r}{1 + u}} \\
&= \frac{1 + \frac{u}{2}}{1 - \frac{1 + \frac{u}{2}}{1 + u}} \\
&= \frac{2\left(1 + \frac{u}{2} \right) \left(1 + u \right)}{u} \\
&\leq 4 + \frac{2}{u}.
\end{align*}
Therefore, if $n_t \geq 13 + \frac{6}{u}$, then by lemma \ref{lem_lb_gam} below, for some constant $\mu_7$,
\begin{equation} \label{eq_ub_norm_pred_gauss}
\mathbb{E} \left[\max_{1 \leq k \leq K} \max_{1 \leq i \leq n_t} 
\bigl| \langle \hat{\theta}_{k,R}, X_i - PX_i  \rangle  \bigr| \right] 
\quad \leq n_t \mu_7 \left( \Norm{Y_1 - q}_{L^{1+u}} \vee c \right) \Norm{\hat{R}_K}_{L^{r'}}. 
\end{equation}
Let us now bound $\Norm{\hat{R}_K}_{L^{r'}}$, where we recall that $\hat{R}_K$ is given by equation \eqref{eq_def_RK}.
Since for any $i \in \{1,\ldots,n_t\}$, $J \subset \{1,\ldots,d\}$, $\Sigma_{JJ}^{- \frac{1}{2}}(X_{i,J} - P X_{i,J})$ is a standard normal vector of size $|J|$,  by the gaussian concentration inequality, there exists some constant $\mu$ such that
\begin{align*}
\Norm{\hat{R}_K}_{L^{r'}} &\leq \max_{J \subset \{1,\ldots,d\}: |J| \leq K} P\left[ \Norm{\Sigma_{JJ}^{- \frac{1}{2}}(X_{i,J} - P X_{i,J})}_2  \right] + \sqrt{\mu \left(\log n_t + \log \sum_{j \leq K} {d \choose j} \right)} + \sqrt{\mu r'} \\
&\leq \sqrt{K} + \sqrt{\mu \log n_t} + 
\sqrt{ \mu (1 + K \log d)}
 + \sqrt{\mu (1 + \tfrac{2}{u})}. 
\end{align*}
Since by assumption $n_t \geq 13 + \frac{6}{u}$ and $K \leq \frac{n_t}{\log d}$ and since $\log n_t \leq n_t$,
\begin{align*}
\Norm{\hat{R}_K}_{L^{r'}} &\leq (1 + \sqrt{\mu})\sqrt{n_t} + \sqrt{\mu(1 + n_t)} + \sqrt{\mu \frac{u + 2}{13 u + 6}} \sqrt{n_t} \\
&\leq (1 + 3\sqrt{\mu})\sqrt{n_t}.
\end{align*}
From equation \eqref{eq_ub_norm_pred_gauss}, we can conclude that for some constant $\mu_7' \geq \mu_7$,
\[ \mathbb{E} \left[\max_{1 \leq k \leq K} \max_{1 \leq i \leq n_t} 
\bigl| \langle \hat{\theta}_{k,R}, X_i - PX_i  \rangle  \bigr| \right] 
\leq \mu_7' \left( \Norm{Y_1 - q}_{L^{1+u}} \vee c \right) n_t^{\frac{3}{2}}. \]
Together with \eqref{eq_ub_L1_thet}, this proves point 2. of hypothesis \ref{hyp_sparse_reg}. 
with $\alpha = \frac{3}{2}$ and $L = \mu_7' (\Norm{Y_1 - q}_{L^{1+u}} \vee c)$. 

By equation \eqref{eq_ub_kappalogkappa}, hypothesis \textbf{(Ni)} holds with $\nu_0 = 0.6 \sqrt{\frac{\log n_t}{n_v}}$. Let $\mu_5 = 0.6 \mu_2 \sqrt{4.5} \geq 0.6 \mu_2 \sqrt{\alpha + 3}$, so that $\theta \geq \frac{\mu_5}{\eta} \sqrt{\frac{\log n_t}{n_v}}$ implies $\theta \geq \sqrt{\alpha + 3} \frac{\mu_2 \nu_0}{\eta}$. Then, by Theorem \ref{agcv_hub} and since $K \log K \leq n_t$ (by equation \eqref{eq_bound_K}), we obtain Corollary \ref{cor_gauss} with $\mu_8 = 7 \mu_1 \mu_7'$.
\end{proof}

\begin{lemma} \label{lem_lb_gam}
There exists a constant $\mu_6$ such that
for any subset $I$ such that $|I| \leq \min \left( \frac{n_t}{\log n_t}, \frac{2}{5}(n_t-1) \right)$ and for all $\varepsilon \in (0;1]$,
\[ \mathbb{P} \left( \hat{\gamma}_I \leq \varepsilon \right) \leq 2 \sqrt{e} (\mu_6 \varepsilon)^{\frac{n_t - 1}{2}}. \]
Moreover, if in addition $|I| \leq \frac{n_t}{\log d}$,
then for all $\varepsilon \in (0;1]$,
\[ \mathbb{P} \left( \hat{\gamma} \leq \varepsilon \right) 
\leq 2e^{\frac{5}{2}} (\mu_6 e^2 \varepsilon)^{\frac{n_t - 1}{2}}.  \]
As a result, for any $r \in \left[0, \frac{n_t - 1}{3} \right]$,
\[ \Norm{\frac{1}{\hat{\gamma}}}_{L^r} \leq 2 \mu_6 e^2 \left[ 2 (1+2 e^{\frac{5}{2}}) \right]^{\frac{1}{r}}. \]
\end{lemma}
\begin{proof}
By restricting to a subspace, we can always assume that 
$M(\theta) = \sqrt{\theta \Sigma \theta}$ is a norm.
Let  $S_\Sigma = \{ \theta \in E_I: \sqrt{\theta \Sigma \theta} = 1 \}$
be the unit sphere in norm $M$.
Let $\varepsilon > 0$.
By changing coordinates, it is easy to see that the metric 
entropy of $S_\Sigma$ in norm $M$ is the same as that
of the euclidean sphere $S$ in the euclidean norm.
Therefore, for any $\delta > 0$, there exists a finite set $S_{\Sigma,\delta}$, of
cardinality less than $\left(\frac{6}{\delta} \right)^d$ and such that for any $u \in S_\Sigma$, there exists $v \in S_{\Sigma,\delta}$ such that $M(u - v) \leq \frac{\delta}{2}$.
Therefore,
\begin{align*}
\mathbb{P} \left( \hat{\gamma}_I \leq \frac{\varepsilon}{2} \right)
&= \mathbb{P} \left( \inf_{\theta \in S_\Sigma} \hat{N}(\theta) \leq \frac{\varepsilon}{2}\right) \\
&\leq \mathbb{P} \left( \inf_{\theta \in S_{\Sigma,\delta}} \hat{N}(\theta) \leq \varepsilon \right) + \mathbb{P} \left( \sup_{\theta \in E_I: M(\theta) \leq \delta} \hat{N}(\theta) \geq \frac{\varepsilon}{2}  \right). \numberthis \label{eq_discr_inf_gauss}
\end{align*}
By definition,
\begin{align*}
\sup_{\theta \in E_I: M(\theta) \leq \delta} \hat{N}(\theta) 
&=  \sup_{\theta \in E_I: \sqrt{\theta^T \Sigma \theta } \leq \delta} 
\max_{1 \leq i \leq n_t} |\langle \theta, X_i - \bar{X}_{n_t} \rangle| \\
&= \delta \max_{1 \leq i \leq n_t} \sqrt{(X_{i,I} - \bar{X}_{n_t,I})^T \Sigma_{I,I}^{-1}(X_{i,I} - \bar{X}_{n_t,I})} \\
&\leq 2\delta \max_{1 \leq i \leq n_t} \Norm{\Sigma_{I,I}^{-\frac{1}{2}}(X_{i,I} - P X_{i,I})}_2.
\end{align*}
As $\Sigma_{I,I}^{-\frac{1}{2}}(X_{i,I} - P X_{i,I})$ is a standard 
normal vector,
$P \left[ \Norm{\Sigma_{I,I}^{-\frac{1}{2}}(X_{i,I} - PX_{i,I})}_2 \right]
\leq \sqrt{|I|}$.
 Hence, by the union bound and the Gaussian concentration 
inequality,
\begin{align*}
\mathbb{P} \left( \sup_{\theta \in E_I: M(\theta) \leq \delta} \hat{N}(\theta) \geq \frac{\varepsilon}{2}  \right) &\leq 
n_t \mathbb{P} \left( \Norm{\Sigma_{I,I}^{-\frac{1}{2}}(X_{i,I} - P X_{i,I})}_2 \geq \frac{\varepsilon}{4\delta} \right) \\
&\leq n_t \exp \left( - \frac{1}{2} \bigl( \tfrac{\varepsilon}{4\delta} 
- \sqrt{|I|} \bigr)^2 \right). \numberthis \label{eq_ub_sup_gauss}
\end{align*}
On the other hand, for any $\theta \in S_\Sigma$, 
$\langle \theta, X_i - P X_i \rangle$ is standard normal, therefore
\begin{align*}
\mathbb{P} \left(\hat{N}(\theta) \leq \varepsilon \right) &=
\mathbb{P} \left( \max_{1 \leq i \leq n_t} |\langle \theta, X_i - \bar{X}_{n_t}| \leq \varepsilon   \right) \\ 
&\leq 
\mathbb{P} \left( \inf_{m \in \mathbb{R}} \max_{1 \leq i \leq n_t} |\langle \theta, X_i - P X_i \rangle - m| \leq \varepsilon   \right) \\
&\leq \mathbb{P} \left( \max_{1 \leq i \leq n_t} |\langle \theta, X_i - P X_i \rangle - \langle \theta, X_1 - P X_1 \rangle| \leq 2 \varepsilon   \right) \\
&\leq \left( \frac{2\varepsilon}{\sqrt{2 \pi}} \wedge 1 \right)^{n_t - 1}.
\numberthis \label{eq_lb_inf_gauss}
\end{align*} 
By the union bound, it follows from equations \eqref{eq_discr_inf_gauss}, \eqref{eq_ub_sup_gauss} and \eqref{eq_lb_inf_gauss} that
\[ \mathbb{P} \left( \hat{\gamma}_I \leq \frac{\varepsilon}{2} \right)
\leq \left( \frac{6}{\delta} \right)^{|I|} \left( \frac{2\varepsilon}{\sqrt{2 \pi}} \wedge 1 \right)^{n_t - 1} + n_t \exp \left( - \frac{1}{2} \bigl( \tfrac{\varepsilon}{4\delta} 
- \sqrt{|I|} \bigr)^2 \right).   \]
Let now $\delta = \frac{\varepsilon}{4 \left(\sqrt{|I|} + \sqrt{2(\log n_t + n_t \log \frac{1}{\varepsilon})} \right)}$.
Then 
\[ n_t \exp \left( - \frac{1}{2} \bigl( \tfrac{\varepsilon}{4\delta} 
- \sqrt{|I|} \bigr)^2 \right) \leq \varepsilon^{n_t}. \]
Moreover, there exists a constant $\mu$ such that
\[ \left( \frac{6}{\delta} \right)^{|I|} \left( \frac{2\varepsilon}{\sqrt{2 \pi}} \wedge 1 \right)^{n_t - 1} \leq 
\mu^{|I|} \max \left(\sqrt{|I|}^{|I|}, \sqrt{\log n_t}^{|I|}, \sqrt{n_t \log \tfrac{1}{\varepsilon}}^{|I|} \right) \varepsilon^{n_t - 1 - |I|}. \]
Because $|I| \leq \frac{n_t}{\log n_t}$, $\sqrt{|I|}^{|I|} = \exp \left( \frac{1}{2} |I| \log(|I|) \right) \leq e^{\frac{n_t}{2}}$.
Using the inequality $\log n_t \leq \sqrt{n_t}$, it follows by the same argument that $\sqrt{\log n_t}^{|I|} \leq e^{\frac{n_t}{4}}$. 
Since $\log \tfrac{1}{\varepsilon} \leq \frac{1}{\sqrt{\varepsilon}}$,
 \[ \sqrt{n_t \log \tfrac{1}{\varepsilon}}^{|I|} \leq 
 \exp \bigl( \frac{1}{2} |I| \log n_t + \frac{1}{4} |I| \log \tfrac{1}{\varepsilon} \bigr) \leq e^{\frac{n_t}{2}} \varepsilon^{\frac{- |I|}{4}}. \]
It follows that 
\[ \left( \frac{6}{\delta} \right)^{|I|} \left( \frac{2\varepsilon}{\sqrt{2 \pi}} \wedge 1 \right)^{n_t - 1} \leq e^{\frac{n_t}{2}} \mu^{|I|} 
 \varepsilon^{n_t - 1 - \frac{5}{4}|I|}.  \]
 Finally, since $|I| \leq \frac{2}{5} (n_t - 1)$,
 \[ \left( \frac{6}{\delta} \right)^{|I|} \left( \frac{2\varepsilon}{\sqrt{2 \pi}} \wedge 1 \right)^{n_t - 1} \leq 
 e^{\frac{1}{2}} \bigl( e \mu^{\frac{4}{5}} \varepsilon \bigr)^{\frac{n_t - 1}{2}}, \]
 which yields the first inequality for some constant $\mu_6$.
 The second inequality then follows from the union bound:
 \begin{align*}
 \mathbb{P}(\hat{\gamma} \leq \varepsilon) &\leq \sum_{I \subset \{1,\ldots,d\} : |I| \leq K} \mathbb{P}(\hat{\gamma}_I \leq \varepsilon) \\
 &\leq 2\sqrt{e} \bigl( \mu_6 \varepsilon \bigr)^{\frac{n_t - 1}{2}} \times \sum_{k = 1}^K {K \choose k} \\
 &\leq 2\sqrt{e} \bigl( \mu_6 \varepsilon \bigr)^{\frac{n_t - 1}{2}}
 \times \sum_{k = 1}^K \frac{d^ k}{k!} \\
 &\leq 2e d^ K \sqrt{e} \bigl( \mu_6 \varepsilon \bigr)^{\frac{n_t - 1}{2}}.
 \end{align*}
 By assumption, $d ^K = e^{K \log d} \leq e^{n_t}$, 
 which yields the second equation.
As a result, for any $r \leq \frac{n_t-1}{3}$
\begin{align*}
\mathbb{E} \left[ \frac{1}{\hat{\gamma}^r} \right]
&= \int_0^{+ \infty} \mathbb{P} \left(\frac{1}{\hat{\gamma}^r} \geq t  \right) dt \\
&\leq (\mu_6 e^2)^r + \int_{(\mu_6 e^2)^r}^{+ \infty} \mathbb{P} \left(\frac{1}{\hat{\gamma}^r} \geq t  \right) dt \\
&= (\mu_6 e^2)^r + (\mu_6 e^2)^r \int_{1}^{+ \infty} \mathbb{P} \left(\hat{\gamma} \leq \frac{1}{(\mu_6 e^2)t^{\frac{1}{r}}}  \right) dt \\
&\leq ( \mu_6 e^2)^r + (\mu_6 e^2)^r \times 2 e^{\frac{5}{2}} 
 \int_{1}^{+ \infty} \left(\frac{1}{t^{\frac{1}{r}}} \right)^{\frac{n_t - 1}{2}} \\
 &= (1+2 e^{\frac{5}{2}}) (\mu_6 e^2)^r \frac{1}{\frac{n_t - 1}{2r} - 1} \\
 &\leq 2(1+2 e^{\frac{5}{2}}) (\mu_6 e^2)^r. 
\end{align*}
\end{proof}

\subsection{Fourier series} \label{app.sec.trig}
\subsubsection{Proof of Corollary \ref{cor_trigo}}
Let $I \subset \{1,\ldots, d\}$ and $\theta \in \mathbb{R}^I$.
Since $\psi_j(\mathbb{R}) \subset [-\sqrt{2}; \sqrt{2}]$, 
for any $x \in \mathbb{R}^d$, by the Cauchy Schwarz inequality,
\begin{align*}
|\langle \theta, x - P X \rangle| &= \left| \sum_{j \in I} \theta_j (\psi_j(x) - E [\psi_j(U)] ) \right| \\
&\leq \sqrt{\sum_{j \in I} \theta_j^2} \sqrt{8 |I|}. 
\end{align*}
Therefore, 
\[ \Norm{\langle \theta, X - P X \rangle}_{L^{\psi_1}} 
\leq \frac{1}{\log 2} \Norm{\langle \theta, X - P X \rangle}_{L^\infty} 
\leq \frac{\sqrt{8 |I|}}{\log 2} \Norm{\theta}_{\ell^2}.  \]
On the other hand, for all $j$, $\psi_j(U) = \psi_j(U - \lfloor U \rfloor)$,
where the variable $U - \lfloor U \rfloor$ has density $\sum_{j \in \mathbb{Z}} p_U(\cdot + j)$ on $[0;1]$, which by assumption is greater than $p_0$. Therefore, by orthonormality of the trigonometric basis,
\begin{align*}
P \left( \langle \theta, X - P X \rangle^2 \right) 
&\geq p_0 \int_0^1 \left( \sum_{j \in I} \theta_j (\psi_j(u) - P \psi_j) \right)^2 du \\
&\geq  p_0 \Norm{\theta}_{\ell^2}^2.
\end{align*}
Thus, for any $I \subset \{1,\ldots, d\}$ and $\theta \in \mathbb{R}^I$,
\[ \Norm{\langle \theta, X - P X \rangle}_{L^{\psi_1}}
\leq \frac{1}{\log 2} \sqrt{\frac{8 |I|}{p_0}} \Norm{\langle \theta, X - P X \rangle}_{L^2},  \]
which proves that $\kappa(K) \leq \sqrt{8\frac{K}{p_0 \log^2 2}}$. Take $\mu_9 = \frac{\mu_2 \sqrt{40}}{2 \log 2} \geq 4$ in equation \eqref{hyp_ub_K_cor_trigo}, so that, since $n_t \geq 3$ and $n_v \leq n_t$,
\[ \kappa(K) \leq \frac{2 \theta \eta}{\mu_2 \sqrt{5}} \frac{\sqrt{n_v}}{\log^{\frac{3}{2}} n_t} < \sqrt{n_v} \leq \sqrt{n_t}. \]

Then equation \eqref{Hw_huber} of Theorem \ref{agcv_hub} holds with $\nu_0 = \frac{\theta \eta}{\mu_2 \sqrt{5}}$. 

Moreover, since the support $I_k$ of $\tilde{\theta}_k$ has cardinality $|I_k| \leq K$, by the Cauchy-Schwarz inequality,
\begin{align*}
\max_{1 \leq k \leq K} \left| \langle \tilde{\theta}_k, X - P X \rangle \right| &\leq \sqrt{8K}  \max_{1 \leq k \leq K} \Norm{\tilde{\theta}_k}_{\ell^2} \\
&\leq \sqrt{8K} n_t^{\frac{3}{2}} \\
 \max_{1 \leq k \leq K} \max_{1 \leq i \leq n_t} \left| \langle \tilde{\theta}_k, X_i - P X_i \rangle \right| &\leq \sqrt{8K}  \max_{1 \leq k \leq K}  \Norm{\tilde{\theta}_k}_{\ell^2} \\
 &\leq \sqrt{8K} n_t^{\frac{3}{2}}.
\end{align*}
Since by assumption (equation \eqref{hyp_ub_K_cor_trigo}), $K \leq \frac{n_v}{\mu_9^2 \log n_t} \leq \frac{n_t}{16}$, hypothesis \textbf{(Uub)} holds with $L = \frac{1}{\sqrt{2}}$ and $\alpha = 2$.
As a result, applying Theorem \ref{agcv_hub}
yields equation \eqref{or_ineq_cor_trigo}.

\subsubsection{Proof of proposition \ref{prop_trunc_improves}}

Let $\tilde{t}_k : x \mapsto \tilde{q}_k + \langle \tilde{\theta}_k, x \rangle$ and $\hat{t}_k : x \mapsto \hat{q}_k + \langle \hat{\theta}_k, x \rangle$.
By lemma \ref{lem_lb_loss}, 
\begin{align*}
\loss{\hat{t}_k} &\geq P \left[ \frac{\eta c}{4} |\hat{t}_k(X) - s(X)| \mathbb{I}_{|\hat{t}_k(X) - s(X)| \geq \frac{c}{2}} \right] \\
&\geq \frac{\eta c}{4}  \Norm{\hat{t}_k(X) - s(X)}_{L^1} - \frac{\eta c^2}{8} \\
&\geq \frac{\eta c}{4} \Norm{\hat{t}_k(X) - \tilde{q}}_{L^1} - 
\frac{\eta c}{4} \Norm{s(X) - \tilde{q}}_{L^1} - \frac{\eta c^2}{8}. \numberthis \label{eq_lb_loss_tk}
\end{align*} 
Let $I$ be the support of $\hat{\theta}_k$, and
$\hat{\theta}_{k,j}$ denote the $j^{th}$ component of the vector
$\hat{\theta}_k$. By 
the Cauchy-Schwarz inequality and 
orthogonality of the trigonometric basis,
\begin{align*}
\NormInfinity{\hat{t}_k - \tilde{q}} &= \sup_{x \in \mathbb{R}} 
\left| \hat{q}_k - \tilde{q} + \sum_{j \in I} \hat{\theta}_{k,j} \psi_j(x) \right| \\
&\leq \sqrt{(\hat{q}_k - \tilde{q})^2 + \Norm{\hat{\theta}_k}_{\ell^2}^2} \sqrt{2|I| + 1} \\
&\leq \sqrt{2K + 1} \Norm{\hat{t}_k - \tilde{q}}_{L^2}.
\end{align*}
Since $\Norm{\hat{t}_k - \tilde{q}}^2_{L^2} \leq \NormInfinity{\hat{t}_k - \tilde{q}} \Norm{\hat{t}_k(X) - \tilde{q}}_{L^1} $, it follows that
\[ \Norm{\hat{t}_k(X) - \tilde{q}}_{L^1} \geq \frac{\Norm{\hat{t}_k - \tilde{q}}_{L^2}}{\sqrt{2K + 1}} \geq \frac{\Norm{\hat{\theta}_k}_{\ell^2}}{\sqrt{2K + 1}}, \]
therefore by equation \eqref{hyp_ub_K_cor_trigo}, on the event $\Norm{\hat{\theta}_k}_{\ell^2} \geq n_t^{\frac{3}{2}}$,
\begin{align*}
\Norm{\hat{t}_k(X) - \tilde{q}}_{L^1} &\geq \frac{n_t^\frac{3}{2}}{\sqrt{\eta^2 \frac{n_t}{8} + 1}} \\
&= \frac{n_t}{\sqrt{\frac{\eta^2}{8} + \frac{1}{n_t}}} \\
&\geq  \frac{3 n_t}{2 \eta} \text{ since } n_t \geq \frac{4}{\eta^2} \text{ by equation } \eqref{hyp_lb_nt_cor_trigo}.    
\end{align*}

On the event $\Norm{\hat{\theta}_k}_{\ell^2} \geq n_t^{\frac{3}{2}}$, $\loss{\tilde{t}_k} = \loss{\tilde{q}} \leq c \Norm{s(X) - \tilde{q}}_{L^1}$,
therefore by equation \eqref{eq_lb_loss_tk},
\begin{align*}
\loss{\hat{t}_k} - \loss{\tilde{t}_k} &\geq \frac{3c n_t}{8} - \frac{5 c}{4} \Norm{s(X) - \tilde{q}}_{L^1} - \frac{\eta c^2}{8} \\
&\geq \frac{3c n_t}{8} - \frac{5 c}{4} \Norm{s(X) - q_*}_{L^1} - \frac{5 c}{4} |\tilde{q} - q_*|  - \frac{\eta c^2}{8} \\
&\geq \frac{c n_t}{4} - \frac{5c}{4}|\tilde{q} - q_*| \text{ by assumption } \eqref{hyp_lb_nt_cor_trigo}.
\end{align*}
Let $\hat{k} \in \argmin_{1 \leq k \leq K} \loss{\hat{t}_k}$.
Thus, on the event that $\Norm{\hat{\theta}_{\hat{k}}}_{\ell^2} > n_t^{\frac{3}{2}}$,
\[ \min_{1 \leq k \leq K} \loss{\hat{t}_k} - \min_{1 \leq k \leq K} \loss{\tilde{t}_k} \geq \loss{\hat{t}_{\hat{k}}} - \loss{\tilde{t}_{\hat{k}}} \geq 
\frac{c n_t}{4} - \frac{5c}{4}|\tilde{q} - q_*|. \]
On the other hand, if $\Norm{\hat{\theta}_{\hat{k}}}_{\ell^2} \leq n_t^{\frac{3}{2}}$, $\tilde{t}_{\hat{k}} = \hat{t}_{\hat{k}}$ by definition, so
\[ \min_{1 \leq k \leq K} \loss{\hat{t}_k} - \min_{1 \leq k \leq K} \loss{\tilde{t}_k} \geq \loss{\hat{t}_{\hat{k}}} - \loss{\tilde{t}_{\hat{k}}} \geq 0. \]
Let $\delta_0 = \mathbb{P} \left( \Norm{\hat{\theta}_{\hat{k}}}_{\ell^2} > n_t^{\frac{3}{2}} \right)$.
By Hölder's inequality,
\begin{align*}
\mathbb{E} \left[ \min_{1 \leq k \leq K} \loss{\hat{t}_k} \right] 
- \mathbb{E} \left[ \min_{1 \leq k \leq K} \loss{\tilde{t}_k} \right] 
&\geq \delta_0 \frac{c n_t}{4} - \frac{5c}{4} \delta_0^{\frac{3}{4}} \mathbb{E}[(\tilde{q} - q_*)^4]^{\frac{1}{4}} \\
&\geq \inf_{\delta \in (0,1]} \delta \frac{c n_t}{4} - \frac{5c}{4} \delta^{\frac{3}{4}} \mathbb{E}[(\tilde{q} - q_*)^4]^{\frac{1}{4}}.
\end{align*}
Hence, by lemma \ref{lem_inf_prob} with $\alpha = \frac{3}{4}$ , there exists a constant $\mu$ such that
\begin{equation} \label{eq_lb_diff_min}
\mathbb{E} \left[ \min_{1 \leq k \leq K} \loss{\hat{t}_k} \right] 
- \mathbb{E} \left[ \min_{1 \leq k \leq K} \loss{\tilde{t}_k} \right]  
\geq - \mu \frac{c \mathbb{E}[(\tilde{q} - q_*)^4]}{n_t^3}.
\end{equation} 
Moreover, by lemma \ref{lem_ub_risque_hub_loc} below, for all $n_t \geq \frac{16}{\alpha}$,
\[ \mathbb{E}[|\tilde{q} - q_*|^4]^{\frac{1}{4}} \leq c + 1.4 \times 2^{\frac{2}{\alpha}} P(|Y - q_*|^\alpha)^{\frac{1}{\alpha}}. \]
Thus, equation \eqref{eq_ub_tilde_cor_trigo} follows from equation \eqref{eq_lb_diff_min} and the additional assumption that $n_t \geq \frac{16}{\alpha}$.

\begin{lemma} \label{lem_inf_prob}
Let $a,b$ be positive real numbers and 
let $\alpha \in [0,1)$.
Then 
\[ \inf_{\delta > 0} a \delta - b \delta^\alpha \geq \left[ \alpha^{\frac{1}{1 - \alpha}} - \alpha^{\frac{\alpha}{1 - \alpha}} \right] \frac{b^{\frac{1}{1 - \alpha}}}{a^{\frac{\alpha}{1 - \alpha}}}. \]
\end{lemma}

\begin{proof}
The function $f: \delta \to a \delta - b \delta^\alpha$ 
is continuous, tends to $+ \infty$ at $+\infty$ and 
$f(0) = 0$, so $f$ reaches a global minimum $\delta_*$ on $[0,+\infty)$.
As $f$ is differentiable,
\begin{align*}
0 = f'(\delta_*) &\iff a - \frac{\alpha b}{\delta_*^{1 - \alpha}} = 0  \\
&\iff a \delta_*^{1 - \alpha} = \alpha b \\
&\iff \delta_* = \left( \frac{\alpha b}{a} \right)^{\frac{1}{1 - \alpha}}.
\end{align*}
Thus, for all $\delta \in [0,+\infty)$,
\begin{align*}
a \delta - b \delta^\alpha &\geq a  \left( \frac{\alpha b}{a} \right)^{\frac{1}{1 - \alpha}} - b \left( \frac{\alpha b}{a} \right)^{\frac{\alpha}{1 - \alpha}} \\
&\geq \left[ \alpha^{\frac{1}{1 - \alpha}} - \alpha^{\frac{\alpha}{1 - \alpha}} \right] \frac{b^{\frac{1}{1 - \alpha}}}{a^{\frac{\alpha}{1 - \alpha}}}.
\end{align*}
\end{proof}

\begin{lemma} \label{lem_ub_risque_hub_loc}
Let $n_t \geq 4$ be an integer and $Y_1, \ldots , Y_{n_t}$ be iid random variables 
such that, for some $q_* \in \mathbb{R}$ and $\alpha \in \left[\tfrac{4}{n_t},1 \right]$,
$E[|Y_1 - q_*|^\alpha] < +\infty$.
Let
\[ \tilde{q} \in \argmin_{q \in \mathbb{R}} \sum_{i = 1}^{n_t} \phi_c(Y_i - q). \]
Then for all $r \in \left[ 1,\frac{\alpha n_t}{4} \right]$,
\[ \mathbb{E}[|\tilde{q} - q_*|^r]^{\frac{1}{r}} \leq c + 2^{\frac{2}{\alpha}} 3^{\frac{1}{r}} E[|Y_1 - q_*|^\alpha]^{\frac{1}{\alpha}}. \]
\end{lemma}

\begin{proof}
Remark first that for any $x \in \mathbb{R}$, 
\[ \phi_c'(x) = 
\begin{cases}
-c \text{ if } x \leq - c \\
 x \text{ if } |x| \leq c \\
 c  \text{ if } x \geq c. 
\end{cases} \] 
For any $q \in \mathbb{R}$, let $I_+(q) = \{i : Y_i > q + c \}$,
$I_-(q) = \{i : Y_i < q - c \}$ and 
$I_0(q) = \{i : |Y_i - q| \leq c \}$.
Thus, 
\[ \sum_{i = 1}^{n_t} \phi_c'(Y_i - q) 
= c |I_+(q)| - c |I_-(q)| + \sum_{i \in I_0(q)} Y_i - q, \]
so that 
\[ c\left( |I_+(q)| - |I_-(q)| - |I_0(q)| \right) \leq 
\sum_{i = 1}^{n_t} \phi_c'(Y_i - q)
\leq c \left( |I_+(q)| + |I_0(q)| - |I_-(q)| \right). \]
Let $q_g$ be such that $|I_+(q_g)| > \frac{n_t}{2}$ and 
let $q_d$ be such that $|I_-(q_d)| > \frac{n_t}{2}$.
By monotony of $\phi_c'$, for all $q \leq q_g$,
$\sum_{i = 1}^{n_t} \phi_c'(Y_i - q) > 0$ and for all $q \geq q_d$,
$\sum_{i = 1}^{n_t} \phi_c'(Y_i - q) < 0$.
Since by definition of $\tilde{q}$,
\[ \frac{1}{n_t} \sum_{i = 1}^{n_t} \phi_c'(Y_i - \tilde{q}) = 0, \] 
it follows that $\tilde{q} \in [q_g,q_d]$.

Let $\sigma = E[|Y - q_*|^\alpha]^{\frac{1}{\alpha}}$.
By the union bound and Markov's inequality, for all $u > 0$,
\begin{align*}
\mathbb{P} \left( |I_+(q_* - c - u\sigma)| \leq \frac{n_t}{2} \right)
&= \mathbb{P} \left( |\{i : Y_i \geq q_* + \sigma u \}| > \frac{n_t}{2} \right) \\
&\leq {n_t \choose \lceil \tfrac{n_t}{2} \rceil } \mathbb{P} (Y_i \geq q_* + \sigma u )^{\frac{n_t}{2}} \\
&\leq \frac{2^{n_t}}{u^{\frac{\alpha n_t}{2}}}.
\end{align*}
Symetrically, \[ \mathbb{P} \left( |I_-(q_* + c + u\sigma)| \leq \frac{n_t}{2} \right) \leq \frac{2^{n_t}}{u^{\frac{\alpha n_t}{2}}}, \]
so that one can take $q_g = q_* - c - u\sigma $ and 
$q_d = q_* + c + u\sigma$ with probability greater than
$1 - \frac{2^{n_t + 1}}{u^{\frac{\alpha n_t}{2}}}$.
It follows that, for any $u > 0$,
\begin{equation} \label{eq_ub_dev_hub_loc_est}
\mathbb{P}(|\tilde{q} - q_*| > c + u\sigma) 
\leq \frac{2^{n_t + 1}}{u^{\frac{\alpha n_t}{2}}}.
\end{equation}

For any $r \geq 1$,  $\mathbb{E} \left[|\tilde{q} - q_*|^r \right]^{\frac{1}{r}} \leq c + 
\mathbb{E} \left[ (|\tilde{q} - q_*| - c)_+^r \right]^{\frac{1}{r}}$,
where 
\begin{align*}
\mathbb{E} \left[ (|\tilde{q} - q_*| - c)_+^r \right] &\leq \sigma^r \int_{0}^{+\infty} 
\mathbb{P} \left( \frac{(|\tilde{q} - q_*| - c)_+^r}{\sigma^r} \geq u \right) du \\
&\leq  \sigma^r \int_{0}^{+\infty} 
\mathbb{P} \left( (|\tilde{q} - q_*| \geq c + \sigma u^{\frac{1}{r}} \right) du \\
&\leq \sigma^r \int_{0}^{+\infty}  \min \left(1, \frac{2^{n_t + 1}}{u^{\frac{\alpha n_t}{2r}}} \right) \text{ by equation } \ref{eq_ub_dev_hub_loc_est} \\
&\leq 2^{\frac{2r}{\alpha}} \sigma^r + 2\sigma^r \int_{2^{\frac{2r}{\alpha}}}^{+ \infty} \left( \frac{2^{\frac{2r}{\alpha}}}{v} \right)^\frac{\alpha n_t}{2r} dv \\
&\leq 2^{\frac{2r}{\alpha}} \sigma^r + 2 \cdot 2^{\frac{2r}{\alpha}} \sigma^r 
\int_1^{+ \infty} \frac{dx}{x^\frac{\alpha n_t}{2r}} \\
&\leq 2^{\frac{2r}{\alpha}} \left( 1 + \frac{2}{\frac{\alpha n_t}{2r} - 1} \right) \sigma^r. 
\end{align*}
This yields the result under the condition that $r \leq \frac{\alpha n_t}{4}$.
\end{proof}

\subsection{Proof of proposition \ref{prop_ag_quant_improve}} \label{proof_prop_ag_quant_improve}
For any $i \in \{1,\ldots,V\}$, denote $\ERMho{T_i}$ by $\hat{f}_i(X)$ for simplicity. 
For any $u \in \mathbb{R}$, let
\[ \xloss{u} = P \left[\phi_c(Y - u) - \phi_c(Y - s(X)) | X \right].   \]
Let also  
\[ \hat{I} = \left\{ i \in \{1,\ldots,V\} : |(\ERMho{T_i} - s)(X)| \leq \frac{c}{2} \right\}. \]
By Jensen's inequality,
\begin{align*}
    \ell_X \Bigl(\frac{1}{V} \sum_{i = 1}^V \hat{f}_i \Bigr) &\leq \frac{|\hat{I}|}{V} \ell_X \Bigl(\frac{1}{|\hat{I}|} \sum_{i \in \hat{I}} \hat{f}_i \Bigr) + \frac{V - |\hat{I}|}{V} \ell_X \Bigl( \frac{1}{V - |\hat{I}|} \sum_{i \notin \hat{I}} \hat{f}_i \Bigr) \\
    &\leq \frac{|\hat{I}|}{V} \ell_X \Bigl(\frac{1}{|\hat{I}|} \sum_{i \in \hat{I}} \hat{f}_i \Bigr) + \frac{1}{V} \sum_{i \notin \hat{I}} \ell_X (\hat{f}_i).
    \numberthis \label{eq_ub_jensen_loss_ag}
\end{align*}
Let now $\bar{f}_{\hat{I}} = \frac{1}{|\hat{I}|} \sum_{i \in \hat{I}} \ERMho{T_i}$.
By claim \ref{lem_lb_loss},
for any $i \in \hat{I}$,
\[ \xloss{\hat{f}_i} \geq \xloss{\bar{f}_{\hat{I}}} + \ell_X'(\bar{f}_{\hat{I}}) (\hat{f}_i - \bar{f}_{\hat{I}}) + \frac{\eta}{2} \left( \hat{f}_i -  \bar{f}_{\hat{I}} \right)^2.  \]
Averaging over $i \in \hat{I}$ yields
\[ \frac{1}{|\hat{I}|} \sum_{i \in \hat{I}} \xloss{\hat{f}_i} \geq \ell_X \Bigl(\frac{1}{|\hat{I}|} \sum_{i \in \hat{I}} \ERMho{T_i} \Bigr) + \frac{\eta}{4 |\hat{I}|^2} \sum_{i \in \hat{I} } \sum_{j \in \hat{I}} (\hat{f}_i - \hat{f}_j)^2.  \]
Combining this bound with equation \eqref{eq_ub_jensen_loss_ag} yields
\begin{align*}
  \ell_X \Bigl(\frac{1}{V} \sum_{i = 1}^V \hat{f}_i \Bigr) &\leq   
  \frac{1}{|\hat{I}|} \sum_{i \in \hat{I}} \xloss{\hat{f}_i} 
  - \frac{\eta}{4 V |\hat{I}|} \sum_{i \in \hat{I}} \sum_{j \in \hat{I}} (\hat{f}_i - \hat{f}_j)^2 + \frac{1}{V} \sum_{i \notin \hat{I}} \xloss{\hat{f}_i} \\
  &\leq \frac{1}{V} \sum_{i = 1}^V \xloss{\hat{f}_i} - \frac{\eta}{4 V^2} \sum_{i \in \hat{I}} \sum_{j \in \hat{I}} (\hat{f}_i - \hat{f}_j)^2.
\end{align*}
Taking expectations yields equation \eqref{eq_comp_agghoo_ho} by exchangeability of the $\hat{f}_i$. 
Assume now that $\mathbb{E}[\loss{\hat{f}_1}] \leq \frac{\eta c^2}{64}$. 
By claim \ref{lem_lb_loss}, 
\[ \mathbb{E}[\loss{\hat{f}_1}] \geq \frac{\eta c^2}{8} \mathbb{P} \left( \bigl|(\hat{f}_1 - s)(X) \bigr| \geq \frac{c}{2} \right) = \frac{\eta c^2}{8} \mathbb{P}(E_1(c)).  \]
It follows that $\mathbb{P}(E_1(c)) \leq \frac{1}{8}$.
Since the $\hat{f}_i$ have the same distribution, $\mathbb{P}(E_2(c)) \leq \frac{1}{8}$ also.
Thus, by definition of the median,
\[ \mathbb{P} \left( E_1(c) \cap E_2(c) \cap \bigl\{ (\hat{f}_1 - \hat{f}_2)^2(X) \geq \text{Med} \bigl[ (\hat{f}_1 - \hat{f}_2)^2(X) \bigr] \bigr\} \right) \geq \frac{1}{4}. \]
Equation \eqref{eq_comp_agghoo_ho_med} then follows from equation \eqref{eq_comp_agghoo_ho}.

%
%
%
\bibliographystyle{plain}
\bibliography{premier_article_biblio}

\end{document}